\documentclass[10pt,reqno,a4paper]{amsart}

\usepackage{amsthm}
\usepackage{mathtools}
\usepackage{amssymb}
\usepackage{hyperref}
\usepackage{enumitem}
\usepackage{graphicx}
\usepackage{xcolor}
\usepackage{kotex}
\usepackage{comment}

\numberwithin{equation}{section}
\numberwithin{figure}{section}
\allowdisplaybreaks

\theoremstyle{plain}
\newtheorem{thm}{\protect\theoremname}[section]
\newtheorem{lem}[thm]{\protect\lemmaname}
\newtheorem{prop}[thm]{\protect\propositionname}

\theoremstyle{remark}
\newtheorem{rem}[thm]{\protect\remarkname}

\providecommand{\theoremname}{Theorem}
\providecommand{\corollaryname}{Corollary}
\providecommand{\lemmaname}{Lemma}
\providecommand{\remarkname}{Remark}
\providecommand{\propositionname}{Proposition}

\begin{document}
\global\long\def\Im{\mathrm{Im}}%
\global\long\def\Re{\mathrm{Re}}%
\global\long\def\s{\sigma}%
\global\long\def\d{\partial}%
\global\long\def\mc#1{\mathcal{#1}}%
\global\long\def\Right{\Rightarrow}%
\global\long\def\Left{\Leftarrow}%
\global\long\def\les{\lesssim}%
\global\long\def\hook{\hookrightarrow}%
\global\long\def\D{\mathbf{D}}%
\global\long\def\rad{\mathrm{rad}}%
\global\long\def\jp#1{\langle#1\rangle}%
\global\long\def\norm#1{\|#1\|}%
\global\long\def\ol#1{\overline{#1}}%
\global\long\def\wt#1{\widehat{#1}}%
\global\long\def\br#1{(#1)}%
\global\long\def\Bb#1{\Big(#1\Big)}%
\global\long\def\bb#1{\big(#1\big)}%
\global\long\def\lr#1{\left(#1\right)}%
\global\long\def\la{\lambda}%
\global\long\def\al{\alpha}%
\global\long\def\be{\beta}%
\global\long\def\ga{\gamma}%
\global\long\def\La{\Lambda}%
\global\long\def\De{\Delta}%
\global\long\def\na{\nabla}%
\global\long\def\fl{\flat}%
\global\long\def\sh{\sharp}%
\global\long\def\calN{\mathcal{N}}%
\global\long\def\avg{\mathrm{avg}}%
\global\long\def\bbR{\mathbf{\mathbb{R}}}%
\global\long\def\bbC{\mathbf{\mathbb{C}}}%
\global\long\def\bbZ{\mathbf{\mathbb{Z}}}%
\global\long\def\bbN{\mathbf{\mathbb{N}}}%
\global\long\def\bbT{\mathbb{T}}%
\global\long\def\bbM{\mathbb{M}}%
\global\long\def\bbP{\mathbb{P}}%
\global\long\def\bbS{\mathbb{S}}%
\global\long\def\bfD{\mathbf{D}}%
\global\long\def\bfL{\mathbf{L}}%
\global\long\def\calF{\mathcal{F}}%
\global\long\def\calH{\mathcal{H}}%
\global\long\def\calL{\mathcal{L}}%
\global\long\def\calE{\mathcal{E}}%
\global\long\def\calO{\mathcal{O}}%
\global\long\def\calR{\mathcal{R}}%
\global\long\def\calD{\mathcal{D}}%
\global\long\def\calI{\mathcal{I}}%
\global\long\def\calT{\mathcal{T}}%
\global\long\def\calY{\mathcal{Y}}%
\global\long\def\calA{\mathcal{A}}%
\global\long\def\calB{\mathcal{B}}%
\global\long\def\calC{\mathcal{C}}%
\global\long\def\calG{\mathcal{G}}%
\global\long\def\calM{\mathcal{M}}%
\global\long\def\calP{\mathcal{P}}%
\global\long\def\calS{\mathcal{S}}%
\global\long\def\calU{\mathcal{U}}%
\global\long\def\Lmb{\Lambda}%
\global\long\def\eps{\varepsilon}%
\global\long\def\lmb{\lambda}%
\global\long\def\gmm{\gamma}%
\global\long\def\rd{\partial}%
\global\long\def\chf{\mathbf{1}}%
\global\long\def\td#1{\widetilde{#1}}%
\global\long\def\sgn{\mathrm{sgn}}%
\global\long\def\blambda{\boldsymbol\lambda}%
\global\long\def\biota{\boldsymbol\iota}%
\global\long\def\Red#1{\textcolor{red}{#1}}%
\global\long\def\NL{\mathrm{NL}}%

\title[Classification for HMHF and NLH]{Classification for the $2$-equivariant harmonic map heat flow and the radial energy-critical nonlinear heat equation in dimension $6$}


\author{Uihyeon Jeong}
\email{uihyeon.jeong@univie.ac.at}
\address{Fakultät für Mathematik, Oskar-Morgenstern-Platz 1, 1090 Wien, Austria}

\author{Taegyu Kim}
\email{k1216300@kias.re.kr}
\address{Korea Institute for Advanced Study, 80 Hoegi-ro, Dongdaemun-gu, Seoul 02455, Korea}

\begin{abstract}
    We consider the global dynamics of finite-energy solutions to the $2$-equivariant harmonic map heat flow from $\bbR^2$ to $\bbS^2$ and $\dot H^1$-bounded radial solutions to the energy-critical nonlinear heat equation in dimension $6$. Building upon soliton resolution, we obtain a classification of their multi-bubble dynamics. First, we prove that all such solutions are global. For the harmonic map heat flow, the number and sign of the bubbles are determined by the energy class of the initial data, while for the nonlinear heat equation the bubble signs necessarily alternate. In both models, the largest scale is determined, up to a positive multiplicative constant, by the spatial tail of the initial datum. The remaining scales concentrate according to exponential and iterated-exponential laws determined by the largest scale. 
    Finally, global bubble trees with an arbitrary number of bubbles exist for both models. For the harmonic map heat flow, this follows directly from the classification; for the nonlinear heat equation, we construct alternating bubble trees with any prescribed radial $\dot H^1$ tail, thereby realizing the scale laws arising in the classification.
\end{abstract}

\maketitle
\tableofcontents{}

\section{Introduction}
We consider the $D$-equivariant harmonic map heat flow
\begin{equation}\label{eq:general HMHF}
	v_t=v_{rr}+\frac1r v_r-\frac{D^2}{2r^2}\sin(2v),\qquad (t,r)\in[0,T_+)\times(0,\infty)
\end{equation}
with integer equivariance index $D\geq1$, and the radial energy-critical nonlinear heat equation
\begin{equation}\label{eq:general NLH}
	u_t=u_{rr}+\frac{d-1}{r}u_r+|u|^{\frac4{d-2}}u,\qquad (t,r)\in[0,T_+)\times(0,\infty)
\end{equation}
in spatial dimension $d\geq3$. Our results concern the cases $D=2$ and $d=6$, denoted by \eqref{eq:hmhf 2} and \eqref{eq:NLH 6} below.

Continuous-in-time bubble decompositions are known for finite-energy solutions to \eqref{eq:general HMHF} \cite{JL2023CVPDE} and $\dot H^1$-bounded radial solutions to \eqref{eq:general NLH} \cite{Aryan2024solResolHeat}. These results reduce the classification problem to determining whether finite-time blow-up can occur, which bubble configurations are possible, and how their scales evolve.

For $D\geq3$ and $d\geq7$, K.~Kim and Merle classified these dynamics \cite{KimMerle2025CPAM}. The solutions are global, the bubble signs agree for the harmonic map heat flow and alternate for the nonlinear heat equation, and the scales obey universal asymptotic laws up to a single positive constant. Whenever bubbles are present, the largest scale tends to a positive limit and the inner scales concentrate at polynomial rates. 

In contrast, the second author excluded bubble trees for $D=1$ and $d\in\{3,4,5\}$ \cite{Kim2026arXivNobubbletree}: every finite-time blow-up has exactly one bubble, while every global solution has at most one bubble at infinite time. This exclusion does not determine the asymptotic law of the remaining scale.

The present work treats the borderline cases $D=2$ and $d=6$. We prove that all finite-energy solutions to \eqref{eq:hmhf 2} and all $\dot H^1$-bounded radial solutions to \eqref{eq:NLH 6} are global, and classify their bubble configurations and scale laws. In both cases, the initial tail determines the largest scale up to a positive constant. The inner scales exhibit exponential and iterated-exponential concentration, with their asymptotic laws determined by the largest scale. We also construct six-dimensional nonlinear heat flows with an arbitrary number of bubbles and a largest-scale law determined by a prescribed radial $\dot H^1$ tail.

\subsection{The harmonic map heat flow}

The harmonic map heat flow for maps from $\bbR^2$ to the unit 2-sphere $\bbS^2$ is the $L^2$-gradient flow of the Dirichlet energy
\begin{equation*}
    E(\Phi) = \frac{1}{2} \int_{\mathbb{R}^2} |\nabla \Phi|^2 dx,\qquad \Phi:\mathbb{R}^2 \to \mathbb{S}^2.
\end{equation*}
The associated equation is given by
\begin{equation}\label{eq:original HMHF}
	  \partial_t\Phi=\Delta_{\bbR^2}\Phi+|\nabla\Phi|^2\Phi,\qquad (\Phi, \partial_t \Phi)(t) \in \mathbb{S}^2\times T_{\Phi} \bbS^2.  
\end{equation}
Eells and Sampson introduced the harmonic map heat flow to deform maps between Riemannian manifolds toward harmonic maps, which are the critical points of the Dirichlet energy \cite{EellsSampson1964AJM}. In two dimensions, equation \eqref{eq:original HMHF} is called \emph{energy-critical} since the scaling $\Phi(t,x)\mapsto\Phi(t/\lambda^2,x/\lambda)$ preserves both the equation and the energy.

For maps from a closed surface, Struwe established global existence of a weak solution in the energy class and proved that it is smooth away from finitely many spacetime singularities \cite{Struwe1985}. Near a singular point, suitable rescalings of the flow converge locally to a nonconstant harmonic map. This bubbling phenomenon has its counterpart in the harmonic map problem going back to Sacks and Uhlenbeck \cite{SacksUhlenbeck1981}. The energy identity and its localized versions play a central role in the analysis. A series of works by Qing \cite{Qing1995}, Ding and Tian \cite{DingTian1995}, Wang \cite{Wang1996}, Qing and Tian \cite{QingTian1997CPAM}, and Lin and Wang \cite{LinWang1998CVPDE} provided a more refined understanding of the energy identity and the associated compactness issues. In particular, along suitable sequences of times, the flow decomposes into a body map and finitely many bubbles. See also Topping \cite{Topping1997JDG,Topping2004AnnMath,Topping2004MathZ} for further results on bubbling and convergence.

To study the dynamics of these bubbles more precisely, we restrict to the equivariant setting. For an integer $D\geq1$, the $D$-equivariant symmetry class consists of maps of the form
\begin{equation*}
	\Phi(t,r,\theta)=(\sin v(t,r)\cos(D\theta),\sin v(t,r)\sin(D\theta),\cos v(t,r)).
\end{equation*}
Here $(r,\theta)$ are polar coordinates on $\bbR^2$. This class is preserved by the flow, and \eqref{eq:original HMHF} reduces to \eqref{eq:general HMHF}. The Dirichlet energy becomes
\begin{equation*}
	E(v)=\pi\int_0^\infty \bigg(v_r^2+\frac{D^2\sin^2v}{r^2}\bigg)r dr.
\end{equation*}
The finite-energy space is the disjoint union of the components $\calE_{\ell,m}$, $\ell,m\in\bbZ$, where
\begin{equation*}
	\calE_{\ell,m} \coloneqq \big\{v:E(v)<\infty,\ \lim_{r\to0}v(r)=\ell\pi,\ \lim_{r\to\infty}v(r)=m\pi\big\}.
\end{equation*}
We equip $\calE\coloneqq\calE_{0,0}$ with the norm 
\begin{equation}\label{eq:energy transform}
	\|h\|_{\calE}^2\coloneqq\int_0^\infty\bigg(h_r^2+\frac{h^2}{r^2}\bigg)r dr,
\end{equation}
which induces a metric on each $\calE_{\ell,m}$. Moreover, for the maximal solution $v$ on $[0,T_+)$ with $0< T_+ \le \infty$, the energy identity reads
\begin{equation}\label{eq:energy identity}
	E(v(t_2))+2\pi\int_{t_1}^{t_2}\int_0^{\infty} |v_t(t,r)|^2 rdrdt=E(v(t_1)), \qquad 0\leq t_1\leq t_2<T_+.
\end{equation}
The equation \eqref{eq:general HMHF} is invariant under $v\mapsto v+\pi$, $v\mapsto-v$, and the scaling $v(t,r)\mapsto v(t/\lambda^2,r/\lambda)$, $\lambda>0$. These transformations preserve the energy.

The normalized $D$-equivariant harmonic map $Q\in\calE_{0,1}$ and its rescalings are
\begin{equation*}
	Q(r)=2\arctan(r^D), \qquad Q_{[\lambda]}(r)=Q(r/\lambda).
\end{equation*}
Apart from the constants $n\pi$, every finite-energy stationary solution of \eqref{eq:general HMHF} is obtained from $Q$ by the symmetries above. The minimizers of $E$ in $\calE_{0,1}$ are precisely $Q_{[\lambda]}$, $\lambda>0$, and $E(Q_{[\lambda]})=4\pi D$.

Jendrej and Lawrie proved that every finite-energy equivariant solution admits a continuous-in-time bubble decomposition near its maximal time of existence \cite{JL2023CVPDE}; Proposition~\ref{thm:HMHF sol resol} records the result for $D=2$. Without symmetry, Jendrej, Lawrie, and Schlag established continuous-in-time convergence to the family of multi-bubble configurations \cite{JLS2025Pi}. Within equivariance, the boundary values constrain the signed sum of the bubbles, but they do not determine the number of bubbles until cancellations between opposite signs have been excluded.

We now set $D=2$, so \eqref{eq:general HMHF} becomes
\begin{equation}\tag{HMHF}\label{eq:hmhf 2}
	v_t=v_{rr}+\frac1r v_r-\frac2{r^2}\sin(2v),\qquad (t,r)\in[0,T_+)\times(0,\infty).
\end{equation}
Our goal is to classify the long-time dynamics of arbitrary finite-energy solutions to \eqref{eq:hmhf 2}, including the number and signs of the bubbles and the asymptotic behavior of their scales. In the single-bubble regime, Gustafson, Nakanishi, and Tsai obtained such a classification \cite{GustafsonNakanishiTsai2010CMP}: depending on the spatial decay of the initial datum, the scale may tend to zero, to a positive limit, or to infinity, and may also oscillate. 

Our first main theorem extends this picture to arbitrary finite-energy $2$-equivariant solutions. It determines the bubble configuration and shows that the initial tail fixes the largest scale, which in turn determines the exponential and iterated-exponential concentration of the inner scales.

To state the result, let $\log^{(p)}$ denote the $p$-fold composition of $\log$ for $p\geq1$. For $v_0\in\calE_{\ell,m}$, define
\begin{equation*}
	\bfD_0(r)\coloneqq \partial_rv_0(r)+\frac2r\sin(v_0(r)-m\pi).
\end{equation*}
We also define the initial scale function by
\begin{equation}\label{eq:intro HMHF initial tail}
	I_0(t)\coloneqq\exp\bigg\{-\frac1\pi\int_1^{\sqrt t}\iota\bfD_0(r)\,dr\bigg\},\quad \iota \in\{-1,1\},\quad t\geq1.
\end{equation}

\begin{thm}[Classification of $2$-equivariant solutions]\label{thm:hmhf main}
	Let $\ell,m\in\bbZ$ and $v_0\in\calE_{\ell,m}$ be an initial datum. Then, the corresponding solution $v(t)$ to \eqref{eq:hmhf 2} is global and admits the decomposition
	\begin{equation}\label{eq:main classification decomposition}
		v(t)-\ell\pi-\iota\sum_{j=1}^NQ_{[\lambda_j(t)]}\to0\quad\text{in}\quad\calE \quad\text{as}\quad t\to\infty
	\end{equation}
	for some constant $L=L(v_0)>0$, time $t_1^*\geq1$, and continuous scales $\lambda_j:[t_1^*,\infty)\to(0,\infty)$, $1\leq j\leq N$, where
	\begin{equation}\label{eq:main number sign}
		N=|m-\ell|,\qquad \iota=\sgn(m-\ell).
	\end{equation}
	For every sufficiently large $T_0\geq t_1^*$, the scales satisfy, as $t\to\infty$,
	\begin{align}
		\lambda_1(t)&=(L+o(1))I_0(t), && N\geq1,\label{eq:lambda1 initial law}\\
		\log\{\lambda_2(t)^{-1}\}&=\bigg(\frac{16}{\pi}+o(1)\bigg)\int_{T_0}^t\frac{ds}{\lambda_1(s)^2}, && N\geq2,\label{eq:second scale initial law}\\
		\log^{(k-1)}\{\lambda_k(t)^{-1}\}&=\bigg(\frac{32}{\pi}+o(1)\bigg)\int_{T_0}^t\frac{ds}{\lambda_1(s)^2}, && 3\leq k\leq N.\label{eq:inner scale initial law}
	\end{align}
\end{thm}
\begin{rem}
	For $N\geq1$, the scale $\lambda_1(t)$ may tend to zero, to a positive limit, or to infinity, and may also oscillate; all these behaviors already occur for $N=1$ in \cite[Theorem~1.2]{GustafsonNakanishiTsai2010CMP}. For $N\geq2$, every inner scale tends to zero. More precisely, finite energy and the initial-tail formula imply
	\begin{equation*}
		\lambda_1(t)=t^{o(1)}, \qquad \log^{(k-1)}\{\lambda_k(t)^{-1}\}=t^{1+o(1)}, \qquad 2\leq k\leq N.
	\end{equation*}
	In particular, every inner bubble concentrates, i.e., $\lambda_k(t)\to0$ for $2\leq k\leq N$.
\end{rem}

\begin{rem}
	Let $\ell,m\in\bbZ$ with $m\neq\ell$, and let $v_0\in\calE_{\ell,m}$. Assume that
	\begin{equation*}
		\int_{R_*}^{\infty}|\bfD_0(r)|dr<\infty
	\end{equation*}
	for some $R_*>0$. Let $\lambda_1(t)>\cdots>\lambda_N(t)$ be the scales from Theorem~\ref{thm:hmhf main}. Then there exists $L_\infty\in(0,\infty)$ such that
	\begin{equation*}
		\lambda_1(t)=L_\infty+o(1).
	\end{equation*}
	The inner scales exhibit exponential and iterated-exponential concentration:
	\begin{align*}
		\lambda_2(t)=e^{-(\frac{16}{\pi L_\infty^2}+o(1))t},\qquad
		\lambda_k(t)=\underbrace{e^{-e^{e^{\ldots e^{(\frac{32}{\pi L_\infty^2}+o(1))t}}}}}_{k-1\text{ exponentials}},\quad 3\leq k\leq N.
	\end{align*}
\end{rem}

\subsection{The energy-critical nonlinear heat equation}
Our second main result concerns the radial energy-critical nonlinear heat equation in dimension $6$. For $d\geq3$, the energy-critical nonlinear heat equation on $\bbR^d$ is given by
\begin{equation*}
	\partial_tu=\Delta_{\bbR^d} u+|u|^{\frac4{d-2}}u,
\end{equation*}
whose radial form is \eqref{eq:general NLH}. It is the $L^2$-gradient flow associated with the energy
\begin{equation*}
	E(u)\coloneqq\int_{\bbR^d}\bigg(\frac12|\nabla u|^2-\frac{d-2}{2d}|u|^{\frac{2d}{d-2}}\bigg)dx.
\end{equation*}
For $\lambda>0$, the scaling $u(t,x)\mapsto\lambda^{-(d-2)/2}u(\lambda^{-2}t,\lambda^{-1}x)$ preserves both the equation and the energy. Thus $\dot H^1(\bbR^d)$ is the critical energy space. The Cauchy problem is locally well-posed in $\dot H^1$ \cite{Weissler1980,BrezisCazenave1996}. We consider maximal radial solutions $u\in C([0,T_+);\dot H^1(\bbR^d))$, which are classical for positive times and satisfy
\begin{equation}\label{eq:NLH energy identity}
	E(u(t_2))+\int_{t_1}^{t_2}\|u_t(t)\|_{L^2}^2dt=E(u(t_1)),\qquad 0\leq t_1\leq t_2<T_+.
\end{equation}
The stationary equation $-\Delta_{\bbR^d}W=|W|^{4/(d-2)}W$ admits, up to scaling and sign, a unique nontrivial radial finite-energy solution, given by the Aubin--Talenti bubble
\begin{equation*}
	W(r)=\bigg(1+\frac{r^2}{d(d-2)}\bigg)^{-(d-2)/2},\qquad W_\lambda(r)=\lambda^{-(d-2)/2}W(r/\lambda).
\end{equation*}

Aryan proved soliton resolution for $\dot H^1(\bbR^d)$-bounded radial solutions in every dimension $d\geq3$ \cite{Aryan2024solResolHeat}; Proposition~\ref{thm:NLH sol resol} records the statement for $d=6$.

We now set $d=6$, so \eqref{eq:general NLH} becomes
\begin{equation}\tag{NLH}\label{eq:NLH 6}
	u_t=u_{rr}+\frac5r u_r+|u|u,\qquad (t,r)\in[0,T_+)\times(0,\infty).
\end{equation}
In the rest of this subsection, $\dot H^1=\dot H^1(\bbR^6)$ and $W(r)=(1+r^2/24)^{-2}$. Our goal is to classify the dynamics of arbitrary $\dot H^1$-bounded radial solutions to \eqref{eq:NLH 6}, including their global existence, bubble configurations, and scale asymptotics.

Near a single bubble, Harada classified the dynamics for initial data $u_0\in H^1(\bbR^6)$ sufficiently close to $W$ in $\dot H^1$ \cite{Harada2026CVPDE}. Such solutions either converge to a stationary solution, decay to zero, or undergo finite-time type I blow-up. For multi-bubble dynamics, see \cite{delPinoMussoWei2021AnalPDE}.

Our second main theorem classifies arbitrary $\dot H^1$-bounded radial solutions. It determines the bubble configuration and the inner-scale laws, with the latter depending on the initial tail through the largest scale. The boundedness assumption cannot be omitted: Harada constructed a radial type II finite-time blow-up solution in dimension six whose $\dot H^1$ norm is unbounded \cite{Harada2020annPDE}.

For radial $u_0\in\dot H^1$ and $\iota\in\{-1,1\}$, define
\begin{equation}\label{eq:intro NLH initial tail}
	I_0(t)\coloneqq\exp\bigg\{\frac{5\iota}{16}\int_1^{\sqrt t}u_0(r)r\,dr\bigg\},\qquad t\geq1.
\end{equation}

\begin{thm}[Classification of $\dot H^1$-bounded radial solutions in dimension $6$]\label{thm:nlh main}
    Let $u_0\in\dot H^1$ be a radial initial datum, and let $u(t)$ be the corresponding radial solution to \eqref{eq:NLH 6} with maximal forward time of existence $T_+\in(0,\infty]$. Assume
    \begin{equation}\label{eq:NLH bounded}
    	\sup_{0\leq t<T_+}\|u(t)\|_{\dot H^1}<\infty.
    \end{equation}
    Then $T_+=\infty$, and $u$ admits the decomposition
	\begin{equation}\label{eq:NLH alternating signs}
		u(t)-\iota\sum_{j=1}^N(-1)^{j-1}W_{\lambda_j(t)}\to0\quad\text{in}\quad\dot H^1\quad\text{as}\quad t\to\infty
	\end{equation}
	for some integer $N\geq0$, sign $\iota\in\{-1,1\}$, constant $L=L(u_0)>0$, time $t_1^*\geq1$, and continuous scales $\lambda_j:[t_1^*,\infty)\to(0,\infty)$, $1\leq j\leq N$. Moreover, for every sufficiently large $T_0\geq t_1^*$, the scales $\lambda_j(t)$ satisfy, as $t\to\infty$,
	\begin{align}
		\lambda_1(t)&=(L+o(1))I_0(t), && N\geq1,\label{eq:NLH lambda1 initial law}\\
		\log\{\lambda_2(t)^{-1}\}&=\bigg(\frac54+o(1)\bigg)\int_{T_0}^t\frac{ds}{\lambda_1(s)^2}, && N\geq2,\label{eq:NLH second scale initial law}\\
		\log^{(k-1)}\{\lambda_k(t)^{-1}\}&=\bigg(\frac52+o(1)\bigg)\int_{T_0}^t\frac{ds}{\lambda_1(s)^2}, && 3\leq k\leq N.\label{eq:NLH inner scale initial law}
	\end{align}
\end{thm}

For \eqref{eq:hmhf 2}, every number of bubbles occurs by choosing the component $\calE_{\ell,m}$. For \eqref{eq:NLH 6}, there is no corresponding topological constraint, and the existence of an alternating bubble tree requires a separate argument. Our next theorem constructs such a solution for every number of bubbles and every prescribed radial tail $q\in\dot H^1$. The prescribed tail determines the largest-scale law up to a positive constant and thereby the asymptotic laws of the inner scales. Define
\begin{equation}\label{eq:intro NLH prescribed tail}
	I_q(t)\coloneqq\exp\bigg\{\frac5{16}\int_1^{\sqrt t}q(r)r\,dr\bigg\},\qquad t\geq1.
\end{equation}

\begin{thm}[Construction of bubble trees]\label{thm:NLH bubble tree construction}
	Let $N\geq1$ be an integer, and let $q\in\dot H^1$ be radial. Then there exists a radial initial datum $u_0\in\dot H^1$ such that the corresponding $\dot H^1$-solution $u$ of \eqref{eq:NLH 6} is global, satisfies \eqref{eq:NLH bounded}, and has the following properties. There exist $t_1^*\geq1$ and $C^1$ scales $\lambda_1,\ldots,\lambda_N:[t_1^*,\infty)\to(0,\infty)$ such that
	\begin{equation}\label{eq:NLH constructed bubble tree}
		\bigg\|u(t)-\sum_{j=1}^N(-1)^{j-1}W_{\lambda_j(t)}\bigg\|_{\dot H^1}+\sum_{j=2}^N\frac{\lambda_j(t)}{\lambda_{j-1}(t)}\to0 \quad\text{as}\quad t\to\infty.
	\end{equation}
	Moreover, for some $L>0$,
	\begin{equation*}
		\lambda_1(t)=(L+o(1))I_q(t).
	\end{equation*}
	If $N\geq2$, then for any sufficiently large $T_0\geq t_1^*$,
	\begin{align}
		\log\{\lambda_2(t)^{-1}\}&=\bigg(\frac54+o(1)\bigg)\int_{T_0}^t\frac{ds}{\lambda_1(s)^2}, \label{eq:NLH constructed second scale}
        \\
        \log^{(k-1)}\{\lambda_k(t)^{-1}\}&=\bigg(\frac52+o(1)\bigg)\int_{T_0}^t\frac{ds}{\lambda_1(s)^2},\qquad 3\leq k\leq N. \label{eq:NLH constructed inner scales}
	\end{align}
\end{thm}

\subsection{Discussion of the results}

\begin{rem}[Method and novelty]
	Our proof is based on the energy method with modulation analysis, following K.~Kim--Merle \cite{KimMerle2025CPAM}. In the present cases $D=2$ and $d=6$, the interactions between the bubbles can be estimated directly, without introducing elliptic corrections to the multi-bubble profile. Combining these estimates with coercivity and energy dissipation gives the spacetime control needed to integrate the modulation equations.

	The main difficulty comes from the slow decay of the scaling direction, which prevents a direct single-scale correction to the modulation law. We adapt the approach of the second author in \cite{Kim2026arXivNobubbletree} and study the relative dynamics of $\lambda_k$ and $\lambda_{k-1}$ for every $2\leq k\leq N$, rather than deriving a refined evolution law for each scale separately. The common leading tails of the two scaling directions cancel in their difference, leading to coupled modulation laws for consecutive scales. These laws determine the relative signs of the bubbles and reduce the inner-scale asymptotics to the dynamics of the largest scale. The largest-scale law is obtained separately by adapting the exterior argument of Gustafson--Nakanishi--Tsai \cite[Section~9]{GustafsonNakanishiTsai2010CMP} to the multi-bubble setting in both models.
\end{rem}

\begin{rem}[Dependence on the initial tail]
	As explained in \cite{GustafsonNakanishiTsai2010CMP}, the largest scale $\lambda_1$ is affected by the cumulative contribution of the initial tail over spatial regions reached by the heat flow. Finite energy controls the size of the tail but does not force this cumulative contribution to converge, so $\lambda_1$ may retain a nontrivial dependence on the initial data at arbitrarily large distances. In a bubble tree, the inner scales inherit this dependence through their interactions with the next larger bubbles.
\end{rem}

\begin{rem}[Equivariance and spatial dimension]
	For harmonic map heat flow, the dynamics depend strongly on the equivariance index. When $D\geq3$, asymptotic stability of a single harmonic map, with convergence of its scale to a positive limit, was proved in \cite{GustafsonNakanishiTsai2010CMP}. At the borderline $D=2$, the scale may instead depend on the initial tail \cite{GustafsonNakanishiTsai2010CMP}. Below the borderline, finite-time blow-up already occurs for $D=1$ \cite{CDY1992JDG,RaphaelSchweyer2013CPAMHeat,RaphaelSchweyer2014AnalPDEHeatQuantized}, while global single-bubble regimes were constructed in \cite{WZZ2026JFA}.

	For the energy-critical nonlinear heat equation, the corresponding borderline is the spatial dimension $d=6$. When $d\geq7$, the dynamics near the ground-state manifold were classified in \cite{CollotMerleRaphael2017CMP}, and alternating infinite-time bubble towers were constructed in \cite{delPinoMussoWei2021AnalPDE}. At $d=6$, near-ground-state dynamics were studied in \cite{Harada2026CVPDE}, while two- and three-bubble regimes were considered in \cite{delPinoMussoWei2021AnalPDE}; our construction extends the latter to an arbitrary number of bubbles. Below the borderline, finite-time type II blow-up occurs for $d=3,4,5$ \cite{Schweyer2012JFA,PMW2019,delPinoMussoWeiZhang2020HeatQuantized,Harada2020AIHPC}, together with global single-bubble concentration \cite{delPinoMussoWei2020Infheat,WZZ2024JDE,LWZZ2024}; see also \cite{FilippasHerreroVelazquez2000} for earlier formal predictions of discrete finite-time blow-up rates.

	At the level of arbitrary solutions, K.~Kim and Merle classified finite-energy equivariant solutions to harmonic map heat flow for $D\geq3$ and $\dot H^1$-bounded radial solutions to the energy-critical nonlinear heat equation for $d\geq7$, proving globality, sign rigidity, and universal polynomial scale laws \cite{KimMerle2025CPAM}. Our results extend the corresponding classifications to the borderline cases $D=2$ and $d=6$. For the nonlinear heat equation, the $\dot H^1$-boundedness assumption is essential; see \cite{Harada2020annPDE} for a finite-time type II blow-up solution outside the $\dot H^1$-bounded regime. Below the borderline, bubble trees are excluded for $D=1$ and $d=3,4,5$ under the $\dot H^1$-boundedness assumptions \cite{Kim2026arXivNobubbletree} by the second author.

	For related multi-bubble rigidity results for the energy-critical nonlinear heat equation without radial symmetry when $d\geq7$, see \cite{KimMerle2026arXiv}.
\end{rem}

\begin{rem}[On bubble trees]
	Bubble-tree (or bubble-tower) dynamics also arise in several related parabolic problems. Infinite-time or ancient trees have been constructed for harmonic map heat flow with suitable targets, the Yamabe flow, and the high-dimensional energy-critical nonlinear heat equation \cite{Topping2000,DPS2018crelle,delPinoMussoWei2021AnalPDE,SunWeiZhang2022CVPDE}. In the disk problem for harmonic map heat flow, finite-time bubble trees were excluded for $D=1$ in \cite{Hout2003JDE} and, more recently, for every $D\geq1$ in \cite{Samuelian2026CVPDE}; see also \cite{Samuelian2026arXiv}.

	In dispersive settings, two-bubble solutions and their threshold dynamics have been studied for several critical wave, Schr\"odinger, Hartree, and wave-map equations \cite{Jendrej2019AJM,Jendrej2017AnalPDE,JendrejLawrie2018Invent,JendrejLiXu2026arXiv,JendrejLiXu2026arXivNLS}.
    More recently, finite-time wave-map bubble trees were first constructed with two bubbles in \cite{JendrejKrieger2025arXiv} and then with arbitrarily many bubbles in \cite{KriegerPalacios2026arXiv}; infinite-time trees with arbitrarily many bubbles were constructed in \cite{HwangKim2026arXiv}.
    By contrast, at most one soliton occurs for the equivariant self-dual Chern--Simons--Schr\"odinger equation \cite{KimKwonOh2025AJM}, while bubble trees are excluded for the Calogero--Moser derivative nonlinear Schr\"odinger equation \cite{KimTKwon2024arxivSolResol}. See also \cite{Shen2026arXiv} for the radial energy-critical wave equation.
\end{rem}

\begin{rem}[Classification of blow-up dynamics]
	For the semilinear heat equation 
    \begin{equation*}
        u_t=\Delta_{\bbR^d}u+|u|^{p-1}u,\qquad p>1,
    \end{equation*}
    blow-up is of type I if $\|u(t)\|_{L^\infty}\lesssim(T_+-t)^{-1/(p-1)}$, and of type II otherwise. In the subcritical range $p<p_S=(d+2)/(d-2)$, every finite-time blow-up is of type I by Giga--Kohn \cite{GigaKohn1985CPAM} and Giga--Matsui--Sasayama \cite{GMS2004}. In the supercritical range $p>p_S$, the Joseph--Lundgren exponent $p_{JL}>p_S$ marks a further distinction: when $p>p_{JL}$, Herrero and Vel\'azquez constructed radial type II solutions with discrete polynomial rates \cite{HerreroVelazquez1992,HerreroVelazquez1994CRASPSI}, and Mizoguchi classified the rates of radial nonnegative type II solutions under additional assumptions on the initial data \cite{Mizoguchi2007Math.Ann.,Mizoguchi2011TranAMS}. In the intermediate range $p_S<p<p_{JL}$, Matano and Merle excluded radial type II blow-up \cite{MatanoMerle2004CPAM}. At the energy-critical exponent $p=p_S$ in dimensions $d\geq 7$, Wang and Wei obtained the corresponding exclusion for nonnegative data without radial symmetry \cite{WangWei2021arXiv}. For sign-changing solutions, see also the radial classification and non-radial multi-bubble rigidity results in \cite{KimMerle2025CPAM,KimMerle2026arXiv}.

    Related classification problems arise throughout critical evolution equations. For single-bubble dynamics, we refer to the series of works on the mass-critical nonlinear Schr\"odinger equation \cite{MerleRaphael2005AnnMath,MerleRaphael2003GAFA,Raphael2005MathAnnalen,MerleRaphael2004Invent,MerleRaphael2006JAMS,MerleRaphael2005CMP}, the mass-critical generalized Korteweg--de Vries equation \cite{MartelMerleRaphael2014Acta}, and other Schr\"odinger-type equations \cite{Kim2025JEMS,JeongKimKimKwon2026arXiv}. In the parabolic setting, refined blow-up asymptotics for the simplified chemotaxis system were obtained in \cite{Mizoguchi2022CPAM}. For semilinear wave equations, related blow-up-rate estimates were established in \cite{MerleZaag2003AJM,MerleZaag2005Math.Ann.}; in dimension $1$, the multi-soliton dynamics near characteristic blow-up points were classified in \cite{MerleZaag2012AJM} and constructed in \cite{CoteZaag2013CPAM}.
\end{rem}

\subsection{Statements of soliton resolution}

We record the soliton resolutions that start our analysis. They allow arbitrary bubble signs and cover both finite-time blow-up and global alternatives. The globality, sign restrictions, and scale laws in the main theorems are obtained by studying the dynamics of these decompositions.

\begin{prop}[Bubble decomposition for $2$-equivariant solutions, \cite{JL2023CVPDE}]\label{thm:HMHF sol resol}
	Let $\ell,m\in\bbZ$ and $v_0\in\calE_{\ell,m}$ be an initial datum; let $v(t)$ be the corresponding solution to \eqref{eq:hmhf 2} with $T_+\in(0,\infty]$ its maximal time of existence.

	\emph{(Finite-time blow-up case)} If $T_+<\infty$, then there exist $T_0<T_+$, an integer $N\geq1$, signs $\iota_1,\ldots,\iota_N\in\{-1,1\}$, continuous functions $\wt\lambda_1,\ldots,\wt\lambda_N:[T_0,T_+)\to(0,\infty)$, and a function $v^*\in\calE_{0,m-\ell-\sum_{j=1}^N\iota_j}$ such that
	\begin{equation}
		\bigg\|v(t)-\ell\pi-\sum_{j=1}^N\iota_jQ_{[\wt\lambda_j(t)]}-v^*\bigg\|_{\calE}
		+\sum_{j=2}^N\frac{\wt\lambda_j(t)}{\wt\lambda_{j-1}(t)}+\frac{\wt\lambda_1(t)}{\sqrt{T_+-t}}\to0\quad\text{as}\quad t\to T_+.
	\end{equation}

	\emph{(Global case)} If $T_+=\infty$, then there exist $T_0>0$, an integer $N\geq0$, signs $\iota_1,\ldots,\iota_N\in\{-1,1\}$, and continuous functions $\wt\lambda_1,\ldots,\wt\lambda_N:[T_0,\infty)\to(0,\infty)$ such that
	\begin{equation*}
		\bigg\|v(t)-\ell\pi-\sum_{j=1}^N\iota_jQ_{[\wt\lambda_j(t)]}\bigg\|_{\calE}
		+\sum_{j=2}^N\frac{\wt\lambda_j(t)}{\wt\lambda_{j-1}(t)}+\frac{\wt\lambda_1(t)}{\sqrt t}\to0\quad\text{as}\quad t\to\infty.
	\end{equation*}
\end{prop}

\begin{prop}[Soliton resolution for $\dot{H}^1$-bounded radial solutions, \cite{Aryan2024solResolHeat}]\label{thm:NLH sol resol}
	Let $u\in C([0,T_+);\dot H^1)$ be a radial solution of \eqref{eq:NLH 6} with $T_+\in(0,\infty]$ its maximal time of existence. Assume
	\begin{equation*}
		\sup_{0\leq t<T_+}\|u(t)\|_{\dot H^1}<\infty.
	\end{equation*}

	\emph{(Finite-time blow-up case)} If $T_+<\infty$, then there exist an integer $N\geq1$, signs $\iota_1,\ldots,\iota_N\in\{-1,1\}$, continuous functions $\wt\lambda_1,\ldots,\wt\lambda_N$ with $\wt\lambda_1(t)>\cdots>\wt\lambda_N(t)>0$, and a unique radial function $u^*\in\dot H^1$ such that
	\begin{equation}\label{eq:NLH finite decomposition}
		\bigg\|u(t)-\sum_{j=1}^N\iota_jW_{\wt\lambda_j(t)}-u^*\bigg\|_{\dot H^1}
		+\sum_{j=2}^N\frac{\wt\lambda_j(t)}{\wt\lambda_{j-1}(t)}+\frac{\wt\lambda_1(t)}{\sqrt{T_+-t}}\to0\quad\text{as}\quad t\to T_+.
	\end{equation}

	\emph{(Global case)} If $T_+=\infty$, then there exist an integer $N\geq0$, signs $\iota_1,\ldots,\iota_N\in\{-1,1\}$, and continuous functions $\wt\lambda_1,\ldots,\wt\lambda_N$ with $\wt\lambda_1(t)>\cdots>\wt\lambda_N(t)>0$ such that
	\begin{equation}\label{eq:NLH global decomposition}
		\bigg\|u(t)-\sum_{j=1}^N\iota_jW_{\wt\lambda_j(t)}\bigg\|_{\dot H^1}
		+\sum_{j=2}^N\frac{\wt\lambda_j(t)}{\wt\lambda_{j-1}(t)}+\frac{\wt\lambda_1(t)}{\sqrt t}\to0\quad\text{as}\quad t\to\infty.
	\end{equation}
\end{prop}

The proofs of these decomposition results are motivated by developments in soliton resolution for critical wave equations. These include the radial critical wave equation \cite{DuyckaertsKenigMerle2013Camb}, equivariant wave maps \cite{Cote2015CPAMsolitonResol}, sequential resolution results \cite{JiaKenig2017AJMwaveSolResol,DJKM2017GaFASolresolSequence}, and continuous-in-time resolution results in \cite{DuyckaertsKenigMerle2023Acta,JendrejLawrie2025JAMS,JendrejLawrie2023AnnPDESolResol}.

\subsection{Strategy of the proof}

We start from soliton resolution and study the dynamics of modulation parameters of the resulting bubble decomposition. We use the unified notation of Section~\ref{sec:notation}. On a first reading, we suggest considering only the case of \eqref{eq:NLH 6}.

For \eqref{eq:hmhf 2}, we first replace $v$ by $v-\ell\pi$ and set $u=r^{-2}v$; for \eqref{eq:NLH 6}, we keep $u$ unchanged. Both equations then take the form
\begin{equation*}
	u_t=\calT(u)=\Delta u+r^{-4}f(r^2u),
\end{equation*}
where $\Delta$ is the radial Laplacian on $\bbR^6$, and $f$ and $W=r^{-2}Q$ are given by \eqref{eq:Q f}.

Starting from Propositions~\ref{thm:HMHF sol resol} and~\ref{thm:NLH sol resol}, we choose nearby $C^1$ scales and write
\begin{gather*}
	u(t)=U(t)+g(t),\qquad U(t)=\sum_{j=1}^N\iota_jW_{\lambda_j(t)},\qquad (Z_{\underline{\lambda_j}},g)=0,\quad 1\leq j\leq N,\\
	\lambda_N(t)\ll\cdots\ll\lambda_1(t)\ll\min\{\sqrt{T_+-t},\sqrt t\},
\end{gather*}
where $Z$ is the compactly supported test function in \eqref{eq:Z normalization}, and $T_+-t$ is interpreted as $\infty$ when $T_+=\infty$. The remainder $g$ converges in $\dot H^1$ to the asymptotic profile given by soliton resolution; in particular, $\|g(t)\|_{\dot H^1}\to0$ when $T_+=\infty$.

We first explain how energy dissipation controls the remainder. For the decomposition $u=U+g$, the equation takes the form
\begin{equation*}
	u_t=\calT(U)-H_Ug+\NL_U(g).
\end{equation*}
Since each bubble is stationary, $\calT(U)$ consists of the interactions between bubbles. Separating the leading interactions between consecutive bubbles, we write
\begin{equation}\label{eq:intro profile leading order}
	\calT(U)=-\sum_{j=2}^N\frac{\iota_{j-1}\kappa}{\lambda_{j-1}^2}(\Lambda W)_{\lambda_j}+\text{error},
\end{equation}
where $\kappa$ is given by \eqref{eq:kappa}. The displayed interactions determine the modulation laws, while the error must be controlled.

The key observation is that the interactions also contribute to energy dissipation. Although $\calT(U)$ prevents direct control of $H_Ug$ by $u_t$, its leading terms lie in the scaling directions, which are almost orthogonal to $H_Ug$ by the identities $H_{\lambda_j}(\Lambda W)_{\lambda_j}=0$, self-adjointness, and scale separation. It therefore remains to control the error in \eqref{eq:intro profile leading order}. In higher dimensions this error is not integrable and requires the profile correction of \cite{KimMerle2025CPAM}, whereas here it is already small enough for the simple sum $U$. Combining these two facts, we obtain
\begin{equation*}
	\|u_t\|_{L^2}^2\approx\|\calT(U)-H_Ug\|_{L^2}^2\gtrsim\|\calT(U)\|_{L^2}^2+\|H_Ug\|_{L^2}^2.
\end{equation*}
By the coercivity, we have $\|H_Ug\|_{L^2}^2\approx\|g\|_{\dot H^2}^2$. Moreover, \eqref{eq:intro profile leading order} implies
\begin{equation*}
	\|\calT(U)\|_{L^2}^2\sim\sum_{j=2}^N\frac{\lambda_j^2}{\lambda_{j-1}^4}\eqqcolon\calD.
\end{equation*}
Consequently, the dissipation controls both the remainder $\|g\|_{\dot H^2}^2$ and the bubble interactions $\|\calT(U)\|_{L^2}^2$.

When $\|g\|_{\dot H^1}$ is small, the estimate $\|\NL_U(g)\|_{L^2}\lesssim\|g\|_{\dot H^1}\|g\|_{\dot H^2}$ allows us to absorb the nonlinear term and obtain $\|u_t\|_{L^2}^2\gtrsim\calD+\|g\|_{\dot H^2}^2$. The energy identity therefore yields
\begin{equation}\label{eq:intro spacetime}
	\int_{t_0^*}^{T_+}(\|u_t\|_{L^2}^2+\calD(t)+\|g(t)\|_{\dot H^2}^2)dt<\infty.
\end{equation}
In the finite-time blow-up case, $g$ need not be small globally in $\dot H^1$, so an additional localization argument is required.

We now turn to the modulation law. Since $U_t=-\sum_{j=1}^N\iota_j(\lambda_{j,t}/\lambda_j)(\Lambda W)_{\lambda_j}$, comparison with \eqref{eq:intro profile leading order} formally suggests
\begin{equation*}
	\frac{\lambda_{1,t}}{\lambda_1}\approx0,\qquad \frac{\lambda_{j,t}}{\lambda_j}\approx\iota_{j-1}\iota_j\frac{\kappa}{\lambda_{j-1}^2},\qquad 2\leq j\leq N.
\end{equation*}
The first equation suggests an asymptotically constant largest scale, and successive integration predicts exponential and iterated-exponential laws for the inner scales, with the signs constrained by scale separation.

The difficulty in justifying this picture is that $r\Lambda W\notin L^2$. Thus the single-scale correction $((\Lambda W)_{\underline{\lambda_j}},g)$ need not be defined in the energy class. For the largest scale, a fixed spatial cutoff $\chi$ suffices to exclude finite-time blow-up. Indeed, if $T_+<\infty$, the spacetime estimate \eqref{eq:intro spacetime} makes the errors integrable, while the localized correction satisfies $|((\Lambda W)_{\underline{\lambda_1}},g\chi)|\lesssim(1+\log\lambda_1^{-1})^{1/2}$. This growth cannot compensate for the divergence of $\log\lambda_1$ as $\lambda_1\to0$, and hence $T_+=\infty$.

For the inner scales, introducing cutoffs would produce further errors; see \cite{KimKimKwon2024arxiv} for a related discussion. We instead study the relative dynamics of $\lambda_k$ and $\lambda_{k-1}$, $2\leq k\leq N$, by comparing their scaling directions as in \cite{Kim2026arXivNobubbletree}. Since the leading $r^{-4}$ tail of $(\Lambda W)_{\underline\lambda}$ is independent of $\lambda$, it cancels in the difference $(\Lambda W)_{\underline{\lambda_k}}-(\Lambda W)_{\underline{\lambda_{k-1}}}$. Taking inner products against this difference yields the coupled laws
\begin{equation*}
	\frac d{dt}(\log\lambda_k+\varsigma\iota_{k-1}\iota_k\log\lambda_{k-1})\approx\iota_{k-1}\iota_k\frac{\kappa}{\lambda_{k-1}^2},\qquad 2\leq k\leq N,
\end{equation*}
where $\varsigma=1$ for \eqref{eq:hmhf 2} and $\varsigma=-1$ for \eqref{eq:NLH 6}. The actual laws in Proposition~\ref{prop:modulation} contain error terms. 
These coupled laws determine the signs. Since $\varsigma\kappa<0$, the choice $\iota_{k-1}\iota_k=-\varsigma$ would force $\log(\lambda_k/\lambda_{k-1})\to+\infty$ after integration, because $\int_{T_0}^t\lambda_{k-1}(s)^{-2}ds\to\infty$. This contradicts scale separation. Thus $\iota_{k-1}\iota_k=\varsigma$: the signs agree for \eqref{eq:hmhf 2} and alternate for \eqref{eq:NLH 6}.

Unlike the formal prediction, the correction to the largest-scale $\lambda_1$ equation need not be negligible at large times. Thus, we determine its contribution by adapting the exterior argument of Gustafson--Nakanishi--Tsai \cite[Section~9]{GustafsonNakanishiTsai2010CMP}, obtaining $\lambda_1(t)=(L+o_t(1))I_0(t)$ for some $L>0$, where $I_0$ is defined in \eqref{eq:intro HMHF initial tail} or \eqref{eq:intro NLH initial tail}. With the signs fixed, integration of the coupled laws gives
\begin{equation*}
	\log\lambda_k(t)^{-1}=(|\kappa|+o_t(1))\int_{T_0}^t\frac{ds}{\lambda_{k-1}(s)^2},\qquad 2\leq k\leq N.
\end{equation*}
Successively integrating the coupled modulation laws yields the exponential and iterated-exponential laws in the main theorems, with their time dependence determined by the largest scale. This completes the classification.

For the construction of solutions to \eqref{eq:NLH 6}, we adapt the single-bubble energy method of \cite{CollotMerleRaphael2017CMP} to the multi-bubble setting. We construct a solution which stays in a small tube around the family of well-separated alternating $N$-bubble profiles, and then apply the classification theorem to determine its scale laws. Thus the scale asymptotics need not be prescribed in the construction.

The energy estimates control the remainder and scale separation as long as the unstable components remain small. Since $H\calY=-e_0\calY$ with $e_0>0$, these components formally satisfy
\begin{equation*}
	\frac d{dt}(\iota_j\calY_{\underline{\lambda_j}},g)
	\approx\frac{e_0}{\lambda_j^2}(\iota_j\calY_{\underline{\lambda_j}},g)
	+2(\calY,W)\sum_{i<j}\frac{\iota_i\iota_j}{\lambda_i^2},\qquad 1\leq j\leq N.
\end{equation*}
We vary one initial parameter along each unstable direction. As long as these components remain controlled, the energy estimates keep the solution inside the multi-bubble tube. The integrated modulation laws show that any finite exit must occur through an unstable component and is directed outward. A topological shooting argument therefore selects initial parameters for which no exit occurs, yielding a global solution that remains in the tube for all forward times.

To prescribe the largest-scale law, we choose the initial data so that $u_0-q\in H^1$ and vary the unstable parameters by compactly supported perturbations, which do not affect the prescribed tail. The classification theorem then identifies exactly $N$ bubbles and yields $\lambda_1(t)=(L+o_t(1))I_q(t)$ for some $L>0$, where $I_q$ is the tail function defined in \eqref{eq:intro NLH prescribed tail}. The inner-scale laws follow from the same classification.

\vspace{5pt}
\noindent\mbox{\textbf{Acknowledgements.~}} Part of this work was carried out while U.~Jeong was partially supported by the National Research Foundation of Korea (RS-2022-NR069873, RS-2024-00333393). T.~Kim is supported by a KIAS Individual Grant (MG105201) at Korea Institute for Advanced Study. The author acknowledges the use of ChatGPT for improving the presentation and readability of the manuscript.

\section{Notation and preliminaries}\label{sec:notation}

\subsection{Notation} 

Throughout the paper, $D=2$ is the equivariance index, $N\in\bbN$ is the number of bubbles, $1\leq j,k\leq N$ are bubble indices, and $\iota_j\in\{-1,1\}$ is the sign of the $j$-th bubble. 

For quantities $A\in\bbR$ and $B\geq0$, we write $A\lesssim B$ if $|A|\leq CB$ holds for some implicit constant $C$. For $A,B\geq0$, we write $A\sim B$ if $A\lesssim B$ and $B\lesssim A$. If $C$ is allowed to depend on some parameters, then we write them as subscripts of $\lesssim,\sim,\gtrsim$ to indicate the dependence. Dependence on the fixed number $N$ of bubbles is suppressed.
For positive functions $a(t)$ and $b(t)$, the notation $a(t)\ll b(t)$ as $t\to T_+$ means $a(t)/b(t)\to0$. The constants implicit in $O(\cdot)$ and $\lesssim$ are independent of time and of the bubble scales unless a dependence is displayed explicitly.

The symbol $\mathbf1_A$ denotes the indicator function of a set $A$. We fix a radial cutoff $\chi\in C_c^\infty([0,\infty))$ satisfying $0\leq\chi\leq1$, $\chi(r)=1$ for $0\leq r\leq1$, and $\chi(r)=0$ for $r\geq2$, and define $\chi_R(r)\coloneqq\chi(r/R)$. We use $\delta_{jk}$ for the Kronecker delta and write $\langle r\rangle\coloneqq(1+r^2)^{1/2}$. For $p\geq1$, we denote by $\log^{(p)}$ the $p$-fold iterate of $\log$, with $\log^{(0)}F=F$.

Unless otherwise specified, $L^p$, $H^s$, and $\dot H^s$ denote the corresponding spaces on $\bbR^6$. We use the dual space $\dot H^{-1}=(\dot H^1)^*$. For radial functions, we suppress the spherical constant $|\bbS^5|$ in norms and inner products. Thus, for $1\leq p<\infty$,
\begin{equation*}
    \norm{h}_{L^p} \coloneqq \bigg(\int_0^\infty |h(r)|^p r^5 dr\bigg)^{1/p}, \qquad (h_1,h_2) \coloneqq \int_0^\infty h_1(r)h_2(r)r^5 dr.
\end{equation*}
From this point onward, we divide the energy for \eqref{eq:NLH 6} by the spherical constant $|\mathbb S^5|$ and continue to denote it by $E$.
For two-dimensional radial functions, we retain the subscript $2$:
\begin{equation*}
    \norm{h}_{L^p_2} \coloneqq \bigg(\int_0^\infty |h(r)|^p r dr\bigg)^{1/p}, \qquad (h_1,h_2)_2 \coloneqq \int_0^\infty h_1(r)h_2(r)r dr.
\end{equation*}
We also use
\begin{equation*}
    \norm{h}_{\dot H^1}\coloneqq \norm{\partial_rh}_{L^2}, \quad 
    \norm{h}_{H^1}^2\coloneqq \norm{h}_{L^2}^2 +\norm{h}_{\dot H^1}^2, \quad
    \norm{h}_{\dot H^2}\coloneqq \norm{\Delta h}_{L^2}.
\end{equation*}
Here, on radial functions,
\begin{equation*}
    \Delta\coloneqq\Delta_{\bbR^6}=\partial_{rr}+\frac5r\partial_r,\qquad \Delta_2\coloneqq\Delta_{\bbR^2}=\partial_{rr}+\frac1r\partial_r.
\end{equation*}
For $1\leq p\leq\infty$, we write $L^p_t$ for Lebesgue norms in time with measure $dt$, and use variable subscripts analogously for $s$, $\sigma$, and $\tau_k$. For mixed time and space norms, we write $L^p_tL^q$ or $L^p_t\dot H^s$; time intervals are specified in parentheses when needed. 
For radial functions on $\bbR^6$, direct conjugation gives
\begin{equation*}
    (\phi_1,\phi_2)=(r^2\phi_1,r^2\phi_2)_2,
    \quad
    r^2\Delta\phi=\Delta_2^{(2)}(r^2\phi),
    \quad
    \Delta_2^{(2)}\coloneqq\partial_{rr}+\frac1r\partial_r-\frac4{r^2}.
\end{equation*}

If $q$ is a two-dimensional profile and $h=r^{-2}q$ is the corresponding six-dimensional profile, we use the scalings and generators
\begin{equation}\label{eq:scaling}
    \begin{aligned}
        q_{[\lambda]}(r)&\coloneqq q (r/\lambda),\\
        h_\lambda(r)&\coloneqq\lambda^{-2}h(r/\lambda), &\Lambda h&\coloneqq (2+r\partial_r)h,\\
        h_{\underline{\lambda}}(r)&\coloneqq\lambda^{-2}h_\lambda(r) =\lambda^{-4}h(r/\lambda), &\Lambda_{-1}h&\coloneqq (\Lambda+2)h.
    \end{aligned}
\end{equation}
We have $(\Lambda h_1,h_2)=-(h_1,\Lambda_{-1}h_2)$.
If $h=r^{-2}q$, then $ h_\lambda=r^{-2}q_{[\lambda]}$.
The scale derivatives are
\begin{equation*}
    \lambda\partial_\lambda q_{[\lambda]} =-(r\partial_rq)_{[\lambda]},\qquad
    \lambda\partial_\lambda h_\lambda =-(\Lambda h)_\lambda, \qquad
    \lambda\partial_\lambda h_{\underline{\lambda}} =-(\Lambda_{-1}h)_{\underline{\lambda}}.
\end{equation*}
Direct changes of variables give
\begin{equation}\label{eq:norm scaling}
    \begin{gathered}
        \|h_\lambda\|_{L^2}=\lambda\|h\|_{L^2}, \quad \|h_\lambda\|_{\dot H^1}=\|h\|_{\dot H^1}, \quad \|h_\lambda\|_{\dot H^2}=\lambda^{-1}\|h\|_{\dot H^2},\\
        \|r^ah_{\underline{\lambda}}\|_{L^2} =\lambda^{a-1}\|r^ah\|_{L^2}, \quad (h_{1,\lambda},h_{2,\underline{\lambda}})=(h_1,h_2).
    \end{gathered}
\end{equation}

For the unified formulation of the two models, define 
\begin{equation}\label{eq:Q f}
    Q(r)=
    \begin{cases}
        2\arctan(r^2),&\text{for \eqref{eq:hmhf 2}},\\
        r^2(1+\frac{1}{24}r^2)^{-2},&\text{for \eqref{eq:NLH 6}}.
    \end{cases}
\end{equation}
In both models, $W=r^{-2}Q$. We also define
\begin{equation}
    f(z)\coloneqq
    \begin{cases}
        2(2z-\sin(2z)),&\text{for \eqref{eq:hmhf 2}},\\
        |z|z,&\text{for \eqref{eq:NLH 6}}.
    \end{cases}
    \quad
    \mathfrak s(z)\coloneqq
    \begin{cases}
        \sin z,&\text{for \eqref{eq:hmhf 2}},\\
        z,&\text{for \eqref{eq:NLH 6}}.
    \end{cases}
\end{equation}

The explicit profiles satisfy
\begin{equation}\label{eq:unified kernel formulas}
    \{\Lambda W,\Lambda_{-1}\Lambda W\}
    =\begin{cases}
        \{\frac4{1+r^4},\frac{16}{(1+r^4)^2}\},&\text{for \eqref{eq:hmhf 2}},\\[2mm]
        \{\frac{2(1-\frac{1}{24}r^2)}{(1+\frac{1}{24}r^2)^3},\frac{8(1-\frac{1}{12}r^2)}{(1+\frac{1}{24}r^2)^4}\},&\text{for \eqref{eq:NLH 6}}.
    \end{cases}
\end{equation}
In particular,
\begin{equation*}
    \Lambda W\in L^2, \qquad r\Lambda W\notin L^2, \qquad r\Lambda_{-1}\Lambda W,r^2\Lambda_{-1}\Lambda W\in L^2.
\end{equation*}
For \eqref{eq:hmhf 2}, the explicit formula for $Q$ yields
\begin{align*}
    Q(r)&=2r^2+O(r^6), &Q_r(r)&=4r+O(r^5) &&\text{as }r\to0, \\
    Q(r)&=\pi-2r^{-2}+O(r^{-6}), &Q_r(r)&=4r^{-3}+O(r^{-7}) &&\text{as }r\to\infty.
\end{align*}
Define
\begin{equation}\label{eq:notation rho c ast}
    \rho(r)\coloneqq\frac{1}{(1+r^2)^2},\qquad c_*\coloneqq\|\Lambda W\|_{L^2}^2,\qquad
    \varsigma\coloneqq
    \begin{cases}
        1,&\text{for \eqref{eq:hmhf 2}},\\
        -1,&\text{for \eqref{eq:NLH 6}}.
    \end{cases}
\end{equation}

For signs $\biota=(\iota_1,\ldots,\iota_N)\in\{-1,1\}^N$ and scales $\blambda=(\lambda_1,\ldots,\lambda_N)$, define
\begin{equation*}
    \mu_j\coloneqq\frac{\lambda_j}{\lambda_{j-1}},\quad 2\leq j\leq N,\qquad
    \alpha\coloneqq
    \begin{cases}
        \max_{2\leq j\leq N}\mu_j,&N\geq2,\\
        0,&N=1.
    \end{cases}
\end{equation*}
For $a>0$, define the set of scale vectors
\begin{equation*}
    \calP_N(a)\coloneqq\{\blambda\in(0,\infty)^N:\alpha<a\}.
\end{equation*}
In particular, $\calP_1(a)=(0,\infty)$. For either model, define
\begin{equation}\label{eq:U explicit}
    U(\biota,\blambda)\coloneqq\sum_{j=1}^N\iota_jW_{\lambda_j},\qquad P\coloneqq r^2U.
\end{equation}

The nonlinear operator from the equation and its linearization at $W_\lambda$ are
\begin{equation}\label{eq:operators}
    \begin{gathered}
        \calT(h)\coloneqq\Delta h+r^{-4}f(r^2h),\\
        Hh\coloneqq-\Delta h-r^{-2}f'(Q)h,\qquad
        H_\lambda h\coloneqq-\Delta h-r^{-2}f'(Q_{[\lambda]})h.
    \end{gathered}
\end{equation}
We also have
\begin{equation}
    H_\lambda h_\lambda=(Hh)_{\underline{\lambda}},  
    \qquad
    H(\Lambda W)=0,\qquad H_\lambda(\Lambda W)_\lambda=0. \label{eq:H scaling}
\end{equation}
The decoupled linearized operator is
\begin{equation}\label{eq:decoupled operator}
    H_{\boldsymbol\lambda}\coloneqq-\Delta-r^{-2}\sum_{j=1}^Nf'(\iota_jQ_{[\lambda_j]}).
\end{equation}
Moreover, the linearized operator around the bubble tree $U$ is
\begin{equation*}
    H_Uh\coloneqq -\Delta h-r^{-2}f'(r^2U)h =-\Delta h-r^{-2}f'(P)h.
\end{equation*}

The radial test function $Z$ satisfies
\begin{equation}\label{eq:Z normalization}
    Z\in C_c^\infty(\bbR^6),\quad (Z,\Lambda W)=1,\quad \operatorname{supp}Z\subset[R_0^{-1},R_0],\quad R_0>10.
\end{equation}
For \eqref{eq:NLH 6}, $\calY$ is the positive radial eigenfunction of $H$ with negative eigenvalue $-e_0$, normalized by
\begin{equation}\label{eq:NLH unstable mode}
    H\calY=-e_0\calY,\qquad e_0>0,\qquad \|\calY\|_{L^2}=1.
\end{equation}

For $2\leq j\leq N$, the interaction quantities are
\begin{equation*}
    d_j=\frac{\lambda_j}{\lambda_{j-1}^2},\quad \calD=\sum_{j=2}^Nd_j^2,\quad
    \kappa=-\frac{(\Lambda W,r^{-2}f'(Q)W(0))}{\|\Lambda W\|_{L^2}^2}
    =\begin{cases}
        -\frac{16}{\pi},&\text{for \eqref{eq:hmhf 2}},\\
        \frac54,&\text{for \eqref{eq:NLH 6}}.
    \end{cases}
\end{equation*}
The principal interaction $\calI$ and the nonlinear remainder at $U$ are given by
\begin{equation*}
    \calI=\sum_{j=2}^N\iota_{j-1}\frac{\kappa}{\lambda_{j-1}^2}(\Lambda W)_{\lambda_j},\qquad \NL_U(h)=r^{-4}[f(P+r^2h)-f(P)-f'(P)r^2h].
\end{equation*}


\subsection{Preliminary estimates}
We collect some elementary estimates.
\begin{lem}
	For radial $h\in C_c^\infty((0,\infty))$ one has
	\begin{equation}\label{eq:sobolev}
		\begin{aligned}
			\|h\|_{L^3}&\lesssim\|h\|_{\dot H^1}, &\|h\|_{L^6}&\lesssim\|h\|_{\dot H^2}, \\
			\|r^{-1}h\|_{L^2}&\leq\tfrac12\|h\|_{\dot H^1}, &\|r^{-2}h\|_{L^2}&\lesssim\|h\|_{\dot H^2}.
		\end{aligned}
	\end{equation}
\end{lem}

In the lemma below, we use the notation in \eqref{eq:notation rho c ast}.
\begin{lem}
	Let $a,b>0$, $0<\mu<1$, and let $\psi\in L^1\cap L^\infty$ be radial. Then
	\begin{align}
		(\rho_{\underline b},\rho_a) &\lesssim \min\{1,a^2/b^2\}, \label{eq:rho overlap}\\
		|(\psi_{\underline b},(\Lambda W)_a)| &\lesssim_\psi \min\{a^2/b^2,b^2/a^2\}, \label{eq:test overlap}\\
		\big|((\Lambda W)_{\underline\mu},\Lambda W)-(1+\varsigma)c_*\big| &\lesssim\mu^2(1+|\log\mu|). \label{eq:endpoint kernel overlap}
	\end{align}
\end{lem}

\begin{proof}
	By symmetry, to prove \eqref{eq:rho overlap} it suffices to estimate $(\rho_b,\rho_a)$ for $a\leq b$. Splitting at $a$ and $b$, we obtain
	\begin{equation*}
		(\rho_b,\rho_a)\lesssim a^{-2}b^{-2}\int_0^a r^5dr+a^2b^{-2}\int_a^b r dr+a^2b^2\int_b^\infty r^{-3}dr\lesssim a^2.
	\end{equation*}
	For \eqref{eq:test overlap}, use $|\Lambda W(r)|\lesssim\min\{1,r^{-4}\}$ to obtain
	\begin{equation*}
		|(\psi_{\underline b},(\Lambda W)_a)|
		\lesssim_\psi\min\{(a/b)^2,(b/a)^2\}.
	\end{equation*}

	For \eqref{eq:endpoint kernel overlap}, put $\ell_*\coloneqq\lim_{r\to\infty}r^4\Lambda W(r)$. The explicit formulas in \eqref{eq:unified kernel formulas} give
	\begin{equation*}
		|\Lambda W(r)-\ell_*r^{-4}|\lesssim r^{-6}, \qquad r\geq1,
	\end{equation*}
	and $\Lambda W$ is bounded on $0<r\leq1$. The same formulas and $\Lambda W=2W+rW_r$ give
	\begin{equation*}
		\ell_*\int_0^\infty\Lambda W(r)rdr=\ell_*[r^2W(r)]_0^\infty=(1+\varsigma)c_*.
	\end{equation*}
	Therefore, we have
	\begin{equation*}
		((\Lambda W)_{\underline\mu},\Lambda W)-(1+\varsigma)c_*=\int_0^\infty\Lambda W(r)\{(\Lambda W)_{\underline\mu}(r)-\ell_*r^{-4}\}r^5dr.
	\end{equation*}
	Splitting the integral at $r=\mu$ and $r=1$, we obtain \eqref{eq:endpoint kernel overlap}:
	\begin{equation*}
		\big|((\Lambda W)_{\underline\mu},\Lambda W)-(1+\varsigma)c_*\big|
		\lesssim\mu^2+\mu^2\int_\mu^1\frac{dr}{r}+\mu^2\int_1^\infty\frac{dr}{r^5}\lesssim\mu^2(1+|\log\mu|). \qedhere
	\end{equation*}
\end{proof}

For the model nonlinearity $f$ in \eqref{eq:Q f}, the interaction estimate is expressed using $\mathfrak s(z)=\sin z$ for \eqref{eq:hmhf 2} and $\mathfrak s(z)=z$ for \eqref{eq:NLH 6}.

\begin{lem}[Model nonlinearities]
	For both models and all $a,b\in\bbR$,
	\begin{align}
		|f'(a)-f'(b)|&\lesssim|a-b|, \label{eq:fprime}\\
		|f(a+b)-f(a)-f'(a)b|&\lesssim|b|^2, \label{eq:f Taylor}\\
		|f(a+b)-f(a)-f(b)|&\lesssim|\mathfrak s(a)||\mathfrak s(b)|. \label{eq:f interaction}
	\end{align}
\end{lem}

\begin{proof}
	For \eqref{eq:hmhf 2}, $f'(z)=4-4\cos(2z)$ and $f''(z)=8\sin(2z)$. Thus \eqref{eq:fprime} holds and Taylor's formula proves \eqref{eq:f Taylor}. The trigonometric addition formulas give \eqref{eq:f interaction}.

	For \eqref{eq:NLH 6}, $f'(z)=2|z|$, so $||a|-|b||\leq|a-b|$ yields \eqref{eq:fprime}, and the fundamental theorem of calculus proves \eqref{eq:f Taylor}. $f(z)=|z|z$ and $||a+b|-|a||\leq|b|$ give \eqref{eq:f interaction}. 
\end{proof}

\subsection{Coercivity}

For radial functions, the one-bubble kernel relations are
\begin{equation}\label{eq:H kernel space}
	\ker_{\dot H^2}H =\operatorname{span}\{\Lambda W\},\quad \ker_{\dot H^2}H_\lambda =\operatorname{span}\{(\Lambda W)_\lambda\}.
\end{equation}

We introduce coercivity estimates for the linearized operator $H$ in both models.

\begin{lem}[Coercivity estimate for $H$ in $\dot H^2$]\label{lem:coercivity}
	Let $Z$ be fixed by \eqref{eq:Z normalization}. Then, for any radial $h\in\dot H^2$ with $(Z,h)=0$, we have
	\begin{equation}\label{eq:coercivity 1}
		\|Hh\|_{L^2}\sim\|h\|_{\dot H^2}.
	\end{equation}
\end{lem}

For bubble trees, $H_{\boldsymbol\lambda}$ denotes the decoupled linearized operator in \eqref{eq:decoupled operator}. For $2\leq j\leq N$, $\mu_j=\lambda_j/\lambda_{j-1}$ is the ratio of neighboring scales; $\alpha=\max_{2\leq j\leq N}\mu_j$, with $\alpha=0$ when $N=1$. 

\begin{lem}[Coercivity estimate for $H_{\boldsymbol\lambda}$ in $\dot H^2$]\label{lem:coercivity N}
	Let $Z$ be fixed by \eqref{eq:Z normalization}. For any $N\geq1$, there exists $\alpha_{c}>0$ such that, for all $\blambda\in\calP_N(\alpha_{c})$ and radial $h\in\dot H^2$ satisfying $(Z_{\lambda_j},h)=0$ for $1\leq j\leq N$, we have
	\begin{equation}\label{eq:coercivity N}
		\|H_{\boldsymbol\lambda}h\|_{L^2} \sim\|h\|_{\dot H^2}.
	\end{equation}
\end{lem}
For Lemma~\ref{lem:coercivity}, in the \eqref{eq:NLH 6} case, although \cite[Lemma~2.3]{CollotMerleRaphael2017CMP} is stated for $d\geq7$, the proof in \cite[Appendix~B]{CollotMerleRaphael2017CMP} applies to the $\dot H^2$ estimate for $d\geq5$; see also \cite[Lemma~2.5]{Harada2026CVPDE} for the corresponding estimate in the present dimension $d=6$. In the \eqref{eq:hmhf 2} case, see \cite[Appendix~B]{RaphaelRodnianski2012}. We therefore omit the proof of Lemma~\ref{lem:coercivity} and only sketch the proof of Lemma~\ref{lem:coercivity N} in Appendix~\ref{app:coercivity}, following \cite[Lemma~2.7]{KimMerle2025CPAM}.

For the rest of this subsection, we restrict to \eqref{eq:NLH 6} and record the $\dot H^1$ coercivity estimates needed for the construction in Section~\ref{sec:NLH construct}. The eigenvalue $-e_0$ in \eqref{eq:NLH unstable mode} is the unique negative eigenvalue of $H$ in the radial class, and the corresponding eigenfunction $\calY$ is smooth and exponentially decaying; see \cite[Proposition~2.2]{CollotMerleRaphael2017CMP}. By \eqref{eq:H scaling}, \eqref{eq:NLH unstable mode}, and the self-adjointness of $H$, we have $(\calY,\Lambda W)=0$. By scaling, 
\begin{equation*}
	H_\lambda\calY_{\underline\lambda}=-{e_0}{\lambda^{-2}}\calY_{\underline\lambda},\qquad (\calY_{\underline\lambda},\calY_\lambda)=1,\qquad \|\calY_{\underline\lambda}\|_{\dot H^{-1}}=\|\calY\|_{\dot H^{-1}}.
\end{equation*}

\begin{lem}[Coercivity estimate for $H$ in $\dot H^1$ for \eqref{eq:NLH 6}]\label{lem:NLH one bubble H1 coercivity}
	Let $Z$ be as in \eqref{eq:Z normalization}. Then, for every radial $h\in\dot H^1$,
	\begin{equation}\label{eq:NLH one bubble H1 coercivity}
		\|h\|_{\dot H^1}^2\lesssim(h,Hh)+C((\calY,h)^2+(Z,h)^2 ).
	\end{equation}
	Here $C>0$ is sufficiently large.
\end{lem}

We also need coercivity for the linearization $H_U=-\Delta-r^{-2}f'(r^2U)$ at the sum $U=\sum_{j=1}^N\iota_jW_{\lambda_j}$, with $P=r^2U$. 

\begin{lem}[Coercivity estimate for $H_U$ in $\dot H^1$ for \eqref{eq:NLH 6}]\label{lem:NLH shooting coercivity}
	Fix $N\geq1$ and signs $\iota_j=(-1)^{j-1}$ for $1\leq j\leq N$, and let $Z$ be fixed by \eqref{eq:Z normalization}. There exists $\alpha_c>0$ such that, for every $\blambda\in\calP_N(\alpha_c)$, every radial $h\in\dot H^1$ satisfying
	\begin{equation*}
		(Z_{\underline{\lambda_j}},h)=0,\qquad 1\leq j\leq N,
	\end{equation*}
	obeys
	\begin{equation}\label{eq:NLH shooting H1 coercivity}
		\|h\|_{\dot H^1}^2\lesssim(h,H_Uh)+\sum_{j=1}^N(\calY_{\underline{\lambda_j}},h)^2.
	\end{equation}
\end{lem}
For the proof of Lemma~\ref{lem:NLH one bubble H1 coercivity}, see \cite[Appendix~B]{CollotMerleRaphael2017CMP}; see also \cite[Lemma~2.5]{Harada2026CVPDE}. Thus, we omit its proof and only sketch the proof of Lemma~\ref{lem:NLH shooting coercivity} in Appendix~\ref{app:coercivity}.

\section[Decomposition and spacetime estimate]{Decomposition and spacetime estimate}

In this section, we construct a decomposition of the solution into a sum of bubbles and a radiation term satisfying an $L^2_t$ bound: 
\begin{equation*}
	u(t)=\sum_{j=1}^N\iota_jW_{\lambda_j(t)}+g(t), \qquad \int_{t_0^*}^{T_+}\|g(t)\|_{\dot H^2}^2dt<\infty
\end{equation*}
for some $t_0^*<T_+$. The scales $\lambda_j(t)$ are fixed by imposing orthogonality conditions on $g(t)$ and generally differ from those supplied by soliton resolution. With this choice, the scales are $C^1$, and the decomposition retains the scale decoupling and qualitative smallness given by soliton resolution. The $L^2_t$ bound is crucial when we integrate the modulation equations in the next section, since it ensures that the quadratic error terms in $g(t)$ are integrable in time. We also obtain time integrability estimates for the interactions between bubbles.

In this decomposition, we use the sum of solitons as the profile, without introducing a modified profile. In higher dimensions, the elliptic corrections introduced by K.~Kim--Merle \cite[Section~3]{KimMerle2025CPAM} compensate for the larger interactions between bubbles; see also \cite{DengSunWei2025Duke}. In our $D=2$ setting, the interactions generated by the sum of solitons already satisfy the $L^2$ bounds needed to combine coercivity with the energy identity. We therefore establish the decomposition and spacetime estimate needed to follow \cite{KimMerle2025CPAM}, using the sum of solitons in place of a modified profile.

\subsection[Decomposition]{Decomposition}\label{subsec:decomposition}
We construct a decomposition with $C^1$ scales determined by orthogonality conditions and establish the corresponding scale decoupling and qualitative smallness of the radiation. As a starting point, we use the soliton resolution results, Propositions~\ref{thm:HMHF sol resol} and~\ref{thm:NLH sol resol}, which provide decompositions with continuous scales. Let $u$ be a solution whose soliton resolution contains $N\geq1$ bubbles. In the case of \eqref{eq:hmhf 2}, by the invariance of the equation, we may make the replacement
\begin{equation}\label{eq:v shift}
	v\mapsto v-\ell\pi
\end{equation}
and retain the notation $v$. In the global case, take the asymptotic profile to be zero and interpret $T_+-t=\infty$. Write $\wt{\blambda}(t)\coloneqq(\wt\lambda_1(t),\ldots,\wt\lambda_N(t))$. For \eqref{eq:hmhf 2}, we have
\begin{equation*}
	v(t)=\sum_{j=1}^N\iota_jQ_{[\wt\lambda_j(t)]}+v^*+\wt\eps(t).
\end{equation*}
Then, set $u=r^{-2}v$, $\wt\eps_6=r^{-2}\wt\eps$, and $z^*=r^{-2}v^*$. For \eqref{eq:NLH 6}, retain the notation $u$ for the solution and write $z^*$ and $\wt\eps_6$ for the asymptotic profile and remainder in \eqref{eq:NLH finite decomposition} or \eqref{eq:NLH global decomposition}. Thus, in both cases,
\begin{equation}\label{eq:soliton resolution}
	u(t)=\sum_{j=1}^N\iota_jW_{\wt\lambda_j(t)}+z^*+\wt\eps_6(t),
\end{equation}
where $\|\wt\eps_6(t)\|_{\dot H^1}\to0$, and $z^*=0$ when $T_+=\infty$. Moreover, the scales satisfy
\begin{equation}\label{eq:soliton resolution scales}
	\sum_{j=2}^N\frac{\wt\lambda_j(t)}{\wt\lambda_{j-1}(t)}
	+\frac{\wt\lambda_1(t)}{\min\{\sqrt{T_+-t},\sqrt t\}}\to0 \qquad\text{as }t\to T_+.
\end{equation}
Since the scales $\wt\lambda_j$ are only continuous, we choose nearby $C^1$ scales by imposing orthogonality conditions. Fix a sufficiently small $\alpha_0>0$.

\begin{prop}[Decomposition]\label{prop:decomposition solution}
	For a solution satisfying the assumptions of Theorem~\ref{thm:hmhf main} for \eqref{eq:hmhf 2} or Theorem~\ref{thm:nlh main} for \eqref{eq:NLH 6}, there exist $0<t_{\mathrm d}<T_+$ and $C^1$ scales
	\begin{equation*}
		\blambda=(\lambda_1,\ldots,\lambda_N): [t_{\mathrm d},T_+)\to\calP_N(\alpha_0)
	\end{equation*}
	such that the following hold on $[t_{\mathrm d},T_+)$.
	\begin{itemize}
		\item \emph{(Decomposition)} The solution satisfies the decomposition
		\begin{equation}\label{eq:refined decomposition}
			u(t)=U(\biota,\blambda(t))+g(t)=U(t)+g(t)
		\end{equation}
		with the orthogonality conditions
		\begin{equation}\label{eq:orthogonality}
			 (Z_{\underline{\lambda_j}},g(t))=0, \qquad 1\leq j\leq N.
		\end{equation}

		\item (Scale decoupling) The scales satisfy 
		\begin{equation}\label{eq:lambda1 zero}
			\sum_{j=2}^N\mu_j(t)+ \frac{\lambda_1(t)}{\min\{\sqrt{T_+-t},\sqrt t\}}\to0 \qquad\text{as }t\to T_+.
		\end{equation}
		\item (Smallness of remainder) We have
		\begin{equation}
			\|\langle r/\lambda_1(t)\rangle^{-1}g(t)\|_{\dot H^1} \to0. \label{eq:weighted qualitative smallness}
		\end{equation}
		If $T_+<\infty$, then there is a function $0<\delta_*(r_0)\to0$ as $r_0\to0$ such that
		\begin{equation}\label{eq:local qualitative smallness}
			\|\chi_{4r_0}g(t)\|_{\dot H^1} \leq\delta_*(r_0)+o_{t}(1).
		\end{equation}
		If $T_+=\infty$, then
		\begin{equation}\label{eq:global H1 smallness}
			\|g(t)\|_{\dot H^1}\to0.
		\end{equation}
		
	\end{itemize}
\end{prop}
We first prove a local decomposition lemma that fixes the scales near a sum of bubbles.
For $\delta>0$ and $\boldsymbol\nu\in(0,\infty)^N$, define
\begin{gather*}
	B_\delta(U(\biota,\boldsymbol\nu)) \coloneqq \{w:\ w-U(\biota,\boldsymbol\nu)\in\dot H^1,\ \|w-U(\biota,\boldsymbol\nu)\|_{\dot H^1}<\delta\},\\
	B_\delta(\boldsymbol\nu) \coloneqq \{\blambda\in(0,\infty)^N: \operatorname{dist}(\boldsymbol\nu,\blambda)<\delta\}, \qquad
	\operatorname{dist}(\boldsymbol\nu,\blambda) \coloneqq \max_{1\leq j\leq N}|\log\tfrac{\nu_j}{\lambda_j}|.
\end{gather*}

\begin{lem}[Decomposition near a multi-bubble]\label{lem:decomposition near multi bubble}
	There are constants $\alpha_{\mathrm d}\in(0,\alpha_0)$, $\delta_{\mathrm d}>0$, and $\eta_{\mathrm d}>0$ with the following property. For any $\biota\in\{-1,1\}^N$ and $\boldsymbol\nu\in\calP_N(\alpha_{\mathrm d})$, there is a $C^1$ map
	\begin{equation*}
		\mathbf G^{\biota,\boldsymbol\nu}: B_{\delta_{\mathrm d}}(U(\biota,\boldsymbol\nu)) \to B_{\eta_{\mathrm d}}(\boldsymbol\nu), \qquad w\mapsto\blambda=\mathbf G^{\biota,\boldsymbol\nu}(w)
	\end{equation*}
	such that:
	\begin{itemize}
		\item For all $w\in B_{\delta_{\mathrm d}}(U(\biota,\boldsymbol\nu))$, the scales $\blambda=\mathbf G^{\biota,\boldsymbol\nu}(w)$ satisfy
		\begin{equation*}
			(\iota_kZ_{\underline{\lambda_k}},w-U(\biota,\blambda))=0, \qquad 1\leq k\leq N.
		\end{equation*}
		
		\item If $w\in B_{\delta_{\mathrm d}}(U(\biota,\boldsymbol\nu))$ and $\blambda\in B_{\eta_{\mathrm d}}(\boldsymbol\nu)$ satisfy
		\begin{equation*}
			(\iota_kZ_{\underline{\lambda_k}}, w-U(\biota,\blambda))=0, \qquad 1\leq k\leq N,
		\end{equation*}
		then $\blambda=\mathbf G^{\biota,\boldsymbol\nu}(w)$.
		
		\item For any $\epsilon>0$, there are $\alpha_\epsilon\in(0,\alpha_{\mathrm d}]$ and $\delta_\epsilon\in(0,\delta_{\mathrm d}]$ such that
		\begin{equation*}
			\operatorname{dist} (\boldsymbol\nu, \mathbf G^{\biota,\boldsymbol\nu}(w) ) \leq\epsilon
		\end{equation*}
		whenever $\boldsymbol\nu\in\calP_N(\alpha_\epsilon)$ and $w\in B_{\delta_\epsilon}(U(\biota,\boldsymbol\nu))$.
	\end{itemize}
\end{lem}
\begin{proof}
	We fix $\biota\in\{-1,1\}^N$ and consider the function $F^{\biota}=(F_1^{\biota},\ldots,F_N^{\biota})$ whose components $F_k^{\biota}$, $k=1,\ldots,N$, are defined by
	\begin{equation}\label{eq:orthogonality map}
		F_k^{\biota}(\blambda;w)\coloneqq(\iota_kZ_{\underline{\lambda_k}},w-U(\biota,\blambda)).
	\end{equation}
	Now, we fix $\alpha_{\mathrm d}\in(0,\alpha_0)$, $\boldsymbol\nu\in\calP_N(\alpha_{\mathrm d})$, and $w\in B_{\delta_{\mathrm d}}(U(\biota,\boldsymbol\nu))$. The neighborhood $B_{\eta_{\mathrm d}}(\boldsymbol\nu)$ is contained in $\calP_N(\alpha_0)$ provided $e^{2\eta_{\mathrm d}}\alpha_{\mathrm d}<\alpha_0$. First, we have
	\begin{equation*}
		F_k^{\biota}(\boldsymbol\nu;U(\biota,\boldsymbol\nu))=0.
	\end{equation*}
	Next, by $(Z,\Lambda W)=1$, \eqref{eq:sobolev}, and \eqref{eq:test overlap}, we have, for $\blambda\in B_{\eta_{\mathrm d}}(\boldsymbol\nu)$,
	\begin{equation}\label{eq:IFT derivative estimate}
		\begin{aligned}
			\lambda_j\partial_{\lambda_j}F_k^{\biota}(\blambda;w)
			&=(\iota_kZ_{\underline{\lambda_k}},\iota_j(\Lambda W)_{\lambda_j})
			-\delta_{jk}(\iota_k(\Lambda_{-1}Z)_{\underline{\lambda_k}},w-U)\\
			&=\delta_{jk}+O(\alpha_{\mathrm d}^2)+O(\|w-U\|_{\dot H^1})
			=\delta_{jk}+O(\alpha_{\mathrm d}^2+\delta_{\mathrm d}+\eta_{\mathrm d}).
		\end{aligned}
	\end{equation}
	For every fixed radial $\psi\in C_c^\infty(\bbR^6)$ and $h\in\dot H^1$, we have by \eqref{eq:sobolev}
	\begin{equation}\label{eq:compact test H1}
		|(\psi_{\underline\lambda},h)|\leq\|r^{-1}h\|_{L^2}\|r\psi_{\underline\lambda}\|_{L^2}\lesssim_\psi\|h\|_{\dot H^1}.
	\end{equation}
	We also note that $F_k^{\biota}$ is $C^1$ in $w$ with the Lipschitz bound
	\begin{equation}\label{eq:IFT lipschitz}
		|F_k^{\biota}(\blambda;w)-F_k^{\biota}(\blambda;w')|
		=|(\iota_kZ_{\underline{\lambda_k}},w-w')|
		\lesssim\|w-w'\|_{\dot H^1}.
	\end{equation}
	From \eqref{eq:IFT derivative estimate} and \eqref{eq:IFT lipschitz}, the implicit function theorem yields a $C^1$ map $\mathbf G^{\biota,\boldsymbol\nu}:B_{\delta_{\mathrm d}}(U(\biota,\boldsymbol\nu))\to B_{\eta_{\mathrm d}}(\boldsymbol\nu)$ satisfying $F^{\biota}(\mathbf G^{\biota,\boldsymbol\nu}(w);w)=0$, with $\alpha_{\mathrm d},\eta_{\mathrm d}$ sufficiently small and $\delta_{\mathrm d}$ sufficiently small relative to $\eta_{\mathrm d}$. These choices are uniform in $\biota$ and $\boldsymbol\nu$. This proves the first statement. 
	
	For fixed $w$ and $\blambda,\blambda'\in B_{\eta_{\mathrm d}}(\boldsymbol\nu)$, we integrate the derivatives along the line segment joining $\log\blambda$ and $\log\blambda'$. By \eqref{eq:IFT derivative estimate}, for sufficiently small $\alpha_{\mathrm d},\delta_{\mathrm d},\eta_{\mathrm d}$, we get
	\begin{equation}\label{eq:IFT inverse estimate}
		\operatorname{dist}(\blambda,\blambda')\lesssim|F^{\biota}(\blambda;w)-F^{\biota}(\blambda';w)|.
	\end{equation}
	In particular, if $F^{\biota}(\blambda;w)=F^{\biota}(\blambda';w)=0$, then $\blambda=\blambda'$. This proves the second statement.

	For $\blambda=\boldsymbol\nu$ and $\blambda'=\mathbf G^{\biota,\boldsymbol\nu}(w)$, the estimates \eqref{eq:IFT inverse estimate} and \eqref{eq:IFT lipschitz} imply
	\begin{equation*}
		\operatorname{dist}(\boldsymbol\nu,\mathbf G^{\biota,\boldsymbol\nu}(w) )
		\lesssim|F^{\biota}(\boldsymbol\nu;w)|
		\lesssim\|w-U(\biota,\boldsymbol\nu)\|_{\dot H^1}.
	\end{equation*}
	For any $\epsilon>0$, the last statement follows with $\alpha_\epsilon=\alpha_{\mathrm d}$ and $\delta_\epsilon\in(0,\delta_{\mathrm d}]$ sufficiently small.
\end{proof}

\begin{proof}[Proof of Proposition~\ref{prop:decomposition solution}]
	We first choose a cutoff for the asymptotic profile. In the finite-time blow-up case,
	\begin{equation}\label{eq:body local H1}
		\|\chi_Rz^*\|_{\dot H^1}\to0 \quad \text{as} \quad R\to0.
	\end{equation}
	In fact, for \eqref{eq:hmhf 2}, $z^*=r^{-2}v^*$ and $v^*(0)=0$, so $|v^*|\lesssim|\sin v^*|$ near the origin and $\|\chi_Rz^*\|_{\dot H^1}^2\lesssim\int_0^{2R}(|v_r^*|^2+|v^*|^2/r^2)rdr\to0$ by finite energy. For \eqref{eq:NLH 6}, we obtain \eqref{eq:body local H1} from \eqref{eq:sobolev} and absolute continuity. Hence, we choose $r_*>0$ so that $\|\chi_{r_*}z^*\|_{\dot H^1}<\delta_{\mathrm d}/4$. In the global case, we use $r_*=\infty$ and $\chi_{\infty}=1$.
	
	We write $\td u\coloneqq u-(1-\chi_{r_*})z^*$. By \eqref{eq:soliton resolution} and \eqref{eq:soliton resolution scales}, we choose $t_{\mathrm d}$ sufficiently close to $T_+$ so that, for all $t\in[t_{\mathrm d},T_+)$,
	\begin{equation}\label{eq:localized pure decomposition}
		\td u(t)-U(\biota,\wt{\blambda}(t))=\chi_{r_*}z^*+\wt\eps_6(t), \qquad \|\td u(t)-U(\biota,\wt{\blambda}(t))\|_{\dot H^1}<\delta_{\mathrm d}/2,
	\end{equation}
	and $\wt{\blambda}(t)\in\calP_N(\alpha_{\mathrm d})$. In the finite-time blow-up case, we take $t_{\mathrm d}$ close enough to $T_+$ that $R_0e^{\eta_{\mathrm d}}\wt\lambda_1(t)<r_*$ also holds throughout this interval. By Lemma~\ref{lem:decomposition near multi bubble}, we define
	\begin{equation*}
		\blambda(t)\coloneqq\mathbf G^{\biota,\wt{\blambda}(t)}(\td u(t)), \qquad F^{\biota}(\blambda(t);\td u(t))=0,
	\end{equation*}
	where $F^{\biota}$ is defined by \eqref{eq:orthogonality map}.
	For every fixed radial $\psi\in C_c^\infty(\bbR^6)$, we have
	\begin{equation}\label{eq:asym inner product}
		|(\psi_{\underline{\wt\lambda_k(t)}},z^*)|\lesssim_\psi\|r^{-1}z^*\|_{L^2(r\lesssim\wt\lambda_k(t))}\to0.
	\end{equation}
	Indeed, $\wt\lambda_k(t)\to0$ in the finite-time blow-up case, while $z^*=0$ in the global case. By \eqref{eq:soliton resolution}, \eqref{eq:compact test H1}, and \eqref{eq:asym inner product}, we obtain $F^{\biota}(\wt{\blambda}(t);\td u(t))\to0$. By \eqref{eq:IFT inverse estimate}, we obtain
	\begin{equation}\label{eq:JL closeness}
		\operatorname{dist}(\blambda(t),\wt{\blambda}(t))\lesssim|F^{\biota}(\wt{\blambda}(t);\td u(t))|\to0.
	\end{equation}
	We take $t_{\mathrm d}$ closer to $T_+$ so that the distance in \eqref{eq:JL closeness} is less than $\eta_{\mathrm d}/4$. For each $t_0\in[t_{\mathrm d},T_+)$, positive-time parabolic regularity and \eqref{eq:IFT derivative estimate} allow us to apply the implicit function theorem to $F^{\biota}(\blambda;\td u(t))=0$ near $(\blambda(t_0),t_0)$ and obtain local $C^1$ scales. For $t$ near $t_0$, continuity of $\td u$ and $\wt{\blambda}$ ensures that $\td u(t)\in B_{\delta_{\mathrm d}}(U(\biota,\wt{\blambda}(t_0)))$ and that both the local scales and $\blambda(t)$ belong to $B_{\eta_{\mathrm d}}(\wt{\blambda}(t_0))$. The uniqueness assertion of Lemma~\ref{lem:decomposition near multi bubble} identifies them, proving $\blambda\in C^1([t_{\mathrm d},T_+))$.
	
	In the finite-time blow-up case, $R_0\lambda_1(t)<r_*$, so $(1-\chi_{r_*})z^*$ vanishes on every $\operatorname{supp}Z_{\underline{\lambda_k(t)}}$. In the global case, $\td u=u$. Thus $g\coloneqq u-U(\biota,\blambda)$ satisfies \eqref{eq:refined decomposition} and \eqref{eq:orthogonality}. By \eqref{eq:soliton resolution}, we have
	\begin{equation}\label{eq:g qualitative identity}
		g=z^*+\eps_6, \qquad \eps_6\coloneqq\wt\eps_6+\sum_{j=1}^N\iota_j(W_{\wt\lambda_j}-W_{\lambda_j}).
	\end{equation}
	Since $\lambda\partial_\lambda W_\lambda=-(\Lambda W)_\lambda$ and $\Lambda W\in\dot H^1$, we obtain from \eqref{eq:JL closeness}
	\begin{equation}\label{eq:q H1 small}
		\|\eps_6(t)\|_{\dot H^1}\to0.
	\end{equation}
	This proves \eqref{eq:global H1 smallness} when $T_+=\infty$. By \eqref{eq:soliton resolution scales} and \eqref{eq:JL closeness}, we also obtain \eqref{eq:lambda1 zero}. In the finite-time blow-up case, we have by \eqref{eq:sobolev}
	\begin{equation*}
		\|\chi_{4r_0}g(t)\|_{\dot H^1}\lesssim\|\chi_{4r_0}z^*\|_{\dot H^1}+\|\eps_6(t)\|_{\dot H^1}.
	\end{equation*}
	Together with \eqref{eq:body local H1} and \eqref{eq:q H1 small}, this proves \eqref{eq:local qualitative smallness}.
	
	It remains to prove \eqref{eq:weighted qualitative smallness}. We write $w_\lambda(r)\coloneqq\langle r/\lambda\rangle^{-1}$. Multiplication by $w_\lambda$ is uniformly bounded on $\dot H^1$ by $|r\partial_r w_\lambda|\lesssim w_\lambda\leq1$ and \eqref{eq:sobolev}. For \eqref{eq:hmhf 2}, we have
	\begin{align*}
		\bigg|\int_{v^*(1)}^{v^*(r)}|\sin y|dy\bigg|
		&=\bigg|\int_1^r|\sin v^*(s)|\partial_s v^*(s)ds\bigg|\\
		&\leq\bigg(\int_0^\infty|v_r^*|^2r\,dr\bigg)^{1/2}\bigg(\int_0^\infty\frac{\sin^2v^*}{r}\,dr\bigg)^{1/2}.
	\end{align*}
	Since $|\int_0^s|\sin y|dy|\to\infty$ as $|s|\to\infty$, $v^*$ is bounded. Direct differentiation gives
	\begin{equation}\label{eq:weighted body identity}
		|\partial_r(w_\lambda z^*)|^2r^5\lesssim w_\lambda^2|v_r^*|^2r+|r\partial_r w_\lambda-2w_\lambda|^2|v^*|^2 r^{-1}.
	\end{equation}
	By finite energy and $v^*(0)=0$, the integral of the right-hand side over $(0,R)$ is small uniformly in $\lambda$ for small $R>0$. On $(R,\infty)$, boundedness of $v^*$ and $w_\lambda\lesssim\lambda/r$ give $\int_R^\infty|\partial_r(w_\lambda z^*)|^2r^5dr\lesssim_R\lambda^2$. Thus, we obtain, in both models,
	\begin{equation}\label{eq:weighted body small}
		\|w_\lambda z^*\|_{\dot H^1}\to0 \qquad(\lambda\to0).
	\end{equation}
	For \eqref{eq:NLH 6}, this convergence follows directly from $z^*\in\dot H^1$ and \eqref{eq:sobolev}.
	In the finite-time blow-up case, we obtain \eqref{eq:weighted qualitative smallness} from \eqref{eq:g qualitative identity}, \eqref{eq:q H1 small}, and \eqref{eq:weighted body small}, with $\lambda=\lambda_1(t)\to0$. In the global case, we obtain \eqref{eq:weighted qualitative smallness} directly from \eqref{eq:global H1 smallness}.
\end{proof}

\subsection[Spacetime estimate]{Spacetime estimate}\label{subsec:spacetime estimates}

Our aim is to establish an $L^2_t$ bound for $\|g(t)\|_{\dot H^2}$. This bound is crucial for integrating the modulation equations in the next section, as it makes the quadratic error terms in $g$ integrable in time. We also prove the time integrability of $\calD$. We recall that $d_j=\mu_j^2/\lambda_j$ for $2\leq j\leq N$ and $\calD=\sum_{j=2}^Nd_j^2$.

\begin{prop}[Spacetime estimate]\label{prop:spacetime}
	There exists $t_0^*\in[t_{\mathrm d},T_+)$ such that
	\begin{equation}\label{eq:spacetime}
		\int_{t_0^*}^{T_+}\{\|u_t(t)\|_{L^2}^2+\calD(t)+\|g(t)\|_{\dot H^2}^2\}dt<\infty.
	\end{equation}
\end{prop}

Following \cite[Section~4.2]{KimMerle2025CPAM}, we prove an interior estimate and then, in the finite-time blow-up case, an exterior estimate. We first estimate $\calT(U)$ in $L^2$ and when paired with the scaling directions for the coercivity argument below. We also collect the profile estimates needed for the modulation analysis in Section~4.

We separate the principal interaction $\calI$ from the error $\Psi$ by writing
\begin{equation}\label{eq:profile equation}
	\calT(U)=-\calI+\Psi,
\end{equation}
where
\begin{equation}\label{eq:def calI}
	\calI\coloneqq\sum_{j=2}^N\iota_{j-1}\frac{\kappa}{\lambda_{j-1}^2}(\Lambda W)_{\lambda_j}.
\end{equation}
Here $\kappa$ is the coefficient associated with the leading interaction between the bubbles at $\lambda_{j-1}$ and $\lambda_j$ and is given by
\begin{equation}\label{eq:kappa}
	\kappa\coloneqq -\frac{(\Lambda W,r^{-2}f'(Q)W(0))}{\|\Lambda W\|_{L^2}^2}
	=\begin{cases}
		-\frac{16}{\pi},&\text{for \eqref{eq:hmhf 2}},\\[1mm]
		\frac54,&\text{for \eqref{eq:NLH 6}}.
	\end{cases}
\end{equation}
In particular, $\varsigma\kappa<0$. 

We use the conventions $\mu_1=\mu_{N+1}=d_{N+1}=0$, $\lambda_0=\infty$, $\lambda_{N+1}=0$, and $\rho_{\lambda_0}=\rho_{\lambda_{N+1}}=0$. We fix a nonnegative nondecreasing function $\delta_{\mathrm p}$ tending to zero sufficiently slowly as $\alpha\to0$, with $\delta_{\mathrm p}(0)=0$ and $\alpha^{1/2}\lesssim\delta_{\mathrm p}(\alpha)$.
For the partial sums $U_j=\sum_{i=1}^j\iota_iW_{\lambda_i}$, we write
\begin{equation*}
	\calR_j=\calT(U_j)-\calT(U_{j-1}),\qquad 2\leq j \leq N.
\end{equation*}

\begin{lem}[$N$-bubble profile estimates]\label{prop:profile}
	For every fixed $N\geq1$ and sufficiently small $\alpha_0>0$, the following estimates hold for all $\blambda\in\calP_N(\alpha_0)$ and $\biota\in\{-1,1\}^N$.
	\begin{equation}\label{eq:Psi L2}
		\|\calT(U)\|_{L^2}+\|\Psi\|_{L^2}\lesssim\sqrt\calD.
	\end{equation}
	For $2\leq j\leq N$, we also have
	\begin{equation}\label{eq:raw diagonal}
		\bigg|((\Lambda W)_{\underline{\lambda_j}},\calR_j)+\frac{\iota_{j-1}\kappa c_*}{\lambda_{j-1}^2}\bigg|
		\leq\frac{\delta_{\mathrm p}(\alpha)}{\lambda_{j-1}^2}.
	\end{equation}
	For any radial function $\vartheta$ with $|\vartheta|\leq1$, we have
	\begin{align}
		|(\vartheta\rho_{\underline{\lambda_k}},\Psi)|&\lesssim\frac{1}{\lambda_{k-1}^2}+\calD,\qquad 1\leq k\leq N,\label{eq:Psi weighted}\\
		|((\Lambda W)_{\underline{\lambda_k}},\Psi)|&\lesssim\frac{\delta_{\mathrm p}(\alpha)}{\lambda_{k-1}^2}+\calD,\qquad 1\leq k\leq N.\label{eq:Psi kernel}
	\end{align}
	Moreover, for $1\leq k\leq N$,
	\begin{equation}\label{eq:potential interaction}
		\big\|[f'(P)-f'(\iota_kQ_{[\lambda_k]})]\iota_k(\Lambda W)_{\underline{\lambda_k}}\big\|_{L^2}\lesssim\frac{\mu_k+\mu_{k+1}^2}{\lambda_k}.
	\end{equation}
\end{lem}
\begin{proof}
	We estimate $\calT(U)$ and $\Psi$ by summing the interactions $\calR_j$. With $P_j=r^2U_j$, the stationarity of each bubble gives
	\begin{equation}\label{eq:f decom}
		\calR_j=r^{-4}(f(P_{j-1}+\iota_jQ_{[\lambda_j]})-f(P_{j-1})-f(\iota_jQ_{[\lambda_j]}) ),\quad
		\calT(U)=\sum_{j=2}^N\calR_j.
	\end{equation}
	By \eqref{eq:profile equation} and \eqref{eq:def calI}, we have
	\begin{equation*}
		\Psi=\sum_{j=2}^N\bigg(\calR_j+\frac{\iota_{j-1}\kappa}{\lambda_{j-1}^2}(\Lambda W)_{\lambda_j}\bigg).
	\end{equation*}
	We first prove \eqref{eq:Psi L2}. By $|\mathfrak s(a+b)|\leq|\mathfrak s(a)|+|\mathfrak s(b)|$, we have
	\begin{equation}\label{eq:P pointwise}
		r^{-2}|\mathfrak s(P)|\lesssim\sum_{i=1}^N\rho_{\lambda_i},\qquad r^{-2}|\mathfrak s(P-\iota_jQ_{[\lambda_j]})|\lesssim\sum_{i\neq j}\rho_{\lambda_i}.
	\end{equation}
	Applying \eqref{eq:f interaction} in \eqref{eq:f decom} and using \eqref{eq:P pointwise}, we obtain
	\begin{equation}\label{eq:f decom esti}
		|\calR_j|\lesssim r^{-4}|\mathfrak s(P_{j-1})||\mathfrak s(Q_{[\lambda_j]})|
		\lesssim\sum_{i<j}\rho_{\lambda_i}\rho_{\lambda_j}
		\lesssim\lambda_{j-1}^{-2}\rho_{\lambda_j}.
	\end{equation}
	For $i<j$, we have $\|\rho_{\lambda_i}\rho_{\lambda_j}\|_{L^2}^2\lesssim \lambda_j^2\lambda_i^{-4}$.
	Combining this and \eqref{eq:f decom esti} with $\|(\Lambda W)_{\lambda_j}\|_{L^2}=\lambda_j\sqrt{c_*}$, we obtain
	\begin{equation*}
		\|\calT(U)\|_{L^2}+\|\Psi\|_{L^2}\lesssim\sum_{j=2}^N\bigg(\sum_{i<j}\frac{\lambda_j}{\lambda_i^2}+d_j\bigg)\lesssim\sum_{j=2}^Nd_j\lesssim\sqrt\calD.
	\end{equation*}

	For \eqref{eq:Psi weighted} and \eqref{eq:Psi kernel}, we first estimate each term in the sum defining $\Psi$. By \eqref{eq:f decom esti} and \eqref{eq:rho overlap}, for any radial $\vartheta$ with $|\vartheta|\leq1$,
	\begin{equation}
		\bigg|\bigg(\vartheta\rho_{\underline{\lambda_k}},\calR_j+\frac{\iota_{j-1}\kappa}{\lambda_{j-1}^2}(\Lambda W)_{\lambda_j}\bigg)\bigg|
		\lesssim\lambda_{j-1}^{-2}(\rho_{\underline{\lambda_k}},\rho_{\lambda_j})
		\lesssim
		\begin{cases}
			\lambda_{j-1}^{-2},&j\leq k,\\
			d_j^2,&j>k.
		\end{cases}\label{eq:raw differ inner product}
	\end{equation}
	Summing over $2\leq j\leq N$, we obtain
	\begin{equation*}
		|(\vartheta\rho_{\underline{\lambda_k}},\Psi)|\lesssim\sum_{j=2}^k\lambda_{j-1}^{-2}+\sum_{j=k+1}^Nd_j^2\lesssim\frac{1}{\lambda_{k-1}^2}+\calD,
	\end{equation*}
	which proves \eqref{eq:Psi weighted}. This and $|\Lambda W|\lesssim\rho$ with $\lambda_0=\infty$ also prove \eqref{eq:Psi kernel} when $k=1$.

	We next prove \eqref{eq:raw diagonal} for $2\leq j\leq N$. The main contribution is obtained by replacing $U_{j-1}(r)$ with $U_{j-1}(0)$ in the linear interaction. By \eqref{eq:Q f}, we have
	\begin{equation*}
		|U_{j-1}|\lesssim\lambda_{j-1}^{-2},\qquad |U_{j-1}-U_{j-1}(0)|\lesssim r^2\lambda_{j-1}^{-4}, \qquad
		r^{-2}|f'(Q_{[\lambda_j]})|\lesssim\rho_{\lambda_j}.
	\end{equation*}
	Since $f(0)=f'(0)=0$ and $f'$ is even, we obtain from \eqref{eq:f decom} and \eqref{eq:f Taylor}
	\begin{equation*}
		|\calR_j-r^{-2}f'(Q_{[\lambda_j]})U_{j-1}(0)|
		\lesssim |U_{j-1}|^2+\rho_{\lambda_j}|U_{j-1}-U_{j-1}(0)|
		\lesssim\lambda_{j-1}^{-4}.
	\end{equation*}
	On the other hand, both $\calR_j$ and $r^{-2}f'(Q_{[\lambda_j]})U_{j-1}(0)$ are bounded by $\lambda_{j-1}^{-2}\rho_{\lambda_j}$ by \eqref{eq:f decom esti} and \eqref{eq:Q f}. Using $|(\Lambda W)_{\underline{\lambda_j}}|\lesssim r^{-4}$ and $\rho_{\lambda_j}\leq\lambda_j^2r^{-4}$, we therefore obtain
	\begin{align*}
		&|((\Lambda W)_{\underline{\lambda_j}},\calR_j-r^{-2}f'(Q_{[\lambda_j]})U_{j-1}(0))|\\
		&\lesssim\lambda_{j-1}^{-4}\int_0^{\sqrt{\lambda_{j-1}\lambda_j}}rdr
		+\frac{\lambda_j^2}{\lambda_{j-1}^2}\int_{\sqrt{\lambda_{j-1}\lambda_j}}^{\infty}r^{-3}dr
		\lesssim\frac{\mu_j}{\lambda_{j-1}^2}.
	\end{align*}
	It remains to compute the main contribution. Since $U_{j-1}(0)=W(0)\sum_{i<j}\iota_i\lambda_i^{-2}$, we have by scaling and \eqref{eq:kappa}
	\begin{align*}
		((\Lambda W)_{\underline{\lambda_j}},\calR_j)
		&=U_{j-1}(0)(\Lambda W,r^{-2}f'(Q))
		+O\bigg(\frac{\mu_j}{\lambda_{j-1}^2}\bigg)=-\frac{\iota_{j-1}\kappa c_*}{\lambda_{j-1}^2}
		+O\bigg(\frac{\delta_{\mathrm p}(\alpha)}{\lambda_{j-1}^2}\bigg).
	\end{align*}
	Here we used $\sum_{i<j-1}\lambda_i^{-2}\lesssim\alpha^2\lambda_{j-1}^{-2}$ and $\mu_j\leq\alpha$. Thus, we conclude \eqref{eq:raw diagonal} for sufficiently small $\alpha_0$.

	We now use \eqref{eq:raw diagonal} for $j=k$ and \eqref{eq:raw differ inner product} for $j\neq k$ to show \eqref{eq:Psi kernel} for $k\geq2$. Since $((\Lambda W)_{\underline{\lambda_k}},(\Lambda W)_{\lambda_k})=c_*$, we have
	\begin{equation*}
		|((\Lambda W)_{\underline{\lambda_k}},\Psi)|
		\lesssim\bigg|((\Lambda W)_{\underline{\lambda_k}},\calR_k)+\frac{\iota_{k-1}\kappa c_*}{\lambda_{k-1}^2}\bigg|+\sum_{j=2}^{k-1}\lambda_{j-1}^{-2}+\sum_{j=k+1}^Nd_j^2
		\lesssim \frac{\delta_{\mathrm p}(\alpha)}{\lambda_{k-1}^2}+\calD.
	\end{equation*}
	This proves \eqref{eq:Psi kernel} for $k\geq2$.

	For \eqref{eq:potential interaction}, we use $|\sin^2a-\sin^2b|\leq|\sin(a-b)|$ for \eqref{eq:hmhf 2} and $||a|-|b||\leq|a-b|$ for \eqref{eq:NLH 6}. In both cases, we obtain from \eqref{eq:P pointwise}
	\begin{equation*}
		|f'(P)-f'(\iota_kQ_{[\lambda_k]})|\lesssim\sum_{i\neq k}|\mathfrak s(Q_{[\lambda_i]})|\lesssim r^2\sum_{i\neq k}\rho_{\lambda_i}.
	\end{equation*}
	Since $\rho_{\underline{\lambda_k}}\leq r^{-4}$ and $\rho_{\lambda_i}\leq\lambda_i^2r^{-4}$, we have by scaling
	\begin{equation}
		\|r^2\rho_{\lambda_i}\rho_{\underline{\lambda_k}}\|_{L^2}
		\leq\min\{\|r^{-2}\rho_{\lambda_i}\|_{L^2},\lambda_i^2\|r^{-2}\rho_{\underline{\lambda_k}}\|_{L^2}\}
		\lesssim\min\{\lambda_i^{-1},\lambda_i^2\lambda_k^{-3}\}.
	\end{equation}
	Using $|\Lambda W|\lesssim\rho$ and summing over $i\neq k$, we obtain
	\begin{equation*}
		\|[f'(P)-f'(\iota_kQ_{[\lambda_k]})](\Lambda W)_{\underline{\lambda_k}}\|_{L^2}
		\lesssim\sum_{i<k}\lambda_i^{-1}+\lambda_k^{-3}\sum_{i>k}\lambda_i^2
		\lesssim\frac{\mu_k+\mu_{k+1}^2}{\lambda_k}. \qedhere
	\end{equation*}
\end{proof}

We next prove a coercivity estimate for the dissipation term.

\begin{lem}[Coercivity estimate for the dissipation term]\label{lem:dissipation coercivity}
	For sufficiently small $\alpha_0>0$ and every $\blambda\in\calP_N(\alpha_0)$, every radial $h\in\dot H^2$ satisfying $(Z_{\lambda_j},h)=0$ for $1\leq j\leq N$ obeys
	\begin{equation}\label{eq:dissipation coercivity}
		\calD+\|h\|_{\dot H^2}^2\lesssim\|\calT(U)-H_Uh\|_{L^2}^2.
	\end{equation}
\end{lem}

\begin{proof}
	We first compare $H_U$ with $H_{\blambda}$ to apply Lemma~\ref{lem:coercivity N}. On $[\sqrt{\lambda_j\lambda_{j+1}},\sqrt{\lambda_{j-1}\lambda_j}]$, we have
	\begin{equation*}
		\sum_{i<j}|Q_{[\lambda_i]}|+
		\begin{cases}
			\textstyle\sum_{i>j}|Q_{[\lambda_i]}-\pi|,&\text{for \eqref{eq:hmhf 2}},\\
			\textstyle\sum_{i>j}|Q_{[\lambda_i]}|,&\text{for \eqref{eq:NLH 6}}
		\end{cases}
		\lesssim\alpha.
	\end{equation*}
	The function $f'$ is Lipschitz in both cases and $\pi$-periodic for \eqref{eq:hmhf 2}. Consequently, using \eqref{eq:sobolev}, we get
	\begin{equation*}
		\|(H_U-H_{\blambda})h\|_{L^2}\lesssim \bigg\|f'(P)-\sum_{i=1}^Nf'(\iota_iQ_{[\lambda_i]})\bigg\|_{L^\infty}\|h\|_{\dot H^2}
		\lesssim\alpha \|h\|_{\dot H^2}.
	\end{equation*}
	Thus, applying Lemma~\ref{lem:coercivity N}, we arrive at
	\begin{equation}\label{eq:spacetime coercivity}
		\|H_Uh\|_{L^2}\sim\|h\|_{\dot H^2}.
	\end{equation}
	When $N=1$, this proves the lemma because $\calT(U)=0$. For $N\geq2$, \eqref{eq:potential interaction} yields
	\begin{equation*}
		\lambda_j\|[f'(P)-f'(\iota_jQ_{[\lambda_j]})](\Lambda W)_{\underline{\lambda_j}} \|_{L^2}\leq\delta_{\mathrm p}(\alpha).
	\end{equation*}
	Using this with $H_{\lambda_j}(\Lambda W)_{\underline{\lambda_j}}=0$ and \eqref{eq:sobolev}, we obtain
	\begin{equation}\label{eq:profile linear projection}
		\lambda_j((\Lambda W)_{\underline{\lambda_j}},H_Uh)
		=-\lambda_j([f'(P)-f'(\iota_jQ_{[\lambda_j]})](\Lambda W )_{\underline{\lambda_j}},r^{-2}h)
		\lesssim\delta_{\mathrm p}(\alpha)\|h\|_{\dot H^2}.
	\end{equation}
	For $2\leq k\leq N$ and $k\neq j$, \eqref{eq:f decom esti} and \eqref{eq:rho overlap} give
	\begin{equation*}
		\lambda_j|((\Lambda W)_{\underline{\lambda_j}},\calR_k)|
		\lesssim \lambda_j\lambda_{k-1}^{-2}\min\{1,\lambda_k^2\lambda_j^{-2}\}
		=d_k\min\{\lambda_j\lambda_k^{-1},\lambda_k\lambda_j^{-1}\}.
	\end{equation*}
	Since $\calT(U)=\sum_{k=2}^N\calR_k$, this and \eqref{eq:raw diagonal} yield, for $2\leq j\leq N$ and sufficiently small $\alpha$,
	\begin{equation}\label{eq:profile nonlinear projection}
		|\lambda_j((\Lambda W)_{\underline{\lambda_j}},\calT(U))+\kappa c_*\iota_{j-1}d_j |
		\leq\delta_{\mathrm p}(\alpha)\sqrt\calD.
	\end{equation}
	By \eqref{eq:profile linear projection} and \eqref{eq:profile nonlinear projection}, we have
	\begin{equation*}
		|\kappa|c_*d_j\lesssim
		|\lambda_j((\Lambda W)_{\underline{\lambda_j}}, \mathcal T(U)-H_Uh)|
		+\delta_{\mathrm p}(\alpha)\|h\|_{\dot H^2}+\delta_{\mathrm p}(\alpha)\sqrt{\mathcal D}.
	\end{equation*}
	Thus, taking the $\ell^2$ norm in $j$ and absorbing $\delta_{\mathrm p}(\alpha)\sqrt\calD$, we have
	\begin{equation}\label{eq:interaction projection control}
		\sqrt\calD\lesssim\|\calT(U)-H_Uh\|_{L^2}+\delta_{\mathrm p}(\alpha)\|h\|_{\dot H^2}.
	\end{equation}
	On the other hand, \eqref{eq:spacetime coercivity} and \eqref{eq:Psi L2} give
	\begin{equation*}
		\|h\|_{\dot H^2}\lesssim\|H_Uh\|_{L^2}\leq\|\calT(U)-H_Uh\|_{L^2}+\|\calT(U)\|_{L^2}
		\lesssim\|\calT(U)-H_Uh\|_{L^2}+\sqrt\calD.
	\end{equation*}
	Substituting \eqref{eq:interaction projection control} and absorbing $\|h\|_{\dot H^2}$, we conclude \eqref{eq:dissipation coercivity}.
\end{proof}

We now combine Lemma~\ref{lem:dissipation coercivity} with the energy identities to prove Proposition~\ref{prop:spacetime}.

\begin{proof}[Proof of Proposition~\ref{prop:spacetime}]
	We recall that the energy identities \eqref{eq:energy identity} with $u_t=r^{-2}v_t$ for \eqref{eq:hmhf 2} and \eqref{eq:NLH energy identity} for \eqref{eq:NLH 6} provide, for all $t_0^*\in[t_{\mathrm d},T_+)$,
	\begin{equation}\label{eq:ut spacetime}
		\int_{t_0^*}^{T_+}\|u_t(t)\|_{L^2}^2dt<\infty.
	\end{equation}
	
	\textbf{Step 1.}
	We first prove that there exist $t_0^*\in[t_{\mathrm d},T_+)$ and $r_0\in(0,\infty]$ such that
	\begin{equation}\label{eq:interior spacetime}
		\int_{t_0^*}^{T_+} \{\calD(t) +\|\chi_{2r_0}g(t)\|_{\dot H^2}^2\}dt<\infty.
	\end{equation}
	If $T_+=\infty$, then we take $r_0=\infty$.
	
	We fix a sufficiently small $\eta>0$. In the finite-time blow-up case, we use \eqref{eq:sobolev} and \eqref{eq:local qualitative smallness} to choose $r_0>0$ sufficiently small and then $t_0^*\in[t_{\mathrm d},T_+)$ sufficiently close to $T_+$ so that $\sup_{t_0^*\leq t<T_+}\|\chi_{4r_0}g(t)\|_{H^1}\leq\eta$. After increasing $t_0^*$ if necessary, we may also assume that every $Z_{\lambda_j}$ is supported on $r<2r_0$. In the global case, we take $r_0=\infty$, with $\chi_{2r_0}=\chi_{4r_0}\equiv1$.
	By \eqref{eq:ut spacetime}, it suffices to prove
	\begin{equation}\label{eq:pointwise spacetime}
		\begin{cases}
			\calD(t)+\|\chi_{2r_0}g(t)\|_{\dot H^2}^2\lesssim_{r_0}\|\chi_{2r_0}u_t(t)\|_{L^2}^2+\eta^2,&T_+<\infty,\\
			\calD(t)+\|g(t)\|_{\dot H^2}^2\lesssim\|u_t(t)\|_{L^2}^2,&T_+=\infty.
		\end{cases}
	\end{equation}
	To apply Lemma~\ref{lem:dissipation coercivity}, we express $\chi_{2r_0}u_t$ in terms of $\chi_{2r_0}g$. 
	Since $u=U+g$, we obtain from \eqref{eq:profile equation}
	\begin{equation}\label{eq:u equation}
		u_t=-\calI-H_Ug+\NL_U(g)+\Psi,
	\end{equation}
	where
	\begin{equation*}
		\NL_U(g)\coloneqq r^{-4}[f(P+r^2g)-f(P)-f'(P)r^2g].
	\end{equation*}
	By \eqref{eq:fprime} and the fundamental theorem of calculus, we have
	\begin{equation}\label{eq:NL esti}
		|\NL_U(g)|\lesssim|g|^2.
	\end{equation}
	By multiplying \eqref{eq:u equation} by $\chi_{2r_0}$ and using \eqref{eq:profile equation}, we obtain
	\begin{equation}\label{eq:localized tension}
		\chi_{2r_0}u_t=\calT(U)-H_U(\chi_{2r_0}g)+\td\NL,
	\end{equation}
	where
	\begin{equation*}
		\td\NL\coloneqq-(1-\chi_{2r_0})\calT(U)-[\Delta,\chi_{2r_0}]g+\chi_{2r_0}\NL_U(g).
	\end{equation*}
	In the finite-time blow-up case, we use \eqref{eq:f decom}, \eqref{eq:f decom esti}, and $\rho_{\lambda_j}\leq\lambda_j^2r^{-4}$ to obtain
	\begin{equation}\label{eq:profile error exterior}
		\|(1-\chi_{2r_0})\calT(U)\|_{L^2}
		\lesssim \lambda_1r_0^{-1}\sqrt\calD.
	\end{equation}
	Since $\chi_{4r_0}=1$ on $\operatorname{supp}\chi_{2r_0}$, we obtain from \eqref{eq:sobolev} and \eqref{eq:NL esti}
	\begin{align*}
		\|[\Delta,\chi_{2r_0}]g\|_{L^2}&\lesssim_{r_0}\|\chi_{4r_0}g\|_{H^1}\lesssim_{r_0}\eta,\\
		\|\chi_{2r_0}\NL_U(g)\|_{L^2}&\lesssim\|\chi_{4r_0}g\|_{L^3}\|\chi_{2r_0}g\|_{L^6}
		\lesssim\eta \|\chi_{2r_0}g\|_{\dot H^2}.
	\end{align*}
	In the global case $\chi_{2r_0}=1$, so
	\begin{equation*}
		\|\td\NL\|_{L^2}=\|\NL_U(g)\|_{L^2}\lesssim\|g\|_{\dot H^1}\|g\|_{\dot H^2}.
	\end{equation*}
	Since each $Z_{\lambda_j}$ is supported where $\chi_{2r_0}=1$, $\chi_{2r_0}g$ satisfies the orthogonality conditions in Lemma~\ref{lem:dissipation coercivity}. Applying this lemma to \eqref{eq:localized tension}, we get
	\begin{equation*}
		\calD+\|\chi_{2r_0}g\|_{\dot H^2}^2\lesssim\|\calT(U)-H_U(\chi_{2r_0}g)\|_{L^2}^2
		\lesssim \|\chi_{2r_0}u_t\|_{L^2}^2+\|\td\NL\|_{L^2}^2.
	\end{equation*}
	In the finite-time blow-up case, we take $\eta$ sufficiently small and then $t_0^*$ sufficiently close to $T_+$ so that $\lambda_1/r_0$ in \eqref{eq:profile error exterior} is sufficiently small; in the global case, we use \eqref{eq:global H1 smallness}. After absorbing the terms containing $\|\chi_{2r_0}g\|_{\dot H^2}$ and $\sqrt\calD$, we obtain \eqref{eq:pointwise spacetime}.

	We obtain \eqref{eq:interior spacetime} by integrating \eqref{eq:pointwise spacetime} and using \eqref{eq:ut spacetime}. This also proves \eqref{eq:spacetime} when $T_+=\infty$, since $r_0=\infty$.

	\textbf{Step 2.}
	It remains to control the exterior radiation when $T_+<\infty$:
	\begin{equation}\label{eq:exterior spacetime}
		\int_{t_0^*}^{T_+} \|(1-\chi_{r_0})g(t)\|_{\dot H^2}^2dt<\infty.
	\end{equation}
	For \eqref{eq:hmhf 2}, we argue as for $v^*$ in the proof of Proposition~\ref{prop:decomposition solution}. Since $E(v(t))\leq E(v_0)$ and $v(t,0)=v(0,0)$, we obtain
	\begin{equation*}
		\sup_{0\leq t<T_+}\|v(t)\|_{L^\infty_r}=\sup_{0\leq t<T_+}\|r^2u(t)\|_{L^\infty_r}\leq C(E(v_0),v(0,0)).
	\end{equation*}
	For \eqref{eq:NLH 6}, we obtain $\sup_{0\leq t<T_+}\|r^2u(t)\|_{L^\infty}<\infty$ by radial Sobolev and \eqref{eq:NLH bounded}. Thus $r^2u$ and every $Q_{[\lambda_j]}$ are uniformly bounded in both models. Using \eqref{eq:operators}, we have
	\begin{equation*}
		\Delta g=u_t-r^{-4}\bigg(f(r^2u)-\sum_{j=1}^Nf(\iota_jQ_{[\lambda_j]})\bigg),\qquad
		\|(1-\chi_{r_0})\Delta g\|_{L^2}\lesssim\|u_t\|_{L^2}+r_0^{-1}.
	\end{equation*}
	Since $T_+<\infty$, we obtain by \eqref{eq:ut spacetime}
	\begin{equation}\label{eq:exterior Delta g}
		(1-\chi_{r_0})\Delta g\in L^2_t((t_0^*,T_+);L^2).
	\end{equation}
	To control the commutator $[\Delta,\chi_{r_0}]g$, we use $\chi_{2r_0}=1$ and $|\partial_r^2\chi_{r_0}|+r^{-1}|\partial_r\chi_{r_0}|\lesssim r^{-2}$ on the support of $\nabla\chi_{r_0}$. By \eqref{eq:sobolev}, we have
	\begin{equation*}
		\|[\Delta,\chi_{r_0}]g\|_{L^2}\lesssim \|r^{-1}\partial_r(\chi_{2r_0}g)\|_{L^2}+\|r^{-2}\chi_{2r_0}g\|_{L^2}\lesssim\|\chi_{2r_0}g\|_{\dot H^2}.
	\end{equation*}
	Consequently,
	\begin{equation*}
		\|\Delta((1-\chi_{r_0})g)\|_{L^2} \lesssim\|(1-\chi_{r_0})\Delta g\|_{L^2}+\|\chi_{2r_0}g\|_{\dot H^2}.
	\end{equation*}
	We conclude \eqref{eq:exterior spacetime} by integrating in time and using \eqref{eq:exterior Delta g} and \eqref{eq:interior spacetime}.

	Finally, using $g=\chi_{r_0}(\chi_{2r_0}g)+(1-\chi_{r_0})g$ and \eqref{eq:sobolev}, we combine \eqref{eq:interior spacetime}, \eqref{eq:exterior spacetime}, and \eqref{eq:ut spacetime} to conclude \eqref{eq:spacetime}.
\end{proof}

\section{Classification}
In this section, we fix a solution satisfying the assumptions of Theorem~\ref{thm:hmhf main} or Theorem~\ref{thm:nlh main} and work on $[t_0^*,T_+)$ with the decomposition and spacetime estimate established in Sections~\ref{subsec:decomposition} and~\ref{subsec:spacetime estimates}.
Following the overall strategy of K.~Kim--Merle \cite{KimMerle2025CPAM}, we derive and integrate the modulation estimates to complete the proofs of the main theorems.
For the inner scales, however, instead of deriving refined evolution laws for each scale as in \cite[Section~5]{KimMerle2025CPAM}, we adapt the approach of \cite{Kim2026arXivNobubbletree} to study the relative dynamics of $\lambda_k$ and $\lambda_{k-1}$ for $2\leq k\leq N$.
The asymptotic law for $\lambda_1$ is obtained separately from the initial tail by adapting the exterior argument of Gustafson--Nakanishi--Tsai \cite[Section~9]{GustafsonNakanishiTsai2010CMP}.

\subsection{Modulation estimates}
The goal of this subsection is to derive the modulation estimates. To control the error terms in the refined modulation equations, we first need preliminary bounds for the time derivatives of the scales. We define 
\begin{equation}\label{eq:modulation coefficients}
	\calM_{kj}\coloneqq-\delta_{kj}-(\iota_kZ_{\underline{\lambda_k}},\lambda_j\partial_{\lambda_j}U)
    -\delta_{kj}(\iota_k(\Lambda_{-1}Z)_{\underline{\lambda_k}},g), \quad h_k\coloneqq-(\iota_kZ_{\underline{\lambda_k}},u_t),
\end{equation}
for $1\leq k,j\leq N$. Differentiating the orthogonality conditions \eqref{eq:orthogonality} yields a linear system for $\lambda_{j,t}/\lambda_j$. In the following lemma, we estimate $\calM_{kj}$ and $h_k$ and establish the weighted inversion estimate needed to control the scale derivatives.

\begin{lem}[Preliminary estimates]\label{lem:preliminary estimates}
	There exists $t_1^*\in[t_0^*,T_+)$ such that, for any $t\in(t_1^*,T_+)$, the following estimates hold for $1\leq k,j\leq N$.
	\begin{itemize}
		\item (Matrix coefficients) We have
		\begin{equation}\label{eq:matrix}
			|\calM_{kj}|\lesssim \begin{cases}
				(\lambda_k/\lambda_j)^2,&j<k,\\
				\|\langle r/\lambda_1\rangle^{-1}g\|_{\dot H^1}=o_t(1),&j=k,\\
				(\lambda_j/\lambda_k)^2,&j>k.
			\end{cases}
		\end{equation}

		\item (Size of $h_k$) We have
		\begin{equation}\label{eq:hk bound}
			|h_k|\lesssim \frac{\|g\|_{\dot H^2}}{\lambda_k} +\frac{1}{\lambda_{k-1}^2} +\calD+\|g\|_{\dot H^2}^2.
		\end{equation}

		\item (Weighted inversion) The scale derivatives satisfy
		\begin{equation}\label{eq:mod law}
			\sum_{j=1}^N(\delta_{kj}+\calM_{kj})\frac{\lambda_{j,t}}{\lambda_j}=h_k.
		\end{equation}
		Let $\omega_1,\ldots,\omega_N>0$ satisfy
		\begin{equation}\label{eq:weight condition}
			\frac{\omega_k}{\omega_j}\lesssim1 \quad\text{if }j<k,
			\qquad
			\frac{\omega_k}{\omega_j}\bigg(\frac{\lambda_j}{\lambda_k}\bigg)^2\lesssim1
			\quad\text{if }j>k.
		\end{equation}
		Then
		\begin{equation}\label{eq:weighted linear}
			\sum_{j=1}^N\omega_j\bigg|\frac{\lambda_{j,t}}{\lambda_j}\bigg|
			\lesssim\sum_{k=1}^N\omega_k|h_k|.
		\end{equation}
	\end{itemize}
\end{lem}

\begin{proof}
	\textbf{Step 1.} We first prove \eqref{eq:matrix}. By \eqref{eq:Z normalization} and scaling, the diagonal contribution of $-\delta_{kj}-(\iota_kZ_{\underline{\lambda_k}},\lambda_j\partial_{\lambda_j}U)$ is zero. For $-\delta_{kj}(\iota_k(\Lambda_{-1}Z)_{\underline{\lambda_k}},g)$, using $1\lesssim \langle r/\lambda_1\rangle^{-1}$ on the support of $(\Lambda_{-1}Z)_{\underline{\lambda_k}}$, \eqref{eq:sobolev}, and \eqref{eq:weighted qualitative smallness}, we get
	\begin{equation*}
		|(\iota_k(\Lambda_{-1}Z)_{\underline{\lambda_k}},g)|\lesssim \|\langle r/\lambda_1\rangle^{-1}g\|_{\dot H^1}=o_t(1),
	\end{equation*}
    which proves the diagonal line of \eqref{eq:matrix}. 
    In addition, using $\lambda_j\partial_{\lambda_j}U=-\iota_j(\Lambda W)_{\lambda_j}$ and \eqref{eq:test overlap} proves the off-diagonal lines of \eqref{eq:matrix}. 

	\textbf{Step 2.} Now, we show \eqref{eq:hk bound}. First, recall \eqref{eq:u equation}.
	Using \eqref{eq:test overlap} with \eqref{eq:def calI},
	\begin{equation}\label{eq:interaction test bound}
		|(\iota_kZ_{\underline{\lambda_k}},\calI)| \lesssim \sum_{j=2}^N\frac{1}{\lambda_{j-1}^2} \min\bigg\{ \bigg(\frac{\lambda_j}{\lambda_k}\bigg)^2, \bigg(\frac{\lambda_k}{\lambda_j}\bigg)^2 \bigg\} 
		\lesssim \frac{1}{\lambda_{k-1}^2}+\calD.
	\end{equation}
	Indeed, the term $j=k$, when present, is bounded by $1/\lambda_{k-1}^2$.
	If $j<k$, then $\mu_j\leq\alpha_0<1$ and $\lambda_k\leq\lambda_{j+1}$, so
	\begin{equation*}
		\frac{1}{\lambda_{j-1}^2} \bigg(\frac{\lambda_k}{\lambda_j}\bigg)^2 \leq\frac{\lambda_{j+1}^2}{\lambda_j^4} =d_{j+1}^2.
	\end{equation*}
	If $j>k$, then $\lambda_{j-1}\leq\lambda_k$, and hence
	\begin{equation*}
		\frac{1}{\lambda_{j-1}^2} \bigg(\frac{\lambda_j}{\lambda_k}\bigg)^2 =\frac{\lambda_j^2}{\lambda_{j-1}^2\lambda_k^2} \leq\frac{\lambda_j^2}{\lambda_{j-1}^4} =d_j^2.
	\end{equation*}
	Summing these contributions proves \eqref{eq:interaction test bound}.

	For the linear term $H_Ug=-\Delta g-r^{-2}f'(P)g$, note that $f'$ is bounded for \eqref{eq:hmhf 2}, while $\|P\|_{L^\infty}\lesssim1$ gives $\|f'(P)\|_{L^\infty}\lesssim1$ for \eqref{eq:NLH 6}. By \eqref{eq:sobolev} and $\|Z_{\underline{\lambda_k}}\|_{L^2}\lesssim \lambda_k^{-1}$, we obtain
	\begin{equation*}
		|(\iota_kZ_{\underline{\lambda_k}},H_Ug)| \lesssim \lambda_k^{-1}\|g\|_{\dot H^2}.
	\end{equation*}
	For the nonlinear term, thanks to $|Z_{\underline{\lambda_k}}(r)|\lesssim r^{-4}$ with \eqref{eq:sobolev} and \eqref{eq:NL esti}, we get
	\begin{equation*}
		|(\iota_kZ_{\underline{\lambda_k}},\NL_U(g))| \lesssim\|r^{-2}g\|_{L^2}^2 \lesssim\|g\|_{\dot H^2}^2.
	\end{equation*}
	Finally, $|Z_{\underline{\lambda_k}}|\lesssim \rho_{\underline{\lambda_k}}$, so \eqref{eq:Psi weighted} gives
	\begin{equation*}
		|(\iota_kZ_{\underline{\lambda_k}},\Psi)| \lesssim \lambda_{k-1}^{-2}+\calD.
	\end{equation*}
	Combining these bounds proves \eqref{eq:hk bound}.

	\textbf{Step 3.} We show \eqref{eq:mod law} and \eqref{eq:weighted linear}. Differentiating \eqref{eq:orthogonality} with $g=u-U$ and \eqref{eq:scaling}, we obtain \eqref{eq:mod law}.

    Let $\omega_1,\ldots,\omega_N>0$ satisfy \eqref{eq:weight condition}. Then \eqref{eq:mod law} becomes
    \begin{equation}\label{eq:mod law matrix form}
    	\operatorname{diag}(\omega_k)(I+\calM)\operatorname{diag}(\omega_k)^{-1}
    	\begin{pmatrix}
    		\omega_1\lambda_{1,t}/\lambda_1\\
    		\vdots\\
    		\omega_N\lambda_{N,t}/\lambda_N
    	\end{pmatrix}
    	=
    	\begin{pmatrix}
    		\omega_1h_1\\
    		\vdots\\
    		\omega_Nh_N
    	\end{pmatrix}.
    \end{equation}
    The $(k,j)$-entry of the conjugated $\calM$ is
    \begin{equation*}
    	(\operatorname{diag}(\omega_k)\calM\operatorname{diag}(\omega_k)^{-1})_{kj}
    	=
    	\frac{\omega_k}{\omega_j}\calM_{kj}.
    \end{equation*}
    By \eqref{eq:matrix} and \eqref{eq:weight condition}, $\calM_{kk}=o_{t}(1)$ and
    \begin{align*}
    	\frac{\omega_k}{\omega_j}|\calM_{kj}|
    	\lesssim
    	\left(\frac{\lambda_k}{\lambda_j}\right)^2
    	=o_{t}(1), \qquad j<k,
        \\
        \frac{\omega_k}{\omega_j}|\calM_{kj}|
    	\lesssim
    	\frac{\omega_k}{\omega_j}
    	\left(\frac{\lambda_j}{\lambda_k}\right)^2
    	\lesssim1, \qquad j>k.
    \end{align*}
    Consequently,
    \begin{equation*}
    	\operatorname{diag}(\omega_k)(I+\calM)\operatorname{diag}(\omega_k)^{-1}
    	=I+\calU+\calS,
    \end{equation*}
    where $\calU$ is strictly upper triangular and
    \begin{equation*}
    	\|\calU\|\lesssim1,\qquad\calU^N=0,
    	\qquad
    	\|\calS\|=o_{t}(1).
    \end{equation*}
    Since $(I+\calU)^{-1}=\textstyle\sum_{m=0}^{N-1}(-\calU)^m$, we have
    \begin{equation*}
    	\|(I+\calU)^{-1}\calS\|_{\ell^1\to\ell^1}
    	=o_{t}(1).
    \end{equation*}
    After increasing $t_1^*>t_0^*$ if necessary, the factorization
    \begin{equation*}
    	I+\calU+\calS
    	=(I+\calU)
    	\left[I+(I+\calU)^{-1}\calS\right]
    \end{equation*}
    and the Neumann series yield
    \begin{equation*}
    	\|(I+\calU+\calS)^{-1}\|_{\ell^1\to\ell^1}
    	\lesssim1.
    \end{equation*}
    Applying this inverse to \eqref{eq:mod law matrix form} proves \eqref{eq:weighted linear}.
\end{proof}

The following proposition gives bounds for the scale derivatives and refined modulation estimates for the largest and inner scales.

\begin{prop}[Modulation estimates]\label{prop:modulation}
	Let $t_1^*$ be as in Lemma~\ref{lem:preliminary estimates}. The following estimates hold for every $1\leq k\leq N$ and $t\in(t_1^*,T_+)$.
	\begin{itemize}
		\item \emph{(Scale estimates)} We have
		\begin{equation}
			\bigg|\frac{\lambda_{k,t}}{\lambda_k}\bigg| \lesssim \frac{\|g\|_{\dot H^2}}{\lambda_k} +\frac{1}{\lambda_{k-1}^2} +\calD+\|g\|_{\dot H^2}^2. \label{eq:lambda t esti}
		\end{equation}
		If $N\geq2$, then
		\begin{equation}\label{eq:weighted lambda t esti}
			\mu_2^2\bigg|\frac{\lambda_{1,t}}{\lambda_1}\bigg| +\sum_{j=2}^N \bigg(\frac{\lambda_j}{\lambda_1}\bigg)^2 \bigg|\frac{\lambda_{j,t}}{\lambda_j}\bigg| \lesssim\calD+\|g\|_{\dot H^2}^2.
		\end{equation}
		Moreover,
		\begin{equation}\label{eq:physical scale velocity}
			\sum_{j=1}^N|\lambda_{j,t}(t)|^2\lesssim\|u_t(t)\|_{L^2}^2, \qquad \int_{t_1^*}^{T_+}\sum_{j=1}^N|\lambda_{j,t}(t)|^2dt<\infty.
		\end{equation}
		\item For $R\geq1$ and $\iota_1=1$, we have
		\begin{equation}\label{eq:first mod esti}
			\bigg| \frac{\lambda_{1,t}}{\lambda_1} -\frac1{c_*}\frac{d}{dt} ((\Lambda W)_{\underline{\lambda_1(t)}}, g(t)\chi_R) \bigg| \lesssim R^{-1}\|g(t)\|_{\dot H^2} +\calD(t)+\|g(t)\|_{\dot H^2}^2.
		\end{equation}
		In addition, when $T_+<\infty$, we have
		\begin{equation}\label{eq:localized correction bound}
			| ((\Lambda W)_{\underline{\lambda_1(t)}}, g(t)\chi) | \lesssim (1+\log\lambda_1(t)^{-1})^{1/2}.
		\end{equation}
		\item If $T_+=\infty$ and $N\geq2$, then, for every $2\leq k\leq N$, the function $\mathfrak b_k$ given by
        \begin{equation*}
        	\mathfrak b_k(t)\coloneqq\frac{\iota_k}{c_*}((\Lambda W)_{\underline{\lambda_k}}-(\Lambda W)_{\underline{\lambda_{k-1}}},g)
        \end{equation*}
        satisfies
		\begin{equation}\label{eq:frak b esti}
			|\mathfrak b_k|\lesssim\|g\|_{\dot H^1}(1+|\log\mu_k|)^{1/2}=o(|\log\mu_k| ).
		\end{equation}
		Moreover,
		\begin{equation}\label{eq:relation modulation law}
			\begin{aligned}
				&\bigg|\frac d{dt}(\log\lambda_k+\varsigma\iota_{k-1}\iota_k\log\lambda_{k-1}-\mathfrak b_k)-\iota_{k-1}\iota_k\frac{\kappa}{\lambda_{k-1}^2}\bigg|\\
				&\qquad\lesssim\frac{o_{t}(1)}{\lambda_{k-1}^2}+\mathbf1_{\{k\geq3\}}\frac1{\lambda_{k-2}^2}+\frac{\|g\|_{\dot H^2}}{\lambda_{k-1}}+\calD+\|g\|_{\dot H^2}^2.
			\end{aligned}
		\end{equation}
	\end{itemize}

\end{prop}
At $D=2$, the inner product $((\Lambda W)_{\underline{\lambda_k}},g)$ need not be well-defined for $g\in\dot H^1$, since $\Lambda W$ has an $r^{-4}$ tail and $r\Lambda W\notin L^2$; see \eqref{eq:unified kernel formulas}. We therefore use a localized correction for the largest scale $\lambda_1$, while for the inner scales we compare the scaling directions at $\lambda_k$ and $\lambda_{k-1}$.

Accordingly, for $2\leq k\leq N$, we test against $(\Lambda W)_{\underline{\lambda_k}}-(\Lambda W)_{\underline{\lambda_{k-1}}}$. The leading $r^{-4}$ tails cancel in this difference, so the resulting inner product with $g\in\dot H^1$ is well-defined. On the intermediate region $\lambda_k < r <\lambda_{k-1}$, however, the difference is still of size $r^{-4}$, which yields the logarithmic factor in \eqref{eq:frak b esti}. Testing against this difference gives the coupled modulation law \eqref{eq:relation modulation law}.

\begin{proof}[Proof of Proposition~\ref{prop:modulation}]
	\textbf{Step 1.} We first prove \eqref{eq:lambda t esti} and \eqref{eq:weighted lambda t esti}.
	Fix $1\leq k\leq N$ and set
	\begin{equation*}
		\omega_j=\min\{1,(\lambda_j/\lambda_k)^2\},\qquad 1\leq j\leq N.
	\end{equation*}
	The sequences $\omega_j$ and $\omega_j/\lambda_j^2$ are respectively nonincreasing and nondecreasing, so \eqref{eq:weight condition} holds. Thus, we have
	\begin{equation*}
		\sum_{j=1}^N\frac{\omega_j}{\lambda_j}\lesssim\lambda_k^{-1},\qquad
		\sum_{j=1}^N\frac{\omega_j}{\lambda_{j-1}^2}\lesssim\frac{1}{\lambda_{k-1}^2}+\calD.
	\end{equation*}
	Indeed, $\omega_j/\lambda_{j-1}^2$ is at most $1/\lambda_{k-1}^2$ for $j\leq k$, and at most $d_j^2$ for $j>k$. Applying \eqref{eq:weighted linear} with \eqref{eq:hk bound}, we obtain
	\begin{equation*}
		\sum_{j=1}^N\omega_j\bigg|\frac{\lambda_{j,t}}{\lambda_j}\bigg|
		\lesssim\frac{\|g\|_{\dot H^2}}{\lambda_k}+\frac{1}{\lambda_{k-1}^2}+\calD+\|g\|_{\dot H^2}^2.
	\end{equation*}
	Since $\omega_k=1$, this proves \eqref{eq:lambda t esti}. If $N\geq2$, taking $k=2$ and multiplying by $\mu_2^2$ yields
	\begin{equation*}
		\mu_2^2\bigg|\frac{\lambda_{1,t}}{\lambda_1}\bigg|
		+\sum_{j=2}^N\bigg(\frac{\lambda_j}{\lambda_1}\bigg)^2\bigg|\frac{\lambda_{j,t}}{\lambda_j}\bigg|
		\lesssim d_2\|g\|_{\dot H^2}+d_2^2+\mu_2^2(\calD+\|g\|_{\dot H^2}^2)
		\lesssim\calD+\|g\|_{\dot H^2}^2.
	\end{equation*}
	This proves \eqref{eq:weighted lambda t esti}. We also record, for $2\leq k\leq N$,
	\begin{equation}\label{eq:lambda element}
		\bigg(\frac{\lambda_k}{\lambda_1}\bigg)^2\frac{\|g\|_{\dot H^2}}{\lambda_k}
		\leq d_k\|g\|_{\dot H^2},
		\qquad
		\bigg(\frac{\lambda_k}{\lambda_1}\bigg)^2\frac{1}{\lambda_{k-1}^2}\leq d_k^2.
	\end{equation}

	\textbf{Step 2.} Next, we show \eqref{eq:physical scale velocity}.
	Choose $\omega_k=\lambda_k$. Then \eqref{eq:weight condition} holds, and \eqref{eq:weighted linear} gives
	\begin{equation*}
		\sum_{j=1}^N|\lambda_{j,t}| \lesssim\sum_{k=1}^N\lambda_k|h_k|.
	\end{equation*}
	By \eqref{eq:modulation coefficients} and \eqref{eq:norm scaling}, we have
	\begin{equation*}
		\lambda_k|h_k| =\lambda_k|(\iota_kZ_{\underline{\lambda_k}},u_t)| \leq\|Z\|_{L^2}\|u_t\|_{L^2},
	\end{equation*}
	which yields
	\begin{equation*}
		\sum_{j=1}^N|\lambda_{j,t}|^2 \leq\bigg(\sum_{j=1}^N|\lambda_{j,t}|\bigg)^2 \lesssim\|u_t\|_{L^2}^2.
	\end{equation*}
	Integrating in time and applying \eqref{eq:ut spacetime} proves \eqref{eq:physical scale velocity}.

	\textbf{Step 3.} We prove \eqref{eq:first mod esti} and \eqref{eq:localized correction bound}.
	Assume that $\iota_1=1$ in this step. Moreover, if $T_+<\infty$, then by \eqref{eq:lambda1 zero}, after increasing $t_1^*$ if necessary, we may assume that $\lambda_1(t)<1$ on $[t_1^*,T_+)$.
	Define
	\begin{equation}\label{eq:Phi def}
		\Phi_R(t)\coloneqq (\chi_R(\Lambda W)_{\underline{\lambda_1(t)}},g(t)).
	\end{equation}
	Since $\chi_R$ is time-independent, differentiation and \eqref{eq:u equation} yield
	\begin{equation}\label{eq:PhiR equ}
		\partial_t\Phi_R-c_*\frac{\lambda_{1,t}}{\lambda_1}
		=(\mathrm I_R-\mathrm {II}_R)\frac{\lambda_{1,t}}{\lambda_1}
		+\mathrm {III}_R+E_R,
	\end{equation}
	where
	\begin{equation*}
		\begin{aligned}
			\mathrm I_R\coloneqq\,&
			((\Lambda W)_{\underline{\lambda_1}},\chi_R(\Lambda W)_{\lambda_1})-c_*,\\
            \mathrm {II}_R\coloneqq\,&
			(\chi_R(\Lambda_{-1}\Lambda W)_{\underline{\lambda_1}},g),\qquad\qquad
			\mathrm {III}_R \coloneqq
			([\Delta,\chi_R](\Lambda W)_{\underline{\lambda_1}},g),\\
			E_R\coloneqq\,& (\chi_R[f'(P)-f'(\iota_1Q_{[\lambda_1]})](\Lambda W)_{\underline{\lambda_1}},r^{-2}g)-(\chi_R(\Lambda W)_{\underline{\lambda_1}},\calI)
			\\
			&+(\chi_R(\Lambda W)_{\underline{\lambda_1}},\NL_U(g))
			+(\chi_R(\Lambda W)_{\underline{\lambda_1}},\Psi)
			-\sum_{j=2}^N\frac{\lambda_{j,t}}{\lambda_j}(\chi_R(\Lambda W)_{\underline{\lambda_1}}, \lambda_j\partial_{\lambda_j}U).
		\end{aligned}
	\end{equation*}
	Here, we used
	\begin{equation}\label{eq:Hu identity}
		\begin{aligned}
			&-(\chi_R(\Lambda W)_{\underline{\lambda_1}},H_Ug)\\
            &=([\Delta,\chi_R](\Lambda W)_{\underline{\lambda_1}},g)
			+(\chi_R[f'(P)-f'(\iota_1Q_{[\lambda_1]})](\Lambda W)_{\underline{\lambda_1}},r^{-2}g).
		\end{aligned}
	\end{equation}
	We first estimate $E_R$.
	Since $0\leq\chi_R\leq1$ and $|\Lambda W|\lesssim\rho$, we obtain from \eqref{eq:rho overlap} and \eqref{eq:Psi weighted}
	\begin{equation*}
		|(\chi_R(\Lambda W)_{\underline{\lambda_1}},\calI)|
		+|(\chi_R(\Lambda W)_{\underline{\lambda_1}},\Psi)|
		\lesssim\calD.
	\end{equation*}
	Thanks to \eqref{eq:potential interaction} and \eqref{eq:sobolev}, we have
	\begin{equation*}
		|(\chi_R[f'(P)-f'(\iota_1Q_{[\lambda_1]})](\Lambda W)_{\underline{\lambda_1}},r^{-2}g )|
		\lesssim \mu_2 d_2\|g\|_{\dot H^2}
		\lesssim \calD+\|g\|_{\dot H^2}^2.
	\end{equation*}
	Moreover, $(\Lambda W)_{\underline{\lambda_1}}(r)\lesssim r^{-4}$, so \eqref{eq:NL esti} and \eqref{eq:sobolev} yield
	\begin{equation*}
		|(\chi_R(\Lambda W)_{\underline{\lambda_1}},\NL_U(g))| \lesssim \|r^{-2}g\|_{L^2}^2 \lesssim\|g\|_{\dot H^2}^2.
	\end{equation*}
	When $N=1$, the remaining parameter terms vanish. If $N\geq2$, \eqref{eq:weighted lambda t esti} and \eqref{eq:rho overlap} give
	\begin{equation*}
		\sum_{j=2}^N\bigg|\frac{\lambda_{j,t}}{\lambda_j}(\chi_R(\Lambda W)_{\underline{\lambda_1}},\lambda_j\partial_{\lambda_j}U)\bigg|
		\lesssim\calD+\|g\|_{\dot H^2}^2.
	\end{equation*}
	Thus,
	\begin{equation}\label{eq:mod esti uniform error}
		|E_R(t)| \lesssim\calD+\|g\|_{\dot H^2}^2\quad \text{uniformly in } R.
	\end{equation}

	On the other hand, direct integration gives
	\begin{equation}\label{eq:I esti}
		|\mathrm{I}_R|=|((\Lambda W)_{\underline{\lambda_1}},\chi_R(\Lambda W)_{\lambda_1})-c_*|
		\lesssim\min\{\lambda_1^2/R^2,1\}.
	\end{equation}
	The decay in \eqref{eq:unified kernel formulas} and \eqref{eq:norm scaling} imply
	\begin{equation*}
		\|r\langle r/\lambda_1\rangle(\Lambda_{-1}\Lambda W)_{\underline{\lambda_1}}\|_{L^2}\lesssim1,\qquad
		\|r^2(\Lambda_{-1}\Lambda W)_{\underline{\lambda_1}}\|_{L^2}\lesssim\lambda_1.
	\end{equation*}
	Since $|\chi_R|\leq1$, Cauchy--Schwarz, \eqref{eq:sobolev}, and \eqref{eq:weighted qualitative smallness} imply, uniformly for $R\geq1$,
	\begin{equation}\label{eq:g esti}
		|\mathrm {II}_R|\lesssim\min\{\|\langle r/\lambda_1\rangle^{-1}g\|_{\dot H^1},\lambda_1\|g\|_{\dot H^2}\}
		\lesssim\min\{1,\lambda_1\|g\|_{\dot H^2}\}.
	\end{equation}
	Applying \eqref{eq:lambda t esti} and \eqref{eq:g esti}, we obtain, uniformly for $R\geq1$,
	\begin{equation}\label{eq:mod esti first I 12}
		\begin{aligned}
			|\mathrm {I}_R|\bigg|\frac{\lambda_{1,t}}{\lambda_1}\bigg|
			\lesssim R^{-1}\|g\|_{\dot H^2}+\calD+\|g\|_{\dot H^2}^2,
			\qquad
			|\mathrm {II}_R|\bigg|\frac{\lambda_{1,t}}{\lambda_1}\bigg|
			\lesssim\calD+\|g\|_{\dot H^2}^2.
		\end{aligned}
	\end{equation}
	For the last term $\mathrm {III}_R$, we get
	\begin{equation}\label{eq:mod esti I 3}
		|\mathrm {III}_R| \lesssim \|r^{-2}g\|_{L^2} \|r^2[\Delta,\chi_R] (\Lambda W)_{\underline{\lambda_1}}\|_{L^2} \lesssim R^{-1}\|g\|_{\dot H^2}.
	\end{equation}
	Collecting \eqref{eq:mod esti uniform error}, \eqref{eq:mod esti first I 12}, and \eqref{eq:mod esti I 3}, we conclude \eqref{eq:first mod esti}.

	For \eqref{eq:localized correction bound}, we assume $T_+<\infty$ and take $R=1$. Then, from \eqref{eq:local qualitative smallness} with $r_0=1$ and \eqref{eq:sobolev}, we derive
	\begin{align*}
		| ((\Lambda W)_{\underline{\lambda_1}}, g\chi) |
		&\leq \|r^{-1}g\chi\|_{L^2}
		\|r(\Lambda W)_{\underline{\lambda_1}} \mathbf1_{\{r<2\}}\|_{L^2}
		\lesssim(1+\log (\lambda_1^{-1}))^{1/2}.
	\end{align*}

	\textbf{Step 4.} Now, we show \eqref{eq:frak b esti} and \eqref{eq:relation modulation law}.
	Assume that $N\geq2$, and fix $k\in\{2,\ldots,N\}$. For this step, set
	\begin{equation*}
		\Theta_k\coloneqq\tfrac{\iota_k}{c_*}\{(\Lambda W)_{\underline{\lambda_k}}-(\Lambda W)_{\underline{\lambda_{k-1}}}\}.
	\end{equation*}
	Then $\mathfrak b_k=(\Theta_k,g)$. From \eqref{eq:endpoint kernel overlap}, we have
	\begin{equation}\label{eq:Theta normal}
		(\Theta_k,\iota_k(\Lambda W)_{\lambda_k})=1+O(\mu_k^2),\qquad (\Theta_k,\iota_{k-1}(\Lambda W)_{\lambda_{k-1}})=\varsigma\iota_{k-1}\iota_k+o(1).
	\end{equation}
	Moreover, a direct computation yields
	\begin{equation}\label{eq:relative scaling overlap bounds}
		|(\Theta_k,(\Lambda W)_{\lambda_j})|
		\lesssim
		\begin{cases}
			(\lambda_{k-1}/\lambda_j)^2(1+\log(\lambda_j/\lambda_{k-1})),&j<k-1,\\
			(\lambda_j/\lambda_k)^2,&j>k.
		\end{cases}
	\end{equation}
	By \eqref{eq:global H1 smallness}, $\|g(t)\|_{\dot H^1}\to0$.
	The cancellation in $(\Lambda W)_{\underline{\lambda_k}}-(\Lambda W)_{\underline{\lambda_{k-1}}}$ gives
	\begin{equation}\label{eq:relative test weighted norm}
		\|r[(\Lambda W)_{\underline{\lambda_k}}-(\Lambda W)_{\underline{\lambda_{k-1}}}]\|_{L^2}^2\lesssim1+|\log\mu_k|.
	\end{equation}
	Using \eqref{eq:sobolev}, \eqref{eq:relative test weighted norm}, and \eqref{eq:global H1 smallness}, we obtain \eqref{eq:frak b esti}.

	Since $g_t=u_t+\sum_j(\lambda_{j,t}/\lambda_j)\iota_j(\Lambda W)_{\lambda_j}$, differentiating $\mathfrak b_k=(\Theta_k,g)$ in time and using \eqref{eq:u equation} and \eqref{eq:Theta normal}, we obtain
	\begin{equation}\label{eq:relative modulation identity}
		\begin{aligned}
			\frac d{dt}&(\log\lambda_k+\varsigma\iota_{k-1}\iota_k\log\lambda_{k-1}-\mathfrak b_k)-\iota_{k-1}\iota_k\frac{\kappa}{\lambda_{k-1}^2}\\
			=&-\sum_{j\notin\{k-1,k\}}\frac{\lambda_{j,t}}{\lambda_j}(\Theta_k,\iota_j(\Lambda W)_{\lambda_j})
			+(\Theta_k,\calI-\iota_{k-1}\kappa\lambda_{k-1}^{-2}(\Lambda W)_{\lambda_k})\\
			&+(\Theta_k,H_Ug)-(\Theta_k,\NL_U(g))-(\Theta_k,\Psi)-(\partial_t\Theta_k,g)+M,
		\end{aligned}
	\end{equation}
	where
	\begin{align*}
		M=&\frac{\lambda_{k,t}}{\lambda_k}\{1-(\Theta_k,\iota_k(\Lambda W)_{\lambda_k})\}
		-\frac{\lambda_{k-1,t}}{\lambda_{k-1}}\{(\Theta_k,\iota_{k-1}(\Lambda W)_{\lambda_{k-1}})-\varsigma\iota_{k-1}\iota_k\}\\
		&+\iota_{k-1}\frac{\kappa}{\lambda_{k-1}^2}\{(\Theta_k,(\Lambda W)_{\lambda_k})-\iota_k\}.
	\end{align*}
	By \eqref{eq:Theta normal}, $\mu_k^2\lambda_{k-1}^{-2}=d_k^2$, and \eqref{eq:lambda t esti}, we obtain
	\begin{equation}\label{eq:relative modulation normalization error}
		|M|
		\lesssim\frac{o_t(1)}{\lambda_{k-1}^2}+\mathbf1_{\{k\geq3\}}\frac{o_t(1)}{\lambda_{k-2}^2}+\frac{\|g\|_{\dot H^2}}{\lambda_{k-1}}+\calD+\|g\|_{\dot H^2}^2.
	\end{equation}
	We estimate the remaining terms. Using \eqref{eq:relative scaling overlap bounds} and \eqref{eq:lambda t esti} as in the proof of \eqref{eq:lambda element}, we obtain
	\begin{equation}\label{eq:relative modulation off diagonal}
		\begin{aligned}
			&\bigg|\sum_{j\notin\{k-1,k\}}\frac{\lambda_{j,t}}{\lambda_j}(\Theta_k,\iota_j(\Lambda W)_{\lambda_j})\bigg|
			\\
			&\lesssim\sum_{j=1}^{k-2}\bigg(\frac{\lambda_{k-1}}{\lambda_j}\bigg)^2\bigg(1+\log\frac{\lambda_j}{\lambda_{k-1}}\bigg)\bigg|\frac{\lambda_{j,t}}{\lambda_j}\bigg|
			+\sum_{j=k+1}^N\bigg(\frac{\lambda_j}{\lambda_k}\bigg)^2\bigg|\frac{\lambda_{j,t}}{\lambda_j}\bigg|
			\\
			&\lesssim \lambda_{k-2}^{-2}\mathbf1_{\{k\geq3\}}+{\lambda_{k-1}}^{-1}\|g\|_{\dot H^2}+\calD+\|g\|_{\dot H^2}^2.
		\end{aligned}
	\end{equation}
	By \eqref{eq:def calI} and \eqref{eq:relative scaling overlap bounds},
	\begin{equation}\label{eq:relative modulation interaction}
		|(\Theta_k,\calI-\iota_{k-1}\kappa\lambda_{k-1}^{-2}(\Lambda W)_{\lambda_k})|\lesssim \lambda_{k-2}^{-2}\mathbf1_{\{k\geq3\}}+\calD.
	\end{equation}
	Using \eqref{eq:Hu identity}, \eqref{eq:potential interaction}, and \eqref{eq:sobolev}, we obtain, for $k^\prime \in\{k-1,k\}$,
	\begin{equation*}
		|((\Lambda W)_{\underline{\lambda_{k^\prime}}},H_Ug)|\lesssim\|g\|_{\dot H^2}(\lambda_{k^\prime-1}^{-1}\mathbf1_{\{k^\prime\geq2\}}+d_{k^\prime+1}).
	\end{equation*}
	Therefore
	\begin{equation}\label{eq:relative modulation linear}
		|(\Theta_k,H_Ug)|\lesssim {\lambda_{k-1}}^{-1}\|g\|_{\dot H^2}+\calD+\|g\|_{\dot H^2}^2.
	\end{equation}
	Moreover, $|\Theta_k(r)|\lesssim r^{-4}$ and \eqref{eq:NL esti} provide
	\begin{equation}\label{eq:relative modulation nonlinear}
		|(\Theta_k,\NL_U(g))|\lesssim\|g\|_{\dot H^2}^2.
	\end{equation}
	By \eqref{eq:Psi kernel},
	\begin{equation}\label{eq:relative modulation profile error}
		|(\Theta_k,\Psi)|\lesssim o_{t}(1){\lambda_{k-1}^{-2}}+{\lambda_{k-2}^{-2}}\mathbf1_{\{k\geq3\}}+\calD.
	\end{equation}

	It remains to estimate $(\partial_t\Theta_k,g)$. We have
	\begin{equation*}
		\partial_t\Theta_k=\frac{\iota_k}{c_*}\bigg\{-\frac{\lambda_{k,t}}{\lambda_k}(\Lambda_{-1}\Lambda W)_{\underline{\lambda_k}}+\frac{\lambda_{k-1,t}}{\lambda_{k-1}}(\Lambda_{-1}\Lambda W)_{\underline{\lambda_{k-1}}}\bigg\}.
	\end{equation*}
	By \eqref{eq:sobolev}, \eqref{eq:norm scaling}, and \eqref{eq:global H1 smallness}, we obtain
	\begin{equation}\label{eq:relative test radiation}
		\begin{aligned}
			|((\Lambda_{-1}\Lambda W)_{\underline\lambda},g)|\lesssim
			\min(\|g\|_{\dot H^1}=o_{t}(1), \lambda\|g\|_{\dot H^2}).
		\end{aligned}
	\end{equation}
	By \eqref{eq:relative test radiation}, \eqref{eq:lambda t esti}, and \eqref{eq:global H1 smallness}, we obtain
	\begin{equation*}
		\sum_{j=k-1}^k\bigg|\frac{\lambda_{j,t}}{\lambda_j}((\Lambda_{-1}\Lambda W)_{\underline{\lambda_j}},g)\bigg|\lesssim o_{t}(1){\lambda_{k-1}^{-2}}+o_t(1){\lambda_{k-2}^{-2}}\mathbf1_{\{k\geq3\}} +\calD+\|g\|_{\dot H^2}^2,
	\end{equation*}
	which yields
	\begin{equation}\label{eq:relative test time error}
		|(\partial_t\Theta_k,g)|\lesssim o_{t}(1){\lambda_{k-1}^{-2}}+{\lambda_{k-2}^{-2}}\mathbf1_{\{k\geq3\}}+\calD+\|g\|_{\dot H^2}^2.
	\end{equation}
	Substituting \eqref{eq:relative modulation normalization error}--\eqref{eq:relative test time error} into \eqref{eq:relative modulation identity}, we obtain \eqref{eq:relation modulation law}. This yields the proof.
\end{proof}

\subsection{Proof of Theorems~\ref{thm:hmhf main} and~\ref{thm:nlh main}}
It remains to complete the classification. We first prove global existence, determine the bubble signs, and identify the asymptotic behavior of the largest scale $\lambda_1$. The remaining scale dynamics are then obtained by integrating the coupled modulation law \eqref{eq:relation modulation law}. Throughout, we use the modulation estimates established in Proposition~\ref{prop:modulation}. To determine the asymptotic law for $\lambda_1$ from the initial tail, we adapt the argument of Gustafson--Nakanishi--Tsai \cite[Section~9]{GustafsonNakanishiTsai2010CMP}.

\begin{prop}[Globality]\label{prop:globality}
	For the solutions satisfying the assumptions of Theorem~\ref{thm:hmhf main} for \eqref{eq:hmhf 2} and Theorem~\ref{thm:nlh main} for \eqref{eq:NLH 6}, we have $T_+=\infty$.
\end{prop}
\begin{proof}
	Assume $T_+<\infty$. Applying the sign symmetry if necessary, assume $\iota_1=1$. Integrating \eqref{eq:first mod esti} on $[t_1^*,t]$ with $R=1$ and \eqref{eq:spacetime}, we obtain
	\begin{equation*}
		\bigg|c_*\log\frac{\lambda_1(t)}{\lambda_1(t_1^*)}-((\Lambda W)_{\underline{\lambda_1(t)}},g(t)\chi)+((\Lambda W)_{\underline{\lambda_1(t_1^*)}},g(t_1^*)\chi)\bigg|\leq C.
	\end{equation*}
	Therefore, \eqref{eq:localized correction bound} implies
	\begin{equation*}
		(1+\log\lambda_1(t)^{-1})\lesssim 1+(1+\log\lambda_1(t)^{-1})^{1/2},
	\end{equation*}
	which contradicts $\lambda_1(t)\to0$ given by \eqref{eq:lambda1 zero}.
\end{proof}

We next integrate \eqref{eq:relation modulation law} to determine the bubble signs and prepare the analysis of the inner-scale asymptotics. Suppose $N\geq2$, and fix $T_0\geq t_1^*$ sufficiently large. For $2\leq k\leq N$, define
\begin{equation}\label{eq:tau k def}
    \tau_k(t)\coloneqq\int_{T_0}^t\frac{ds}{\lambda_{k-1}(s)^2}.
\end{equation}
Since $\lambda_{k-1}\leq\lambda_1$, \eqref{eq:lambda1 zero} implies $\tau_k(t)\to\infty$. The first two terms on the right-hand side of \eqref{eq:relation modulation law} contribute $o(\tau_k(t))$ after integration, since $\mu_{k-1}\to0$ when $k\geq3$. By \eqref{eq:spacetime}, we have
\begin{equation*}
    \int_{T_0}^t\frac{\|g(s)\|_{\dot H^2}}{\lambda_{k-1}(s)}ds
    \leq\|g\|_{L^2_t((T_0,T_+);\dot H^2)}
    \bigg(\int_{T_0}^t\frac{ds}{\lambda_{k-1}(s)^2}\bigg)^{1/2}
    =o\bigg(\int_{T_0}^t\frac{ds}{\lambda_{k-1}(s)^2}\bigg).
\end{equation*}
Finally, $\calD+\|g\|_{\dot H^2}^2$ is integrable by \eqref{eq:spacetime}. Integrating \eqref{eq:relation modulation law} and absorbing the value at $T_0$ into the error, we get
\begin{equation}\label{eq:integrated interface estimate}
    \log\lambda_k(t)+\varsigma\iota_{k-1}\iota_k\log\lambda_{k-1}(t)-\mathfrak b_k(t)=(\iota_{k-1}\iota_k\kappa+o_t(1))\int_{T_0}^t\frac{ds}{\lambda_{k-1}(s)^2}.
\end{equation}

\begin{prop}[Sign rigidity]\label{lem:sign rigidity}
	If $N\geq2$, then
	\begin{equation*}
		\iota_{k-1}\iota_k=\varsigma,\qquad 2\leq k\leq N.
	\end{equation*}
\end{prop}
\begin{proof}
	Fix $k\in\{2,\ldots,N\}$. If $\iota_{k-1}\iota_k=-\varsigma$, then, by $\varsigma\kappa<0$ and \eqref{eq:integrated interface estimate}, we get
	\begin{equation*}
		\log\mu_k(t)-\mathfrak b_k(t)=|\kappa|\tau_k(t)+o(\tau_k(t))\to\infty.
	\end{equation*}
	On the other hand, \eqref{eq:lambda1 zero} and \eqref{eq:frak b esti} yield a contradiction since
	\begin{equation*}
		\log\mu_k(t)-\mathfrak b_k(t)=-|\log\mu_k(t)|+o(|\log\mu_k(t)|)\to-\infty. \qedhere
	\end{equation*}
\end{proof}

We next determine the asymptotic law for $\lambda_1$ from the initial tail. For $F:(0,\infty)\to\bbR$, define
\begin{equation*}
	\Delta_2^{(1)}F\coloneqq F_{rr}+\frac1rF_r-\frac1{r^2}F.
\end{equation*}
This is the order-one Bessel operator on $\bbR^2$. More precisely, if $F(r)/r$ is regarded as a radial profile on $\bbR^4$, then
\begin{equation*}
	\Delta_{\bbR^2}(F(r)e^{i\theta} )=(\Delta_2^{(1)}F )(r)e^{i\theta},
	\qquad
	\Delta_2^{(1)}F=r\Delta_{\bbR^4}(r^{-1}F).
\end{equation*}

For the asymptotic of the largest scale $\lambda_1$, we use the sign symmetries of the two equations. In the \eqref{eq:hmhf 2} case, we work in the normalized class $v_0\in\calE_{0,N}$; the general case will be reduced to this one by the symmetry $v\mapsto\iota(v-\ell\pi)$. In the \eqref{eq:NLH 6} case, we apply $u\mapsto-u$ if necessary so that $\iota_1=1$, and retain the notation $u,u_0$ for the normalized solution and its initial datum.

For \eqref{eq:hmhf 2}, define the Bogomol'nyi operator
\begin{equation}\label{eq:bfD definition}
	\bfD(t,r)\coloneqq\partial_r v+\frac{2\sin(v-N\pi)}r
	=\partial_r v-\frac{2\sin(v-(N-1)\pi)}r.
\end{equation}

\begin{prop}[Largest-scale asymptotic]\label{prop:lambda1 asym}
	For \eqref{eq:hmhf 2}, let $N\geq1$ and $v_0\in\calE_{0,N}$, and let $v$ be the corresponding global solution. Let $\blambda=(\lambda_1,\ldots,\lambda_N)$ be the $C^1$ scales given by Proposition~\ref{prop:decomposition solution}. Then there exists $L>0$ such that
	\begin{equation}\label{eq:lambda1 law}
		\log\lambda_1(t)
		=-\frac1\pi\int_1^{\sqrt t}\bfD(0,r)\,dr+\log L+o_t(1).
	\end{equation}

    For \eqref{eq:NLH 6}, let $u$ be a global solution satisfying \eqref{eq:NLH bounded}, and suppose that the decomposition \eqref{eq:NLH global decomposition} contains $N\geq1$ bubbles with $\iota_1=1$. Let $\blambda=(\lambda_1,\ldots,\lambda_N)$ be the $C^1$ scales given by Proposition~\ref{prop:decomposition solution}. Then there exists $L>0$ such that
	\begin{equation}\label{eq:NLH lambda1 fixed initial law}
		\log\lambda_1(t)
		=\frac5{16}\int_1^{\sqrt t}u_0(r)r\,dr+\log L+o_t(1).
	\end{equation}
\end{prop}

\begin{proof}
    For the proof, we follow the approach of Gustafson--Nakanishi--Tsai \cite[Section~9]{GustafsonNakanishiTsai2010CMP}, adapting it to the two models considered here.
    
    Fix $T_0\geq t_1^*$. Recall the definition of $\Phi_R$ in \eqref{eq:Phi def}.
    For $R\geq1$ and $T_0\leq t_1<t_2$, integrating \eqref{eq:first mod esti} on $[t_1,t_2]$, we obtain
	\begin{equation}\label{eq:lambda1 finite increment}
		\begin{aligned}
			\bigg|c_*\log\frac{\lambda_1(t_2)}{\lambda_1(t_1)}
			-\Phi_R(t_2)+\Phi_R(t_1)\bigg|
			&\lesssim \frac{1}{R}\int_{t_1}^{t_2}\|g\|_{\dot H^2}ds+\int_{t_1}^{t_2}(\calD+\|g\|_{\dot H^2}^2)ds.
		\end{aligned}
	\end{equation}
	By \eqref{eq:spacetime}, the first term on the right tends to zero as $R\to\infty$ for fixed $t_1,t_2$.

    \textbf{Step 1.} 
    From \eqref{eq:global H1 smallness}, we choose $T_0$ sufficiently large that $\sup_{t\geq T_0}\|g(t)\|_{\dot H^1}\leq1$.
    Set $\td\chi_R\coloneqq1-\chi_R$, and
    \begin{equation}\label{eq:universal const}
        \{\ell_*, c_*\}\coloneqq\{\lim_{r\to\infty}r^4\Lambda W(r), c_*\}=
        \begin{cases}
            \{4,2\pi\},&\text{for \eqref{eq:hmhf 2}},\\
            \{-\frac5{16}\|\Lambda W\|_{L^2}^2, \|\Lambda W\|_{L^2}^2\},&\text{for \eqref{eq:NLH 6}}.
        \end{cases}
    \end{equation}
    Define
    \begin{equation*}
    	\calG\coloneqq\calB+r^2g,\qquad \calB\coloneqq
    	\begin{cases}
    		\sum_{j=2}^N(Q_{[\lambda_j]}-\pi),&\text{for \eqref{eq:hmhf 2}},\\[0.5em]
    		\sum_{j=2}^N\iota_jQ_{[\lambda_j]},&\text{for \eqref{eq:NLH 6}}.
    	\end{cases}
    \end{equation*}
    From \eqref{eq:Q f}, we have
    \begin{equation}\label{eq:inner bubble exterior tail}
    	|\calB|+r|\partial_r\calB|+r^2|\partial_{rr}\calB|
        \lesssim\sum_{j=2}^N\frac{\lambda_j^2}{r^2},\qquad r\geq\lambda_1.
    \end{equation}
    Note that, for \eqref{eq:hmhf 2}, Proposition~\ref{lem:sign rigidity} and $v_0\in\calE_{0,N}$ give $\iota_1=\cdots=\iota_N=1$. Thus \eqref{eq:U explicit} yields
    \begin{equation}\label{eq:exterior decomposition}
    	\begin{cases}
    		v-(N-1)\pi=Q_{[\lambda_1]}+\calG,&\text{for \eqref{eq:hmhf 2}},\\
    		v=Q_{[\lambda_1]}+\calG,&\text{for \eqref{eq:NLH 6}},
    	\end{cases}
    \end{equation}
    where $v=r^2u$ in the second case. For \eqref{eq:hmhf 2}, retain \eqref{eq:bfD definition}; for \eqref{eq:NLH 6}, define
    \begin{equation}\label{eq:NLH defect definition}
    	\bfD\coloneqq\calA \calG=\calA (v-Q_{[\lambda_1]}),\qquad \calA \coloneqq\partial_r+\frac2r.
    \end{equation}
    In both cases, set $\bfD_{\rm e}\coloneqq\td\chi_{\lambda_1}\bfD$. When $N=1$, we use $\calB=0$ and $\mu_2=d_2=0$.
    
    For the rest of the proof, set
    \begin{equation}\label{eq:X def}
    	X(t)\coloneqq\|g(t)\|_{\dot H^2}+d_2(t)+|\lambda_{1,t}(t)|\in L^2_t(T_0,\infty),
    \end{equation}
    where the integrability follows from \eqref{eq:spacetime} and \eqref{eq:physical scale velocity}. On $r\geq\lambda_1$,
    \begin{equation*}
    	\bfD=
    	\begin{cases}
    		\partial_r\calG-\frac2r (\sin(Q_{[\lambda_1]}+\calG)-\sin Q_{[\lambda_1]} ),&\text{for \eqref{eq:hmhf 2}},\\
    		\calA \calG,&\text{for \eqref{eq:NLH 6}},
    	\end{cases}
    \end{equation*}
    and $|\partial_rQ_{[\lambda_1]}|\lesssim r^{-1}$. 
    By \eqref{eq:sobolev} and \eqref{eq:inner bubble exterior tail},
    \begin{equation}\label{eq:bfD ext spacetime}
    	\begin{aligned}
    		&\|\mathbf1_{\{r\geq\lambda_1\}}r^{-1}\bfD\|_{L^2_2}+\|\mathbf1_{\{r\geq\lambda_1\}}\partial_r\bfD\|_{L^2_2}\\
    		&\qquad+\|\mathbf1_{\{r\geq\lambda_1\}}r^{-2}\calG\|_{L^2_2}+\|\mathbf1_{\{r\geq\lambda_1\}}\partial_r(r^{-1}\calG)\|_{L^2_2}\lesssim X(t),\\
    		&\|\mathbf1_{\{r\geq\lambda_1\}}\bfD\|_{L^2_2}+\|\mathbf1_{\{r\geq\lambda_1\}}\calG\|_{L^\infty}\lesssim\|g\|_{\dot H^1}+\mu_2^2=o_t(1).
    	\end{aligned}
    \end{equation}
    By \eqref{eq:bfD ext spacetime} and the Cauchy--Schwarz inequality, we get
    \begin{equation}\label{eq:bfD Linf}
    	\sup_{r\geq\lambda_1} (|\bfD(t,r)|+|r^{-1}\calG(t,r)| )\lesssim X(t).
    \end{equation}
    
    We next compare $\Phi_R$ with $(\ell_*/2)\int_0^\infty\chi_R\bfD_{\rm e}dr$. In both models, we obtain
    \begin{equation}\label{eq:bfD taylor}
    	|\bfD-\calA (r^2g)|\lesssim\sum_{j=2}^N\frac{\lambda_j^2}{r^3}+\frac{\lambda_1^4}{r^5}|r^2g|+\frac{(r^2g)^2}{r},\qquad r\geq\lambda_1.
    \end{equation}
    Indeed, for \eqref{eq:NLH 6}, we have $\bfD-\calA(r^2g)=\calA\calB$. For \eqref{eq:hmhf 2}, we use
    \begin{equation*}
    	\sin(Q_{[\lambda_1]}+\calG)-\sin Q_{[\lambda_1]}+r^2g
    	=\cos Q_{[\lambda_1]}\calB
    	+(1+\cos Q_{[\lambda_1]})r^2g
    	+O(\calG^2)
    \end{equation*}
    together with \eqref{eq:inner bubble exterior tail} and $1+\cos Q_{[\lambda_1]}\lesssim\lambda_1^4r^{-4}$.
    By \eqref{eq:unified kernel formulas},
    \begin{equation*}
    	\int_0^\infty |r^3(\Lambda W)_{\underline{\lambda_1}}-\ell_*r^{-1}\td\chi_{\lambda_1} |\,dr\lesssim1.
    \end{equation*}
    For $R>2\lambda_1$, integration by parts yields
    \begin{align*}
    	\Phi_R-\frac{\ell_*}{2}\int_0^\infty\chi_R\bfD_{\rm e}dr
    	=&\int_0^\infty\chi_R(r^2g) (r^3(\Lambda W)_{\underline{\lambda_1}}-\ell_*r^{-1}\td\chi_{\lambda_1} )dr\\
    	&-\frac{\ell_*}{2}\int_0^\infty\chi_R\td\chi_{\lambda_1} (\bfD-\calA (r^2g) )dr\\
    	&+\frac{\ell_*}{2}\int_0^\infty\partial_r(\chi_R\td\chi_{\lambda_1})(r^2g)dr.
    \end{align*}
    By \eqref{eq:sobolev} and \eqref{eq:bfD taylor}, all terms except the one containing $\partial_r\chi_R$ converge absolutely as $R\to\infty$ and are bounded by $\|g\|_{\dot H^1}+\mu_2^2+\|g\|_{\dot H^1}^2$. The remaining term satisfies
    \begin{equation*}
    	\bigg|\int_0^\infty\partial_r\chi_R\td\chi_{\lambda_1}(r^2g)\,dr\bigg|\lesssim\|\mathbf1_{\{r\geq R\}}rg(t)\|_{L^2_2}\to0\qquad\text{as }R\to\infty.
    \end{equation*}
    Consequently, the limit
    \begin{equation*}
    	\eta_1(t)\coloneqq\lim_{R\to\infty}\bigg(\Phi_R(t)-\frac{\ell_*}{2}\int_0^\infty\chi_R(r)\bfD_{\rm e}(t,r)\,dr\bigg)
    \end{equation*}
    is well-defined and satisfies
    \begin{equation}\label{eq:correction zero}
    	|\eta_1(t)|\lesssim\|g(t)\|_{\dot H^1}+\mu_2(t)^2+\|g(t)\|_{\dot H^1}^2=o_t(1).
    \end{equation}
    
    \textbf{Step 2.} In both models, the exterior variable satisfies
    \begin{equation}\label{eq:bfD ext equ}
    	(\partial_t-\Delta_2^{(1)})\bfD_{\rm e}=F_{\rm e},\qquad F_{\rm e}\coloneqq\td\chi_{\lambda_1}F+(\partial_t\td\chi_{\lambda_1})\bfD-[\Delta_2^{(1)},\td\chi_{\lambda_1}]\bfD,
    \end{equation}
    where
    \begin{equation*}
    	F=
    	\begin{cases}
    		-4r^{-2}(1-\cos(v-N\pi))\bfD+2r^{-1}\sin(v-N\pi)\bfD^2,&\text{for \eqref{eq:hmhf 2}},\\[2mm]
            r^{-2}\partial_r (|v|v-Q_{[\lambda_1]}^2 )+\frac{\lambda_{1,t}}{\lambda_1^2} (\calA (r\partial_rQ) )_{[\lambda_1]},&\text{for \eqref{eq:NLH 6}}.
    	\end{cases}
    \end{equation*}
    Now, we derive some estimates for $F_{\rm e}$ that will be used in both models. The cutoff derivatives are supported on $\lambda_1\leq r\leq2\lambda_1$ and satisfy
    \begin{equation}\label{eq:exterior cutoff bounds}
        |\partial_t\td\chi_{\lambda_1}|\lesssim\lambda_1^{-1}|\lambda_{1,t}|, \qquad|[\Delta_2^{(1)},\td\chi_{\lambda_1}]\bfD|
        \lesssim\lambda_1^{-1}|\partial_r\bfD|+\lambda_1^{-2}|\bfD|.
    \end{equation}
    For \eqref{eq:hmhf 2}, \eqref{eq:inner bubble exterior tail} and \eqref{eq:exterior decomposition} imply
    \begin{equation}\label{eq:HMHF exterior nonlinear bounds}
    	1-\cos(v-N\pi)\lesssim\langle r/\lambda_1\rangle^{-4}+(r^2g)^2,\qquad r\geq\lambda_1.
    \end{equation}
    For \eqref{eq:NLH 6}, the explicit profile gives
    \begin{equation}\label{eq:NLH exterior nonlinear bounds}
    	Q_{[\lambda_1]}+r|\partial_rQ_{[\lambda_1]}|\lesssim\langle r/\lambda_1\rangle^{-2},\qquad |\calA (r\partial_rQ)(r)|\lesssim r\langle r\rangle^{-6}.
    \end{equation}
    Since $v=Q_{[\lambda_1]}+\calG$ and $\partial_r\calG=\bfD-2r^{-1}\calG$, we have
    \begin{equation*}
    	\partial_r (|v|v-Q_{[\lambda_1]}^2 )=2|v|\bfD-4r^{-1}|v|\calG+2 (|v|-Q_{[\lambda_1]} )\partial_rQ_{[\lambda_1]},
    \end{equation*}
    and therefore
    \begin{equation}\label{eq:NLH exterior nonlinear derivative}
    	r^{-2} |\partial_r (|v|v-Q_{[\lambda_1]}^2 ) |\lesssim r^{-2}Q_{[\lambda_1]}|\bfD|+r^{-3}Q_{[\lambda_1]}|\calG|+r^{-2}|\calG||\bfD|+r^{-3}|\calG|^2.
    \end{equation}
    Using \eqref{eq:sobolev}, \eqref{eq:bfD ext spacetime}, \eqref{eq:bfD Linf}, \eqref{eq:exterior cutoff bounds}, \eqref{eq:HMHF exterior nonlinear bounds}, \eqref{eq:NLH exterior nonlinear bounds}, and \eqref{eq:NLH exterior nonlinear derivative}, we obtain, in both models,
    \begin{equation*}
    	\|F_{\rm e}(t)\|_{L^1(dr)\cap L^2_2}\lesssim\lambda_1(t)^{-1}X(t)+X(t)^2.
    \end{equation*}
    Since $\lambda_1^{-1}$ is bounded on every compact time interval $J=[T_0,S]$, \eqref{eq:X def} implies
    \begin{equation}\label{eq:F ext regularity}
    	F_{\rm e}\in L^1_t(J;L^1(dr))\cap L^1_t(J;L^2_2).
    \end{equation}
    Local well-posedness and the $C^1$ regularity of $\lambda_1$ also give $\bfD_{\rm e}\in C(J;L^2_2)$. Moreover, for any nonnegative weight $K$ satisfying $0\leq K(r)\leq r^{-1}$, a similar argument implies
    \begin{equation}\label{eq:exterior kernel}
    	(|F_{\rm e}(t)|,K)_2\lesssim X(t)\bigg(\int_0^\infty K(r)^2\langle r/\lambda_1(t)\rangle^{-4}\frac{dr}{r}\bigg)^{1/2}+X(t)^2.
    \end{equation}
    
    \textbf{Step 3.}
    For $a>0$, set $k_a(r)\coloneqq r^{-1}\mathbf1_{\{r>a\}}$. For $M>a$,
    \begin{equation}\label{eq:exterior weight formula}
    	(k_a-k_M,f)_2=\int_a^M f(r)dr,\qquad \|k_a-k_M\|_{L^2_2}^2=\log\frac Ma.
    \end{equation}
    Since $k_a\notin L^2_2$, we first apply the Duhamel formula with the truncated weight $k_a-k_M$. Fix $t>T_0$ and write
    \begin{equation*}
    	K_{t,s}\coloneqq e^{(t-s)\Delta_2^{(1)}}k_{\lambda_1(t)},\quad T_0\leq s<t; \qquad 
    	\mathfrak A(t,\rho)\coloneqq\int_0^\rho (\bfD_{\rm e}(t,r)-\bfD_{\rm e}(T_0,r) )\,dr.
    \end{equation*}
    By \eqref{eq:bfD ext equ} and \eqref{eq:F ext regularity}, the Duhamel formula holds in $L^2_2$:
    \begin{equation*}
    	\bfD_{\rm e}(t)=e^{(t-T_0)\Delta_2^{(1)}}\bfD_{\rm e}(T_0)+\int_{T_0}^t e^{(t-s)\Delta_2^{(1)}}F_{\rm e}(s)ds.
    \end{equation*}
    For $M>\max\{\lambda_1(t),\lambda_1(T_0)\}$, taking the $L^2_2$ inner product of both sides with $k_{\lambda_1(t)}-k_M$, and using \eqref{eq:exterior weight formula} with the fact that $\bfD_{\rm e}(t,r)=0$ for $r\leq\lambda_1(t)$, we obtain
    \begin{align*}
    	\mathfrak A(t,M)
    	=\,&(e^{(t-T_0)\Delta_2^{(1)}}k_{\lambda_1(t)}-k_{\lambda_1(T_0)}-(e^{(t-T_0)\Delta_2^{(1)}}k_M-k_M),\bfD_{\rm e}(T_0))_2\\
    	&+\int_{T_0}^t(e^{(t-s)\Delta_2^{(1)}}(k_{\lambda_1(t)}-k_M),F_{\rm e}(s))_2ds.
    \end{align*}
    By \eqref{eq:exterior remote endpoint}, the term containing $e^{(t-T_0)\Delta_2^{(1)}}k_M-k_M$ tends to zero. Moreover, \eqref{eq:exterior weight pointwise} gives $0\leq re^{(t-s)\Delta_2^{(1)}}k_M(r)\leq1$ and pointwise convergence to zero as $M\to\infty$. Thus \eqref{eq:F ext regularity} and the DCT yield
    \begin{equation}\label{eq:exterior cutoff limit}
    	\begin{aligned}
    		\mathfrak A_\infty(t)&\coloneqq\lim_{M\to\infty}\mathfrak A(t,M)
    		\\
            &= (e^{(t-T_0)\Delta_2^{(1)}}k_{\lambda_1(t)}-k_{\lambda_1(T_0)},\bfD_{\rm e}(T_0) )_2 
            +\int_{T_0}^t(F_{\rm e}(s),K_{t,s})_2ds.
    	\end{aligned}
    \end{equation}
    
    We first estimate the last term. By \eqref{eq:exterior kernel} and \eqref{eq:exterior weight double} with $m=1$, we have
    \begin{equation*}
    	\int_S^t|(F_{\rm e}(s),K_{t,s})_2|\,ds\lesssim\|X\|_{L^2_s(S,\infty)} (1+\|\lambda_{1,t}\|_{L^2_s(S,\infty)} )+\|X\|_{L^2_s(S,\infty)}^2.
    \end{equation*}
    The right-hand side tends to zero as $S\to\infty$, uniformly for $t\geq S$, by \eqref{eq:X def}. On the other hand, for fixed $S>T_0$, \eqref{eq:exterior weight pointwise} gives $rK_{t,s}(r)\to0$ as $t\to\infty$ for every fixed $(s,r)\in[T_0,S]\times(0,\infty)$, with $0\leq rK_{t,s}(r)\leq1$. By \eqref{eq:F ext regularity} and DCT,
    \begin{equation*}
    	\int_{T_0}^S|(F_{\rm e}(s),K_{t,s})_2|\,ds\to0\qquad\text{as }t\to\infty.
    \end{equation*}
    Consequently,
    \begin{equation}\label{eq:bfD duhamel limit}
    	\int_{T_0}^t(F_{\rm e}(s),K_{t,s})_2\,ds=o_t(1).
    \end{equation}
    For the linear term of $\mathfrak A_\infty$, using \eqref{eq:lambda1 zero} and applying \eqref{eq:exterior endpoint asymptotic} with $f=\bfD_{\rm e}(T_0)$, $a=\lambda_1(T_0)$,
$\sigma=t-T_0$, and $b(\sigma)=\lambda_1(T_0+\sigma)$, we obtain
    \begin{equation}\label{eq:bfD linear kernel}
    	 (e^{(t-T_0)\Delta_2^{(1)}}k_{\lambda_1(t)}-k_{\lambda_1(T_0)},\bfD_{\rm e}(T_0) )_2=-\int_{\lambda_1(T_0)}^{\sqrt{t-T_0}}\bfD_{\rm e}(T_0,r)\,dr+o_t(1).
    \end{equation}
    Combining \eqref{eq:exterior cutoff limit}, \eqref{eq:bfD duhamel limit}, and \eqref{eq:bfD linear kernel}, we arrive at
    \begin{equation}\label{eq:asymptotic for mathfrak A inf}
    	\mathfrak A_\infty(t)=-\int_{\lambda_1(T_0)}^{\sqrt{t-T_0}}\bfD_{\rm e}(T_0,r)\,dr+o_t(1).
    \end{equation}
    For each fixed $t$, integration by parts also yields
    \begin{equation}\label{eq:mathfrak A converge}
        \begin{aligned}
            &\int_0^\infty\chi_R(r) (\bfD_{\rm e}(t,r)-\bfD_{\rm e}(T_0,r) )\,dr-\mathfrak A_\infty(t)\\
    	   &=-\int_R^{2R}\partial_r\chi_R(r) (\mathfrak A(t,r)-\mathfrak A_\infty(t) )\,dr\to0\quad\text{as} \quad R\to\infty.
        \end{aligned}
    \end{equation}
    Combining the preceding limit with the definition of $\eta_1$, we can define
    \begin{equation*}
    	\mathfrak a(t) \coloneqq\lim_{R\to\infty}(\Phi_R(t)-\Phi_R(T_0)).
    \end{equation*}
    Moreover, by \eqref{eq:mathfrak A converge},
    \begin{equation}\label{eq:correction goal}
    	\mathfrak a(t)=\frac{\ell_*}{2}\mathfrak A_\infty(t)+\eta_1(t)-\eta_1(T_0).
    \end{equation}
    Letting $R\to\infty$ in \eqref{eq:lambda1 finite increment} for fixed $T_0\leq t_1<t_2$, we obtain
    \begin{equation*}
    	\bigg|c_*\log\frac{\lambda_1(t_2)}{\lambda_1(t_1)}
    	-\mathfrak a(t_2)+\mathfrak a(t_1)\bigg|
    	\lesssim\int_{t_1}^{t_2}\big(\calD(s)+\|g(s)\|_{\dot H^2}^2\big)ds.
    \end{equation*}
    By \eqref{eq:spacetime}, $c_*\log\lambda_1(t)-\mathfrak a(t)$ is Cauchy as $t\to\infty$. Hence, for some $C\in\bbR$,
    \begin{equation}\label{eq:log lambda1 aux}
    	c_*\log\lambda_1(t) =\mathfrak a(t)+C+o(1).
    \end{equation}
    Combining \eqref{eq:asymptotic for mathfrak A inf}, \eqref{eq:correction goal}, \eqref{eq:log lambda1 aux}, and \eqref{eq:correction zero}, we conclude
    \begin{equation*}
    	c_*\log\lambda_1(t)=-\frac{\ell_*}{2}\int_{\lambda_1(T_0)}^{\sqrt{t-T_0}}\bfD_{\rm e}(T_0,r)\,dr+C_{T_0}+o_t(1).
    \end{equation*}
    Replacing $\bfD_{\rm e}(T_0)$ by $\bfD(T_0)$ and the lower limit by $1$ only changes the constant. Moreover,
    \begin{equation*}
    	\bigg|\int_{\sqrt{t-T_0}}^{\sqrt t}\bfD(T_0,r)\,dr\bigg|\leq\|\mathbf1_{\{r\geq\sqrt{t-T_0}\}}\bfD(T_0)\|_{L^2_2}\bigg(\frac12\log\frac{t}{t-T_0}\bigg)^{1/2}=o_t(1).
    \end{equation*}
    Hence, for some constant $C_{T_0}\in\bbR$,
    \begin{equation}\label{eq:lambda1 fixed time law}
    	c_*\log\lambda_1(t)=-\frac{\ell_*}{2}\int_1^{\sqrt t}\bfD(T_0,r)\,dr+C_{T_0}+o_t(1).
    \end{equation}
    
    For \eqref{eq:hmhf 2}, \eqref{eq:bfD definition} yields
    \begin{align*}
    	\int_1^R (\bfD(T_0,r)-\bfD(0,r) )\,dr
    	=\,&v(T_0,R)-v(0,R)-v(T_0,1)+v(0,1)\\
    	&+2\int_1^R\frac{\sin(v(T_0,r)-N\pi)-\sin(v(0,r)-N\pi)}r\,dr.
    \end{align*}
    The boundary terms at $R$ tend to zero because both maps tend to $N\pi$ at infinity. By \eqref{eq:energy identity} with $|\sin a-\sin b|\leq |a-b|$,
    \begin{equation*}
    	\int_1^\infty\frac{|\sin(v(T_0,r)-N\pi)-\sin(v(0,r)-N\pi)|}{r}dr\leq\bigg(\frac{T_0}{2}\bigg)^{1/2}\|v_t\|_{L^2_t((0,T_0);L^2_2)}<\infty.
    \end{equation*}
    It follows that
    \begin{equation*}
    	\int_1^R\bfD(T_0,r)\,dr=\int_1^R\bfD(0,r)\,dr+C_{T_0}+o_{R\to\infty}(1).
    \end{equation*}
    Substituting this into \eqref{eq:lambda1 fixed time law} and using \eqref{eq:universal const} proves \eqref{eq:lambda1 law}.
    
    For \eqref{eq:NLH 6}, the explicit formula
    \begin{equation*}
    	\calA Q_{[\lambda_1(T_0)]}(r)=\frac{4r}{\lambda_1(T_0)^2}\bigg(1+\frac{r^2}{24\lambda_1(T_0)^2}\bigg)^{-3}
    \end{equation*}
    belongs to $L^1((1,\infty);dr)$. Also, \eqref{eq:sobolev} gives $R^2u(s,R)\to0$ for $s\in\{0,T_0\}$, and hence
    \begin{equation*}
    	\int_1^R\calA (r^2u(s,r))\,dr=2\int_1^Ru(s,r)r\,dr-u(s,1)+o_{R\to\infty}(1).
    \end{equation*}
    By \eqref{eq:NLH energy identity},
    \begin{equation*}
    	\int_1^\infty|u(T_0,r)-u_0(r)|r\,dr\lesssim T_0^{1/2}\|u_t\|_{L^2_t((0,T_0);L^2)}<\infty.
    \end{equation*}
    Using \eqref{eq:NLH defect definition}, we therefore obtain
    \begin{equation*}
    	\int_1^R\bfD(T_0,r)\,dr=2\int_1^Ru_0(r)r\,dr+C_{T_0}+o_{R\to\infty}(1).
    \end{equation*}
    Substituting this into \eqref{eq:lambda1 fixed time law} and using \eqref{eq:universal const}, we complete the proof of \eqref{eq:NLH lambda1 fixed initial law}.
\end{proof}

\begin{proof}[Proof of Theorems~\ref{thm:hmhf main} and~\ref{thm:nlh main}]
	Proposition~\ref{prop:globality} gives $T_+=\infty$ for both equations. Throughout this proof, $v,u$ and $v_0,u_0$ denote the original solutions and their initial data. Thus, in the \eqref{eq:hmhf 2} case, the decomposition constructed after \eqref{eq:v shift} applies to $r^{-2}(v-\ell\pi)$. Whenever $N\geq1$, we keep the $C^1$ scales given by Proposition~\ref{prop:decomposition solution}.

	\textbf{Step 1.}
	For \eqref{eq:hmhf 2}, the boundary values in Proposition~\ref{thm:HMHF sol resol} yield ${\sum_{j=1}^N}\iota_j=m-\ell$.
	If $N\geq2$, Proposition~\ref{lem:sign rigidity} shows that all signs agree; this is immediate when $N=1$. Consequently, we have
	\begin{equation*}
		N=|m-\ell|,\qquad \iota_j=\iota=\sgn(m-\ell),\quad 1\leq j\leq N,
	\end{equation*}
	which proves \eqref{eq:main number sign}. If $N=0$, Proposition~\ref{thm:HMHF sol resol} gives $\|v(t)-\ell\pi\|_{\calE}\to0$, and there are no scale laws to prove. If $N\geq1$, \eqref{eq:refined decomposition}, \eqref{eq:U explicit}, and \eqref{eq:energy transform} imply
	\begin{equation*}
		\bigg\|v(t)-\ell\pi-\iota\sum_{j=1}^NQ_{[\lambda_j(t)]}\bigg\|_{\calE}\lesssim \|g(t)\|_{\dot H^1}\to0
	\end{equation*}
	by \eqref{eq:global H1 smallness}. This proves \eqref{eq:main classification decomposition}.

	For \eqref{eq:NLH 6}, if $N=0$, \eqref{eq:NLH global decomposition} implies $\|u(t)\|_{\dot H^1}\to0$, which proves the theorem in this case. If $N\geq1$, Proposition~\ref{lem:sign rigidity} with $\varsigma=-1$ gives $\iota_j=\iota_1(-1)^{j-1}$ for $1\leq j\leq N$. Thus \eqref{eq:refined decomposition}, \eqref{eq:U explicit}, and \eqref{eq:global H1 smallness} prove \eqref{eq:NLH alternating signs}.

	\textbf{Step 2.}
	We now consider $N\geq1$. 
	For \eqref{eq:hmhf 2}, apply the symmetry $\td v=\iota(v-\ell\pi)$. Then $\td v(0)\in\calE_{0,N}$ and all bubble signs are $+1$. Since $N=\iota(m-\ell)$, the initial Bogomol'nyi quantity for $\td v$ satisfies
	\begin{equation*}
		\partial_r\td v(0,r)+\tfrac2r\sin(\td v(0,r)-N\pi)
		=\iota\left(\partial_rv_0(r)+\tfrac2r\sin(v_0(r)-m\pi)\right)
		=\iota\bfD_0(r).
	\end{equation*}
    Therefore, applying Proposition~\ref{prop:lambda1 asym} to $\td v$ for \eqref{eq:hmhf 2} and to $\iota_1u$ for \eqref{eq:NLH 6}, we obtain \eqref{eq:lambda1 initial law}. This completes both theorems when $N=1$.

	\textbf{Step 3.}
	Fix either model with $N\geq2$ and any sufficiently large $T_0\geq t_1^*$. Let $\tau_k$ be given by \eqref{eq:tau k def}. By \eqref{eq:lambda1 zero}, $\tau_k(t)\to\infty$ for $2\leq k\leq N$. Moreover, \eqref{eq:physical scale velocity} implies
	\begin{equation}\label{eq:previous scale clock bound}
		\bigg|\log\frac{\lambda_{k-1}(t)}{\lambda_{k-1}(T_0)}\bigg|
		\leq\|\lambda_{k-1,t}\|_{L^2_t(T_0,\infty)}\tau_k(t)^{1/2}
		=o(\tau_k(t)),\qquad 2\leq k\leq N.
	\end{equation}
	By Proposition~\ref{lem:sign rigidity} and $\varsigma\kappa=-|\kappa|$, \eqref{eq:integrated interface estimate} becomes
	\begin{equation}\label{eq:common integrated interface law}
		\log\mu_k(t)+2\log\lambda_{k-1}(t)-\mathfrak b_k(t)=-|\kappa|\tau_k(t)+o(\tau_k(t)).
	\end{equation}
	Using \eqref{eq:previous scale clock bound}, we obtain $|\log\mu_k(t)| = O(\tau_k(t)+|\mathfrak b_k(t)|)$.
	Since \eqref{eq:frak b esti} implies $\mathfrak b_k=o(|\log\mu_k|)$, we may absorb the last term and conclude that
	\begin{equation*}
		|\log\mu_k(t)|\lesssim\tau_k(t),\qquad \mathfrak b_k(t)=o(\tau_k(t)).
	\end{equation*}
	Substituting these bounds into \eqref{eq:common integrated interface law} and using \eqref{eq:previous scale clock bound}, we obtain
	\begin{equation}\label{eq:recursive inner law}
		\log\lambda_k(t)^{-1}=(|\kappa|+o_t(1))\tau_k(t),\qquad 2\leq k\leq N.
	\end{equation}
	In particular, $\lambda_k(t)\to0$ for $2\leq k\leq N$.

	When $N\geq3$, we compare $\tau_j$ with $\tau_{j-1}$. By \eqref{eq:tau k def}, \eqref{eq:previous scale clock bound}, and \eqref{eq:recursive inner law},
	\begin{equation*}
		\frac{d\tau_j}{d\tau_{j-1}}=\frac{\lambda_{j-2}^2}{\lambda_{j-1}^2}=\exp((2|\kappa|+o_t(1))\tau_{j-1}),\qquad 3\leq j\leq N.
	\end{equation*}
	Since $\tau_{j-1}\to\infty$, comparison with $\exp((2|\kappa|\pm\epsilon)\tau_{j-1})$ for $0<\epsilon<2|\kappa|$ and integration with respect to $\tau_{j-1}$ yield
	\begin{equation*}
		\log\tau_j(t)=(2|\kappa|+o_t(1))\tau_{j-1}(t),\qquad 3\leq j\leq N.
	\end{equation*}
	Taking logarithms successively and using \eqref{eq:recursive inner law}, we obtain
	\begin{equation}\label{eq:inner scale initial clock law}
		\log^{(k-1)}\lambda_k(t)^{-1}=(2|\kappa|+o_t(1))\tau_2(t),\qquad 3\leq k\leq N.
	\end{equation}
    Here, we recall the definition of $\tau_2$ \eqref{eq:tau k def}, $\tau_2(t)=\int_{T_0}^t\frac{ds}{\lambda_{1}(s)^2}$.
	Combining \eqref{eq:recursive inner law} for $k=2$ with \eqref{eq:tau k def} and \eqref{eq:kappa} proves \eqref{eq:second scale initial law} for \eqref{eq:hmhf 2} and \eqref{eq:NLH second scale initial law} for \eqref{eq:NLH 6}. When $N\geq3$, combining \eqref{eq:inner scale initial clock law}, \eqref{eq:tau k def}, and \eqref{eq:kappa} proves \eqref{eq:inner scale initial law} and \eqref{eq:NLH inner scale initial law}, respectively. This completes the proof.
\end{proof}

\section{Construction of \texorpdfstring{\eqref{eq:NLH 6}}{(NLH)} bubble trees}\label{sec:NLH construct}
We prove Theorem~\ref{thm:NLH bubble tree construction} in this section. We adapt the single-bubble energy method of \cite{CollotMerleRaphael2017CMP} to the multi-bubble setting. We first construct a global solution which remains close to an alternating $N$-bubble profile by modulation analysis on finite time intervals and a topological shooting argument. We then apply Theorem~\ref{thm:nlh main} to identify the number and signs of the bubbles and their scale laws.

Throughout this section, we fix
\begin{equation}
	N\geq1,\qquad \iota_j=(-1)^{j-1},\qquad 1\leq j\leq N.
\end{equation}
When $N=1$, every sum indexed by $2\leq j\leq N$ is understood to be zero.

The projection $(\iota_j\calY_{\underline{\lambda_j}},g)$ measures the unstable component associated with the $j$-th bubble. Its evolution contains a leading contribution from the interaction with the larger bubbles indexed by $i<j$. We account for this interaction by introducing the corrected unstable coordinates
\begin{equation}\label{eq:NLH shifted unstable coordinate}
	a_j=a_j(\blambda,g)\coloneqq (\iota_j\calY_{\underline{\lambda_j}},g)+\frac{2(\calY,W)}{e_0}\sum_{i<j}\iota_i\iota_j\bigg(\frac{\lambda_j}{\lambda_i}\bigg)^2,\qquad 1\leq j\leq N.
\end{equation}
Here $-e_0$ is the negative eigenvalue of $H$ for \eqref{eq:NLH 6}. The correction cancels the leading interaction in the unstable-mode equation, while its time derivative is controlled by the dissipative estimates. It vanishes when $j=1$ and has size $O(\mu_j^2)$ for $j\geq2$.
Along a solution, write $a_j(t)=a_j(\blambda(t),g(t))$; for a family indexed by $\mathbf p$, write $a_{\mathbf p,j}(t)=a_j(\blambda_{\mathbf p}(t),g_{\mathbf p}(t))$.

We use a nonnegative non-decreasing function $\delta_{\mathrm c}$, which may be increased in the estimates below, such that
\begin{equation}
	\delta_{\mathrm c}(0)=0,\qquad \delta_{\mathrm p}(\epsilon)\leq\delta_{\mathrm c}(\epsilon),\qquad \delta_{\mathrm c}(\epsilon)\to0\quad\text{as }\epsilon\to0.
\end{equation}
For the scales $\blambda$ and the remainder $g$, write $u=U(\biota,\blambda)+g$ and denote $I=[0,T)$.

\begin{lem}[Energy and modulation estimates]\label{lem:NLH finite tube}
	Let $\blambda\in\calP_N(\alpha_c)$, and let $g\in\dot H^1$ be radial and satisfy
	\begin{equation}\label{eq:NLH finite tube assumptions}
		(Z_{\underline{\lambda_j}},g)=0,\qquad 1\leq j\leq N,\qquad \alpha_c+\|g\|_{\dot H^1}\ll1.
	\end{equation}
	Then, for some sufficiently large $C>0$,
	\begin{equation}\label{eq:NLH alternating energy coercivity}
		|E(U+g)-NE(W)|\lesssim \|g\|_{\dot H^1}^2+\sum_{j=2}^N\mu_j^2
		\lesssim E(U+g)-NE(W)+C\sum_{j=1}^Na_j^2.
	\end{equation}
	Suppose in addition that $u$ is a $\dot H^1$-solution of \eqref{eq:NLH 6} on its maximal lifespan $[0,T_+)$. Let $T\leq T_+$, and assume that $\blambda\in C^1(I;\calP_N(\alpha_c))$ and \eqref{eq:NLH finite tube assumptions} holds on $I$. Then, for $0<t<T$,
	\begin{equation}\label{eq:NLH tube dissipation coercivity}
		\calD(t)+\|g(t)\|_{\dot H^2}^2\lesssim \|u_t(t)\|_{L^2}^2,\qquad
		\sum_{j=1}^N|\lambda_{j,t}(t)|^2\lesssim \|u_t(t)\|_{L^2}^2,
	\end{equation}
	and
	\begin{equation}\label{eq:NLH refined unstable law}
		\bigg|\frac{da_j}{dt}-\frac{e_0}{\lambda_j^2}a_j\bigg|\lesssim \calD+\|g\|_{\dot H^2}^2+\|u_t\|_{L^2}^2,\qquad 1\leq j\leq N.
	\end{equation}
\end{lem}
\begin{proof}
	\textbf{Step 1.}
	In this step, we derive \eqref{eq:NLH alternating energy coercivity}. First, we claim that
	\begin{equation}\label{eq:NLH alternating profile energy}
		E(U)-NE(W)=\frac{\kappa c_*}{2}\sum_{j=2}^N\mu_j^2+O\bigg(\delta_{\mathrm c}(\alpha)\sum_{j=2}^N\mu_j^2\bigg).
	\end{equation}
	For $2\leq m\leq N$, write $U_m=\sum_{j=1}^m\iota_jW_{\lambda_j}$ and use $\calR_j$ in \eqref{eq:f decom}. Here $\nabla E$ denotes the functional derivative with respect to the $L^2$ inner product. Since $\nabla E(U_m)=-\calT(U_m)$,
	\begin{equation*}
		\lambda_m\partial_{\lambda_m}E(U_m)
		=\iota_m\lambda_m^2\sum_{j=2}^m((\Lambda W)_{\underline{\lambda_m}},\calR_j).
	\end{equation*}
	The term $j=m$ is given by \eqref{eq:raw diagonal}. By \eqref{eq:f decom esti} and \eqref{eq:rho overlap},
	\begin{equation*}
		\lambda_m^2\sum_{j=2}^{m-1}|((\Lambda W)_{\underline{\lambda_m}},\calR_j)|
		\lesssim\mathbf1_{\{m\geq3\}}\frac{\lambda_m^2}{\lambda_{m-2}^2}\lesssim\alpha^2\mu_m^2.
	\end{equation*}
	Using $\iota_m\iota_{m-1}=-1$, we obtain
	\begin{equation}\label{eq:NLH profile energy derivative}
		\lambda_m\partial_{\lambda_m}E(U_m)=\kappa c_*\mu_m^2+O(\delta_{\mathrm p}(\alpha)\mu_m^2).
	\end{equation}
	For $0<s\leq\lambda_m$, we have $\alpha(\lambda_1,\ldots,\lambda_{m-1},s)\leq\alpha(\blambda)<\alpha_c$. Applying \eqref{eq:NLH profile energy derivative} with $\lambda_m=s$ and using the monotonicity of $\delta_{\mathrm p}$, we obtain
	\begin{equation*}
		\frac{d}{ds}E(U_{m-1}+\iota_mW_s)=(\kappa c_*+O(\delta_{\mathrm p}(\alpha)))\frac{s}{\lambda_{m-1}^2},\qquad 0<s\leq\lambda_m.
	\end{equation*}
	Since $E(U_{m-1}+\iota_mW_s) \to E(U_{m-1})+E(W)$ as $s\to0$, we have
	\begin{align*}
		E(U_m)-E(U_{m-1})-E(W)
		&=\int_0^{\lambda_m}(\kappa c_*+O(\delta_{\mathrm p}(\alpha)))
		\bigg(\frac{s}{\lambda_{m-1}}\bigg)^2\frac{ds}{s}\\
		&=\frac{\kappa c_*}{2}\mu_m^2+O(\delta_{\mathrm p}(\alpha)\mu_m^2).
	\end{align*}
	Summing over $2\leq m\leq N$, we obtain \eqref{eq:NLH alternating profile energy}.
	
	We next claim
	\begin{equation}\label{eq:NLH profile energy derivative small}
		\|\calT(U)\|_{\dot H^{-1}}^2\leq\delta_{\mathrm c}(\alpha)\sum_{j=2}^N\mu_j^2.
	\end{equation}
	For $0<a\leq b$, H\"older's inequality and scaling give
	\begin{equation*}
		\|\rho_a\rho_b\|_{L^{3/2}}\leq\|\rho_a\|_{L^{12/7}}\|\rho_b\|_{L^{12}}
		\lesssim(a/b)^{3/2}.
	\end{equation*}
	Since $\nabla E(U)=-\calT(U)$, the dual of \eqref{eq:sobolev} and \eqref{eq:f decom esti} yield
	\begin{equation*}
		\|\calT(U)\|_{\dot H^{-1}}\lesssim\sum_{i<j}(\lambda_j/\lambda_i)^{3/2}
		\lesssim\sum_{j=2}^N\mu_j^{3/2}
		\lesssim\alpha^{1/2}\bigg(\sum_{j=2}^N\mu_j^2\bigg)^{1/2}.
	\end{equation*}
	This proves \eqref{eq:NLH profile energy derivative small}.
	Thus, by \eqref{eq:sobolev} and \eqref{eq:NLH finite tube assumptions}, we obtain
	\begin{equation*}
		|E(U+g)-NE(W)|\lesssim \sum_{j=2}^N\mu_j^2+\|\calT(U)\|_{\dot H^{-1}}\|g\|_{\dot H^1}
		+\|g\|_{\dot H^1}^2
		\lesssim \sum_{j=2}^N\mu_j^2+\|g\|_{\dot H^1}^2,
	\end{equation*}
	which proves the first estimate of \eqref{eq:NLH alternating energy coercivity}.
	Moreover, we have
	\begin{equation*}
		E(U+g)-E(U)=-(g,\calT(U))+\tfrac12(g,H_Ug)+O(\|g\|_{\dot H^1}^3).
	\end{equation*}
	Using Lemma~\ref{lem:NLH shooting coercivity} and \eqref{eq:sobolev}, we obtain
	\begin{equation*}
		\|g\|_{\dot H^1}^2\lesssim E(U+g)-E(U)+C\bigg(\sum_{j=1}^N(\calY_{\underline{\lambda_j}},g)^2+\|\calT(U)\|_{\dot H^{-1}}^2\bigg).
	\end{equation*}
	Combining this estimate with \eqref{eq:NLH alternating profile energy} and \eqref{eq:NLH profile energy derivative small}, and using \eqref{eq:NLH finite tube assumptions}, we deduce
	\begin{equation*}
		\|g\|_{\dot H^1}^2+\sum_{j=2}^N\mu_j^2\lesssim E(U+g)-NE(W)+C\sum_{j=1}^N(\calY_{\underline{\lambda_j}},g)^2.
	\end{equation*}
	For the last term, thanks to \eqref{eq:NLH shifted unstable coordinate}, we get
	\begin{equation*}
		|a_j-(\iota_j\calY_{\underline{\lambda_j}},g)|\lesssim\mu_j^2,\qquad
		\sum_{j=1}^N|a_j-(\iota_j\calY_{\underline{\lambda_j}},g)|^2\lesssim\alpha_c^2\sum_{j=2}^N\mu_j^2.
	\end{equation*}
	Therefore, for sufficiently small $\alpha_c$, we prove the second estimate of \eqref{eq:NLH alternating energy coercivity}.
	
	\textbf{Step 2.}
	We now prove \eqref{eq:NLH tube dissipation coercivity} and \eqref{eq:NLH refined unstable law}.
	By \eqref{eq:sobolev} and \eqref{eq:NL esti},
	\begin{equation*}
		\|\NL_U(g)\|_{L^2}\lesssim\|g\|_{\dot H^1}\|g\|_{\dot H^2}.
	\end{equation*}
	Applying Lemma~\ref{lem:dissipation coercivity} with $h=g$ and using $u_t=\calT(U)-H_Ug+\NL_U(g)$, we obtain the first estimate in \eqref{eq:NLH tube dissipation coercivity} by using $\|g\|_{\dot H^1}\ll 1$. Differentiating the orthogonality conditions in \eqref{eq:NLH finite tube assumptions} and applying \eqref{eq:weighted linear} with $\omega_j=\lambda_j$ and $h_k=-(\iota_kZ_{\underline{\lambda_k}},u_t)$ proves the second estimate of \eqref{eq:NLH tube dissipation coercivity}.
	
	Next, we show \eqref{eq:NLH refined unstable law}. Recall that
	\begin{equation*}
		g_t=-H_Ug+\NL_U(g)+\calT(U)+\sum_{k=1}^N\frac{\lambda_{k,t}}{\lambda_k}\iota_k(\Lambda W)_{\lambda_k}.
	\end{equation*}
	Differentiating $(\iota_j\calY_{\underline{\lambda_j}},g)$, we obtain
	\begin{align*}
		\frac{d}{dt}(\iota_j\calY_{\underline{\lambda_j}},g)-\frac{e_0}{\lambda_j^2}(\iota_j\calY_{\underline{\lambda_j}},g)
		=\,&(\iota_j\calY_{\underline{\lambda_j}},r^{-2}[f'(P)-f'(\iota_jQ_{[\lambda_j]})]g)
		\\
		&+(\iota_j\calY_{\underline{\lambda_j}},\NL_U(g))
		+(\iota_j\calY_{\underline{\lambda_j}},\calT(U))
		\\
		&+\sum_{k\neq j}\frac{\lambda_{k,t}}{\lambda_k}(\iota_j\calY_{\underline{\lambda_j}},\iota_k(\Lambda W)_{\lambda_k})
		-\frac{\lambda_{j,t}}{\lambda_j}(\iota_j(\Lambda_{-1}\calY)_{\underline{\lambda_j}},g).
	\end{align*}
	Here the diagonal term in the sum vanishes because $(\calY,\Lambda W)=0$. For the first term, using the explicit formula for $W$ and the exponential decay of $\calY$, direct integration with $|f'(P)-f'(\iota_jQ_{[\lambda_j]})|\lesssim r^2\sum_{i\neq j}W_{\lambda_i}$ yields
	\begin{equation}\label{eq:NLH pure unstable 1}
		|(\iota_j\calY_{\underline{\lambda_j}},r^{-2}[f'(P)-f'(\iota_jQ_{[\lambda_j]})]g)|\lesssim
		\sum_{i\neq j}\|r^2W_{\lambda_i}\calY_{\underline{\lambda_j}}\|_{L^2}\|g\|_{\dot H^2}
		\lesssim \calD+\|g\|_{\dot H^2}^2.
	\end{equation}
	The nonlinear term is estimated by
	\begin{equation}\label{eq:NLH pure unstable 2}
		|(\iota_j\calY_{\underline{\lambda_j}},\NL_U(g))|\lesssim {\int_0^\infty}|g|^2r\,dr\lesssim\|g\|_{\dot H^2}^2.
	\end{equation}
	Next, we estimate the last two terms. For $j\neq k$,
	\begin{equation*}
		{\lambda_k^{-2}}|(\calY_{\underline{\lambda_j}},(\Lambda W)_{\lambda_k})|^2\lesssim d_{\min\{k,j\}+1}^2.
	\end{equation*}
	Thus, we arrive at
	\begin{equation}\label{eq:NLH pure unstable 3}
		\bigg|\sum_{k\neq j}\frac{\lambda_{k,t}}{\lambda_k}(\iota_j\calY_{\underline{\lambda_j}},\iota_k(\Lambda W)_{\lambda_k})\bigg|
		\lesssim \calD+\sum_k|\lambda_{k,t}|^2 \lesssim \calD+\|u_t\|_{L^2}^2.
	\end{equation}
	By a direct computation with \eqref{eq:NLH tube dissipation coercivity}, the last term satisfies
	\begin{equation}\label{eq:NLH pure unstable 4}
		\bigg|\frac{\lambda_{j,t}}{\lambda_j}(\iota_j(\Lambda_{-1}\calY)_{\underline{\lambda_j}},g)\bigg|\lesssim|\lambda_{j,t}|\|g\|_{\dot H^2}\lesssim \|u_t\|_{L^2}^2+\|g\|_{\dot H^2}^2.
	\end{equation}
	
	We now estimate $(\iota_j\calY_{\underline{\lambda_j}},\calT(U))$. To get \eqref{eq:NLH refined unstable law}, by \eqref{eq:NLH shifted unstable coordinate}, we need to show
	\begin{equation}\label{eq:NLH pure unstable 5}
		\bigg|(\iota_j\calY_{\underline{\lambda_j}},\calT(U))-2(\calY,W)\sum_{i<j}\frac{\iota_i\iota_j}{\lambda_i^2}\bigg|\lesssim\calD,\qquad 1\leq j\leq N.
	\end{equation}
	By \eqref{eq:f decom} and the homogeneity of $f(z)=|z|z$,
	\begin{equation*}
		\calR_\ell=f(\iota_\ell W_{\lambda_\ell}+U_{\ell-1})-f(\iota_\ell W_{\lambda_\ell})-f(U_{\ell-1}).
	\end{equation*}
	For $2\leq\ell\leq N$, set
	\begin{equation}\label{eq:NLH unstable residual form}
		\calR_\ell^{\mathrm e}\coloneqq\calR_\ell-2W_{\lambda_\ell}\sum_{i<\ell}\frac{\iota_i}{\lambda_i^2}.
	\end{equation}
	By \eqref{eq:Q f} with $W(0)=1$, we obtain
	\begin{equation*}
		|U_{\ell-1}|\lesssim\lambda_{\ell-1}^{-2},\qquad \bigg|U_{\ell-1}-\sum_{i<\ell}\frac{\iota_i}{\lambda_i^2}\bigg|\lesssim r^2\lambda_{\ell-1}^{-4},\qquad W_{\lambda_\ell}\lesssim\rho_{\lambda_\ell},\qquad r^2\rho_{\lambda_\ell}\lesssim1.
	\end{equation*}
	Applying \eqref{eq:f Taylor} and \eqref{eq:f interaction}, we obtain, respectively,
	\begin{align*}
		|\calR_\ell^{\mathrm e}|&\lesssim W_{\lambda_\ell}\bigg|U_{\ell-1}-\sum_{i<\ell}\frac{\iota_i}{\lambda_i^2}\bigg|+|U_{\ell-1}|^2\lesssim\lambda_{\ell-1}^{-4},\\
		|\calR_\ell^{\mathrm e}|&\lesssim W_{\lambda_\ell}|U_{\ell-1}|+W_{\lambda_\ell}\sum_{i<\ell}\lambda_i^{-2}\lesssim\lambda_{\ell-1}^{-2}\rho_{\lambda_\ell}.
	\end{align*}
	Note that $f'(\iota_\ell W_{\lambda_\ell})=2W_{\lambda_\ell}$ and $|f(U_{\ell-1})|=|U_{\ell-1}|^2$.
	Hence
	\begin{equation*}
		|\calR_\ell^{\mathrm e}|\lesssim\min\{\lambda_{\ell-1}^{-4},\lambda_{\ell-1}^{-2}\rho_{\lambda_\ell}\}.
	\end{equation*}
	Thus, we have
	\begin{equation}\label{eq:NLH unstable residual}
		|(\calY_{\underline{\lambda_j}},\calR_\ell^{\mathrm e})|\lesssim d_\ell^2.
	\end{equation}
	For $\ell\neq j$, using $(\calY_{\underline a},W_b)\lesssim\min\{(a/b)^2,(b/a)^2\}$, we derive
	\begin{equation}\label{eq:NLH unstable off diagonal interaction}
		\bigg|\bigg(\calY_{\underline{\lambda_j}},2W_{\lambda_\ell}\sum_{i<\ell}\frac{\iota_i}{\lambda_i^2}\bigg)\bigg|\lesssim
		\begin{cases}
			\lambda_{\ell-1}^{-2}(\lambda_j/\lambda_\ell)^2\leq d_{\ell+1}^2,&\ell<j,\\
			\lambda_{\ell-1}^{-2}(\lambda_\ell/\lambda_j)^2\leq d_\ell^2,&\ell>j.
		\end{cases}
	\end{equation}
	When $j\geq2$, for the remaining case $\ell=j$, we have $(\calY_{\underline{\lambda_j}},W_{\lambda_j})=(\calY,W)$. Summing \eqref{eq:NLH unstable residual} and \eqref{eq:NLH unstable off diagonal interaction} with \eqref{eq:f decom} and \eqref{eq:NLH unstable residual form}, we obtain \eqref{eq:NLH pure unstable 5}.
	
	Collecting \eqref{eq:NLH pure unstable 1}--\eqref{eq:NLH pure unstable 5}, we obtain
	\begin{equation}\label{eq:NLH pure raw unstable law}
		\bigg|\frac{d}{dt}(\iota_j\calY_{\underline{\lambda_j}},g)-\frac{e_0}{\lambda_j^2}(\iota_j\calY_{\underline{\lambda_j}},g)-2(\calY,W)\sum_{i<j}\frac{\iota_i\iota_j}{\lambda_i^2}\bigg|\lesssim\calD+\|g\|_{\dot H^2}^2+\|u_t\|_{L^2}^2.
	\end{equation}
	Finally, to conclude \eqref{eq:NLH refined unstable law}, we estimate the time derivative of the final term in \eqref{eq:NLH shifted unstable coordinate}.
	By \eqref{eq:NLH tube dissipation coercivity}, we compute, for $i<j$, 
	\begin{equation*}
		\bigg|\frac{d}{dt}\bigg(\frac{\lambda_j}{\lambda_i}\bigg)^2\bigg|\leq2\frac{\lambda_j}{\lambda_i^2}|\lambda_{j,t}|+2\frac{\lambda_j^2}{\lambda_i^3}|\lambda_{i,t}|\lesssim d_j(|\lambda_{j,t}|+|\lambda_{i,t}|) \lesssim \calD+\|u_t\|_{L^2}^2.
	\end{equation*}
	Therefore, substituting \eqref{eq:NLH shifted unstable coordinate} into \eqref{eq:NLH pure raw unstable law} proves \eqref{eq:NLH refined unstable law}.
\end{proof}

Finally, we finish the proof of Theorem~\ref{thm:NLH bubble tree construction}.

\begin{proof}[Proof of Theorem~\ref{thm:NLH bubble tree construction}]
    Fix $N$ and $q$ as in Theorem~\ref{thm:NLH bubble tree construction}.
    
	\textbf{Step 1.}
	We construct an initial data family $u_{\mathbf p}(0)$ indexed by $\mathbf p\in\calC_\beta=[-\beta,\beta]^N$. Let $\blambda_{\mathbf p}(0)$ denote its initial scales, with $j$-th component $\lambda_{\mathbf p,j}(0)$, and let $g_{\mathbf p}(0)\coloneqq u_{\mathbf p}(0)-U(\biota,\blambda_{\mathbf p}(0))$ be the remainder. We require
	\begin{equation}\label{eq:NLH construct init}
		\begin{gathered}
			a_{\mathbf p,j}(0)=p_j,\qquad 1\leq j\leq N,\\
			\alpha(\blambda_{\mathbf p}(0))<\tfrac{1}{2}\alpha_c,\qquad \|g_{\mathbf p}(0)\|_{\dot H^1}<\tfrac{1}{2}\alpha_c,\qquad 
			|E(u_{\mathbf p}(0))-NE(W)|\lesssim \beta^2.
		\end{gathered}
	\end{equation}
	We choose the parameters in the order $\alpha_c,\beta,\gamma$, sufficiently small for Lemmas~\ref{lem:NLH shooting coercivity} and~\ref{lem:NLH finite tube} and the estimates below, so that
	\begin{equation}\label{eq:NLH shooting hierarchy}
		0<\gamma\ll\beta\ll\alpha_c\ll1.
	\end{equation}
	We also require
	\begin{equation}\label{eq:NLH construc alpha delta condition}
		\alpha_c<\alpha_{\mathrm d},\qquad \alpha_c\ll\delta_{\mathrm d}.
	\end{equation}
	Define $\boldsymbol\nu=(\nu_1,\ldots, \nu_N)$ by
	\begin{equation}\label{eq:NLH prepared scales}
		\nu_1=1,\qquad \nu_j=\gamma\nu_{j-1},\qquad 2\leq j\leq N.
	\end{equation}
	Define $q_R\coloneqq(1-\chi_R)q$. Fix $R>R_0$ so large that
	\begin{equation*}
		\|q_R\|_{\dot H^1}\ll\beta^2.
	\end{equation*}
	Choose a radial $\varphi\in C_c^\infty(\bbR^6)$ such that $\operatorname{supp}\varphi\subset\{2R_0<r<3R_0\}$ and $(\calY,\varphi)=1$. For $\mathbf b=(b_1,\ldots,b_N)\in\bbR^N$, define
	\begin{equation}\label{eq:NLH shooting initial family}
		u_{\mathbf b}(0)=U(\biota,\boldsymbol\nu)+q_R+\sum_{j=1}^Nb_j\iota_j\varphi_{\nu_j}.
	\end{equation}
	For sufficiently small $\gamma$, the supports of $Z_{\underline{\nu_k}}$ and $\varphi_{\nu_j}$ are disjoint for every $j,k$. Since $R>R_0$ and $\nu_1=1$, we also have $(Z_{\underline{\nu_k}},q_R)=0$. Thus the initial decomposition is
	\begin{equation*}
		\blambda_{\mathbf b}(0)=\boldsymbol\nu,\qquad
		g_{\mathbf b}(0)=u_{\mathbf b}(0)-U(\biota,\boldsymbol\nu),\qquad
		(Z_{\underline{\nu_k}},g_{\mathbf b}(0))=0.
	\end{equation*}
	The exponential decay of $\calY$ and \eqref{eq:NLH prepared scales} give
	\begin{equation*}
		(\iota_j\calY_{\underline{\nu_j}},\iota_k\varphi_{\nu_k})=\delta_{jk} +O(\gamma^2).
	\end{equation*}
    For each $\mathbf p\in\calC_\beta$, we choose $\mathbf b$ so that $a_j(\boldsymbol\nu,g_{\mathbf b}(0))=p_j$. By \eqref{eq:NLH shifted unstable coordinate}, these conditions are equivalent to
    \begin{equation*}
    	\sum_{k=1}^N(\iota_j\calY_{\underline{\nu_j}},\iota_k\varphi_{\nu_k})b_k=p_j-a_j(\boldsymbol\nu,q_R),\qquad 1\leq j\leq N.
    \end{equation*}
    The coefficient matrix is $I+O(\gamma^2)$, so there is a unique solution $\mathbf b=\mathbf b(\mathbf p)$ satisfying
    \begin{equation*}
    	|\mathbf b(\mathbf p)|\lesssim|\mathbf p|+\|q_R\|_{\dot H^1}+\gamma^2\lesssim\beta.
    \end{equation*}
	Set $u_{\mathbf p}(0)=u_{\mathbf b(\mathbf p)}(0)$, $\blambda_{\mathbf p}(0)=\boldsymbol\nu$, and $g_{\mathbf p}(0)=g_{\mathbf b(\mathbf p)}(0)$. The difference $u_{\mathbf p}(0)-u_{\mathbf0}(0)$ depends linearly on $\mathbf p$, and the family satisfies
	\begin{equation}\label{eq:NLH shooting initial scale displacement}
		\|g_{\mathbf p}(0)\|_{\dot H^1}\lesssim\beta.
	\end{equation}
	In particular, $\alpha(\blambda_{\mathbf p}(0))\leq\gamma<\alpha_c/2$ and $\|g_{\mathbf p}(0)\|_{\dot H^1}\lesssim\beta<\alpha_c/2$ uniformly on $\calC_\beta$.

	By Lemma~\ref{lem:NLH finite tube}, \eqref{eq:NLH shooting initial scale displacement}, \eqref{eq:NLH shooting hierarchy}, and \eqref{eq:NLH prepared scales},
	\begin{equation*}
		|E(u_{\mathbf p}(0))-NE(W)|\lesssim \|g_{\mathbf p}(0)\|_{\dot H^1}^2+\sum_{j=2}^N\bigg(\frac{\lambda_{\mathbf p,j}(0)}{\lambda_{\mathbf p,j-1}(0)}\bigg)^2\lesssim \beta^2.
	\end{equation*}
	Thus the initial data family satisfies \eqref{eq:NLH construct init} for every $\mathbf p\in\calC_\beta=[-\beta,\beta]^N$.

	\textbf{Step 2.}
	For each $\mathbf p\in\calC_\beta$, let $u_{\mathbf p}$ be the solution with the datum constructed in Step~1, defined on its maximal interval $[0,T_+(\mathbf p))$. Since $u_{\mathbf p}(t)\to u_{\mathbf p}(0)$ in $\dot H^1$, Lemma~\ref{lem:decomposition near multi bubble}, applied in the fixed neighborhood of $U(\biota,\blambda_{\mathbf p}(0))$, extends the initial decomposition to a short time interval:
	\begin{equation}\label{eq:trapping decom}
		u_{\mathbf p}(t)=U(\biota,\blambda_{\mathbf p}(t))+g_{\mathbf p}(t),\qquad (\iota_jZ_{\underline{\lambda_{\mathbf p,j}(t)}},g_{\mathbf p}(t))=0,\qquad 1\leq j\leq N.
	\end{equation}
	The scales are $C^1$, and the local decompositions agree by uniqueness. For a fixed $\mathbf p$, write $u=u_{\mathbf p}$, $\blambda=\blambda_{\mathbf p}$, and $g=g_{\mathbf p}$, with all signed profiles evaluated at $\blambda(t)$.

	For $\mathbf p\in(-\beta,\beta)^N$, let $0\leq T(\mathbf p)\leq T_+(\mathbf p)$ be the maximal time such that \eqref{eq:trapping decom} holds on $[0,T(\mathbf p))$ with
	\begin{equation}\label{eq:NLH shooting tube}
		\alpha(t)<\alpha_c,\quad \|g(t)\|_{\dot H^1}<\alpha_c,\quad
		\bigg|p_j+\int_0^t\frac{e_0}{\lambda_j(s)^2}a_j(s)ds\bigg|<\beta,\quad 1\leq j\leq N.
	\end{equation}
	By the initial bounds \eqref{eq:NLH construct init}, $T(\mathbf p)>0$. For $\mathbf p\in\partial\calC_\beta$, set $T(\mathbf p)=0$.

	Now fix $\mathbf p\in(-\beta,\beta)^N$ and $0\leq t<T(\mathbf p)$. Define
	\begin{equation*}
		X(t)\coloneqq\sup_{0\leq s\leq t}\bigg(\sum_{j=2}^N\mu_j(s)^2+\|g(s)\|_{\dot H^1}^2\bigg)
		+\int_0^t\big(\calD+\|g\|_{\dot H^2}^2+\|u_t\|_{L^2}^2\big)ds.
	\end{equation*}
	Integrating \eqref{eq:NLH refined unstable law} and using $a_j(0)=p_j$, we obtain
	\begin{equation}\label{eq:NLH corrected unstable coordinate}
		\bigg|a_j-p_j-\int_0^t\frac{e_0}{\lambda_j^2}a_jds\bigg|\lesssim X(t).
	\end{equation} 
	By \eqref{eq:NLH shooting tube} and \eqref{eq:NLH corrected unstable coordinate}, we get
	\begin{equation}\label{eq:unstable mode inner product}
		\sup_{0\leq s\leq t}|a_j(s)|\lesssim\beta+X(t).
	\end{equation}
	From \eqref{eq:NLH energy identity}, \eqref{eq:NLH construct init}, and \eqref{eq:NLH alternating energy coercivity}, we obtain
	\begin{equation*}
		\|g\|_{\dot H^1}^2 +\sum_{j=2}^N\mu_j^2 +\int_0^t\|u_t\|_{L^2}^2ds\lesssim\beta^2+\sum_{j=1}^Na_j^2.
	\end{equation*}
	Together with \eqref{eq:NLH tube dissipation coercivity} and \eqref{eq:unstable mode inner product}, this yields
	\begin{equation*}
		X(t)\lesssim\beta^2+\beta X(t)+X(t)^2.
	\end{equation*}
	Since $X(0)\lesssim\beta^2$, a continuity argument gives
	\begin{equation}\label{eq:NLH shooting energy control}
		\sup_{0\leq t<T(\mathbf p)}X(t)\lesssim\beta^2.
	\end{equation}
	In particular,
	\begin{equation}\label{eq:NLH construct global alpha g}
		\alpha(t)+\|g(t)\|_{\dot H^1}\lesssim\beta,\qquad \alpha(t)<\tfrac{1}{2}\alpha_c,\qquad \|g(t)\|_{\dot H^1}<\tfrac12 \alpha_c.
	\end{equation}

	By \eqref{eq:NLH shooting tube}, \eqref{eq:NLH corrected unstable coordinate}, and \eqref{eq:NLH shooting energy control}, we have
	\begin{equation*}
		\frac{d}{dt}\bigg|p_j+\int_0^t\frac{e_0}{\lambda_j^2}a_jds\bigg|^2
		=\frac{2e_0}{\lambda_j^2}\left[\bigg|p_j+\int_0^t\frac{e_0}{\lambda_j^2}a_jds\bigg|^2+O(\beta^3)\right].
	\end{equation*}
	If we have
	\begin{equation*}
		\frac\beta2\leq\bigg|p_j+\int_0^t\frac{e_0}{\lambda_j^2}a_jds\bigg|\leq\beta,
	\end{equation*}
	then, taking $\beta$ sufficiently small, we arrive at
	\begin{equation}\label{eq:NLH unstable outward}
		\frac{d}{dt}\bigg|p_j+\int_0^t\frac{e_0}{\lambda_j^2}a_jds\bigg|^2\geq\frac{e_0}{4\lambda_j^2}\beta^2>0.
	\end{equation}

	\textbf{Step 3.}
	We first show that every finite $T(\mathbf p)$ for $\mathbf p\in(-\beta,\beta)^N$ is an exit time through the last condition in \eqref{eq:NLH shooting tube}.
	If $T(\mathbf p)=T_+(\mathbf p)<\infty$, \eqref{eq:trapping decom} and \eqref{eq:NLH shooting energy control} give \eqref{eq:NLH bounded}, contradicting Proposition~\ref{prop:globality}. Thus $T(\mathbf p)<T_+(\mathbf p)$, and the scales $\lambda_j$ cannot vanish at $T(\mathbf p)$ by standard Cauchy theory. By \eqref{eq:NLH construct global alpha g}, \eqref{eq:NLH construc alpha delta condition}, and Lemma~\ref{lem:decomposition near multi bubble}, the decomposition \eqref{eq:trapping decom} cannot break down at $T(\mathbf p)$. Therefore, we arrive at
	\begin{equation}\label{eq:NLH shooting exit condition}
		\max_{1\leq j\leq N}\bigg|p_j+\int_0^{T(\mathbf p)}\frac{e_0}{\lambda_j^2}a_jds\bigg|=\beta.
	\end{equation}
	Suppose now that $T(\mathbf p)<\infty$ for every $\mathbf p\in\calC_\beta$. Define a map $\Phi:\calC_\beta \to\partial\calC_\beta$ by
	\begin{equation*}
		\Phi(\mathbf p)\coloneqq\mathbf p+
		\bigg(\int_0^{T(\mathbf p)}\frac{e_0}{\lambda_{\mathbf p,j}^2}a_{\mathbf p,j}ds\bigg)_{1\leq j\leq N}.
	\end{equation*}
	By \eqref{eq:NLH shooting exit condition}, this map is well-defined, and $\Phi(\mathbf p)=\mathbf p$ on $\partial\calC_\beta$.

	We next prove that $\Phi$ is continuous. For $\mathbf p_0\in(-\beta,\beta)^N$, continuous dependence and Lemma~\ref{lem:decomposition near multi bubble} extend the nearby solutions and decompositions to a common interval beyond $T(\mathbf p_0)$, with the first two conditions in \eqref{eq:NLH shooting tube}. By \eqref{eq:NLH unstable outward}, $T(\mathbf p)\to T(\mathbf p_0)$, and hence $\Phi(\mathbf p)\to\Phi(\mathbf p_0)$. If $\mathbf p_0\in\partial\calC_\beta$, then $a_{\mathbf p_0,j}(0)=(\mathbf p_0)_j$ by \eqref{eq:NLH construct init}, so \eqref{eq:NLH unstable outward} holds at $t=0$ for any $j$ with $|(\mathbf p_0)_j|=\beta$. The same argument yields $T(\mathbf p)\to0$ and $\Phi(\mathbf p)\to\mathbf p_0$.
	Hence $\Phi$ is a continuous retraction of $\calC_\beta$ onto its boundary, contradicting the Brouwer fixed-point theorem. Thus, there exists $\mathbf p_*\in(-\beta,\beta)^N$ such that $T(\mathbf p_*)=\infty$.

	\textbf{Step 4.}
	Fix $\mathbf p_*$ such that $T(\mathbf p_*)=\infty$, and write $u=u_{\mathbf p_*}$ and $u_0=u_{\mathbf p_*}(0)$. Then $T_+(\mathbf p_*)=\infty$. By \eqref{eq:trapping decom} and \eqref{eq:NLH construct global alpha g}, we obtain \eqref{eq:NLH bounded}. Moreover, \eqref{eq:NLH energy identity}, \eqref{eq:NLH construct init}, and \eqref{eq:NLH shooting energy control} give
	\begin{equation*}
		E(u(t))=NE(W)+O(\beta^2),\qquad t\geq0.
	\end{equation*}
	Let $M\geq0$ be the number of bubbles supplied by Theorem~\ref{thm:nlh main}. By \eqref{eq:NLH global decomposition}, we have $E(u(t))\to ME(W)$. Thus $M=N$ for sufficiently small $\beta$.
	Let $\sigma_1,\ldots,\sigma_N$ be the signs supplied by Theorem~\ref{thm:nlh main}. By \eqref{eq:NLH alternating signs}, $\sigma_j=\sigma_1(-1)^{j-1}$. Comparing \eqref{eq:trapping decom} and \eqref{eq:NLH global decomposition}, and using \eqref{eq:NLH construct global alpha g}, we get $\sigma_1=1$ for sufficiently small $\alpha_c$ and $\beta$. Hence $\sigma_j=(-1)^{j-1}$ for $1\leq j\leq N$.

	For the rest of the proof, let $\blambda(t)$ denote the $C^1$ scales given by Proposition~\ref{prop:decomposition solution}. Then \eqref{eq:NLH constructed bubble tree} holds.
	By \eqref{eq:NLH shooting initial family} and \eqref{eq:U explicit}, we have $u_0-q\in H^1$. Hence, the Cauchy--Schwarz inequality implies $\int_1^\infty|u_0(r)-q(r)|rdr<\infty$. By \eqref{eq:intro NLH initial tail} and \eqref{eq:intro NLH prescribed tail},
	\begin{equation}\label{eq:NLH constructed initial scale comparison}
		I_0(t)=(K+o(1))I_q(t),
	\end{equation}
	where $K=\exp\{\frac5{16}\int_1^\infty(u_0(r)-q(r))r\,dr\}>0$.
	Applying \eqref{eq:NLH lambda1 initial law} with \eqref{eq:NLH constructed initial scale comparison}, we obtain a constant $L>0$ such that
	\begin{equation*}
		\lambda_1(t)=(L+o(1))I_0(t)=(LK+o(1))I_q(t).
	\end{equation*}
    In addition, \eqref{eq:NLH second scale initial law} and \eqref{eq:NLH inner scale initial law} imply \eqref{eq:NLH constructed second scale} and \eqref{eq:NLH constructed inner scales} respectively.
	Renaming $LK$ by $L$, we conclude the proof.
\end{proof}


\appendix
\section{Proofs of the coercivity estimates}\label{app:coercivity}

\begin{proof}[Proof of Lemma~\ref{lem:coercivity N}]
	We sketch the argument of \cite[Lemma~2.7]{KimMerle2025CPAM}. The upper bound follows from \eqref{eq:sobolev} and the boundedness of $f'(Q)$. By \eqref{eq:Q f}, $f'$ is even and, in both models,
	\begin{equation}\label{eq:coercivity potential decay}
		|f'(Q_{[\nu]})(r)|\lesssim\min\{\nu^{-2}r^2,\nu^2r^{-2}\},\qquad \nu,r>0.
	\end{equation}

	We prove the estimate by contradiction. Otherwise there exist scale vectors $\boldsymbol\lambda^{(n)}$ with $\alpha^{(n)}\to 0$ and radial functions $h_n\in\dot H^2$ such that
	\begin{equation}\label{eq:coercivity seq}
		\|h_n\|_{\dot H^2}=1,\quad (Z_{{\lambda_j^{(n)}}},h_n)=0, \quad \|H_{\boldsymbol\lambda^{(n)}}h_n\|_{L^2}\to0, \quad 1\leq j\leq N.
	\end{equation}
    Outside $[A^{-1}\lambda_j^{(n)},A\lambda_j^{(n)}]$, \eqref{eq:coercivity potential decay} implies $|f'(Q_{[\lambda_j^{(n)}]})|\lesssim A^{-2}$. Thus, combining \eqref{eq:decoupled operator}, \eqref{eq:coercivity seq}, and \eqref{eq:sobolev}, we have
    \begin{equation*}
		1-o_n(1)=\|h_n\|_{\dot H^2}-\|H_{\boldsymbol\lambda^{(n)}}h_n\|_{L^2}\lesssim 
        A^{-2}+\max_{1\leq j\leq N}\|\mathbf1_{[A^{-1}\lambda_j^{(n)},A\lambda_j^{(n)}]}r^{-2}h_n\|_{L^2}.
	\end{equation*}
    Choosing $A>1$ sufficiently large, we obtain, for all sufficiently large $n$,
	\begin{equation}\label{eq:coercivity concentration}
		1\gtrsim \max_{1\leq j\leq N}\|\mathbf1_{[A^{-1}\lambda_j^{(n)},A\lambda_j^{(n)}]}r^{-2}h_n\|_{L^2}\gtrsim1.
	\end{equation}
	Thus, there exists a sequence $k_n\in \{1,2,\cdots,N\}$ such that $\|\mathbf1_{r\sim \lambda_{k_n}^{(n)}} r^{-2}h_n \|_{L^2}\sim 1$. 
    Define $\td h_n(y)\coloneqq \lambda_{k_n}^{(n)}h_n(\lambda_{k_n}^{(n)}y)$ and ${\boldsymbol\nu^{(n)}}\coloneqq \blambda^{(n)}/\lambda_{k_n}^{(n)}$. Then, we have
	\begin{equation}\label{eq:coercivity concentration seq}
		\|\td h_n\|_{\dot H^2}=1,\quad (Z,\td h_n)=0, \quad \|H_{{\boldsymbol\nu^{(n)}}}\td h_n\|_{L^2}\to0, \quad \|\mathbf1_{r\sim 1} \td h_n \|_{L^2}\sim 1.
	\end{equation}
	Passing to a subsequence and using the Rellich--Kondrachov theorem, we obtain
	\begin{align*}
		\td h_n\rightharpoonup\td h \quad \text{in}\quad \dot H^2, \qquad \td h_n\to\td h\quad \text{in}\quad L^2_{\mathrm{loc}}(\bbR^6\setminus\{0\}).
	\end{align*}
	In particular, we obtain $\td h\neq 0$ and $(Z,\td h)=0$ by \eqref{eq:coercivity concentration seq}.

	Now, we claim $H\td h=0$. To show this, we compute
	\begin{align*}
		(\varphi,H\td h)=\lim_{n\to \infty}(\varphi,H\td h_n) =&\lim_{n\to \infty}[ (\varphi,H_{\boldsymbol\nu^{(n)}}\td h_n)+{\textstyle \sum_{j\neq k_n}}(\varphi,r^{-2}f'(Q_{[\nu_j^{(n)}]})\td h_n) ] \\
		=&\lim_{n\to \infty}{\textstyle \sum_{j\neq k_n}}(\varphi,r^{-2}f'(Q_{[\nu_j^{(n)}]})\td h_n).
	\end{align*}
    By \eqref{eq:coercivity potential decay} and the scale separation $\alpha(\blambda^{(n)})\to 0$, the last term vanishes for every $\varphi\in C_c^\infty$. Thus, $(\varphi,H\td h)=0$ for all $\varphi \in C_c^{\infty}(0,\infty)$, and we conclude $H\td h=0$. This contradicts the single-bubble coercivity \eqref{eq:coercivity 1}. Therefore, the lower bound in \eqref{eq:coercivity N} follows.
\end{proof}

\begin{proof}[Proof of Lemma~\ref{lem:NLH shooting coercivity}]
	We first note that $2W(r)\leq\frac23+4r^{-2}$, so \eqref{eq:NLH unstable mode} and \eqref{eq:sobolev} yield $e_0\leq\frac23$. For radial $\varphi\in\dot H^1$ with $(Z,\varphi)=0$, set $\varphi_0\coloneqq \varphi-(\calY,\varphi)\calY$. Since the quadratic form of $H$ is nonnegative on $\calY^\perp$ and $1-e_0\geq\frac13$, \eqref{eq:NLH one bubble H1 coercivity} gives
	\begin{equation}\label{eq:NLH one bubble unit penalty}
		\begin{aligned}
			\|\varphi\|_{\dot H^1}^2
			&\lesssim(\varphi_0,H\varphi_0)+(\calY,\varphi)^2\\
			&\lesssim(\varphi_0,H\varphi_0)+(1-e_0)(\calY,\varphi)^2
			=(\varphi,H\varphi)+(\calY,\varphi)^2.
		\end{aligned}
	\end{equation}
    
	Fix $R=\alpha_c^{-1/8}>2R_0$ and choose a smooth partition $\sum_{j=0}^N\chi_j^2=1$ such that
	\begin{gather*}
		\chi_j=1\text{ on }[R^{-1}\lambda_j,R\lambda_j],\qquad \operatorname{supp}\chi_j\subset[R^{-2}\lambda_j,R^2\lambda_j],\qquad 1\leq j\leq N,
        \\
        {\textstyle\sum_{j=0}^N}|\partial_r\chi_j|^2\lesssim(\log R)^{-2}r^{-2} 
        \lesssim (\log \alpha_c)^{-2}r^{-2}.
	\end{gather*}
	By \eqref{eq:sobolev},
	\begin{equation}\label{eq:NLH localized H1 norm}
		{\textstyle\sum_{j=0}^N}\|\chi_jh\|_{\dot H^1}^2
        =(1+O(|\log\alpha_c|^{-2}))\|h\|_{\dot H^1}^2.
	\end{equation}
	Using \eqref{eq:P pointwise} and $|P-\iota_jQ_{[\lambda_j]}|\lesssim r^2\sum_{i\neq j}\rho_{\lambda_i}$, we obtain from $f'(z)=2|z|$ and the supports of $\chi_j$
	\begin{equation*}
		\sup_{\operatorname{supp}\chi_0}|f'(P)|+\max_{1\leq j\leq N}\sup_{\operatorname{supp}\chi_j}|f'(P)-f'(\iota_jQ_{[\lambda_j]})|
			\lesssim R^{-2}+R^4\alpha_c^2=o_{\alpha_c\to0}(1).
	\end{equation*}
	Hence, from \eqref{eq:sobolev}, we get
	\begin{equation}\label{eq:NLH localized potential error}
		{\textstyle\sum_{j=1}^N}|(\chi_jh,(H_U-H_{\lambda_j})\chi_jh)|+|(\chi_0h,r^{-2}f'(P)\chi_0h)|\leq o_{\alpha_c\to0}(1)\|h\|_{\dot H^1}^2.
	\end{equation}
	Since $R>2R_0$, we have $(Z_{\underline{\lambda_j}},\chi_jh)=(Z_{\underline{\lambda_j}},h)=0$. Smoothness of $\calY$ at zero and exponential decay at infinity give
	\begin{equation}\label{eq:NLH localized unstable projection}
		|(\iota_j\calY_{\underline{\lambda_j}},\chi_jh)-(\iota_j\calY_{\underline{\lambda_j}},h)|\lesssim o_{R\to \infty}(1)\|h\|_{\dot H^1}
        \lesssim o_{\alpha_c\to0}(1)\|h\|_{\dot H^1}.
	\end{equation}
	Applying the scaled form of \eqref{eq:NLH one bubble unit penalty} to $\chi_jh$ for $1\leq j\leq N$, summing over $j$, and using \eqref{eq:NLH localized H1 norm}--\eqref{eq:NLH localized unstable projection}, we obtain
	\begin{equation*}
		\|h\|_{\dot H^1}^2\lesssim (h,H_Uh)
        +{\textstyle\sum_{j=1}^N}(\calY_{\underline{\lambda_j}},h)^2  +o_{\alpha_c\to0}(1)\|h\|_{\dot H^1}^2.
	\end{equation*}
	Taking $\alpha_c$ sufficiently small, we conclude \eqref{eq:NLH shooting H1 coercivity}.
\end{proof}

\section{Exterior parabolic estimates}
In this appendix, we establish the properties of $e^{\sigma\Delta_2^{(1)}}k_a(r)$ used in the proof of Proposition~\ref{prop:lambda1 asym}. Recall $k_a(r)= r^{-1}\mathbf1_{\{r>a\}}.$
\begin{lem}[\cite{GustafsonNakanishiTsai2010CMP}]
	For all $a,r,\sigma>0$,
	\begin{equation}\label{eq:exterior weight pointwise}
		0\leq e^{\sigma\Delta_2^{(1)}}k_a(r)\lesssim\min\{a^{-1},r^{-1}, r\sigma^{-1}\}.
	\end{equation}
	Let $T\geq0$ and let $\lambda:[T,\infty)\to(0,\infty)$ be locally absolutely continuous with $\lambda_t\in L^2_t(T,\infty)$. Then, for $m\in\{1,2\}$ and $t\geq T$,
	\begin{equation}\label{eq:exterior weight double}
		\int_T^t\int_0^\infty|e^{(t-s)\Delta_2^{(1)}}k_{\lambda(t)}(r)|^2
		\langle r \lambda(s)^{-1}\rangle^{-4m}\frac{dr}{r}\,ds
		\lesssim1+\|\lambda_t\|_{L^2_s(T,t)}^2.
	\end{equation}
	For every fixed $\sigma>0$, $(e^{\sigma\Delta_2^{(1)}}-I)k_R\in L^2_2$ and
	\begin{equation}\label{eq:exterior remote endpoint}
		\|(e^{\sigma\Delta_2^{(1)}}-I)k_R\|_{L^2_2}\to0
		\qquad\text{as }R\to\infty.
	\end{equation}
	Moreover, if $f\in L^2_2$, $a>0$, and $b:(1,\infty)\to(0,\infty)$ satisfies $b(\sigma)/\sqrt\sigma\to0$, then $e^{\sigma\Delta_2^{(1)}}k_{b(\sigma)}-k_a\in L^2_2$ and
	\begin{equation}\label{eq:exterior endpoint asymptotic}
		(e^{\sigma\Delta_2^{(1)}}k_{b(\sigma)}-k_a,f)_2
		=-\int_a^{\sqrt\sigma}f(r)\,dr+o_{\sigma\to\infty}(1).
	\end{equation}
\end{lem}
\begin{proof}
	The upper bound in \eqref{eq:exterior weight pointwise} is the scaled form of \cite[Lemma~9.1]{GustafsonNakanishiTsai2010CMP}, and \eqref{eq:exterior weight double} follows from the calculation in \cite[(9.11)--(9.15)]{GustafsonNakanishiTsai2010CMP}.
	The argument for \eqref{eq:exterior remote endpoint} and \eqref{eq:exterior endpoint asymptotic} is based on \cite[(9.19)--(9.20)]{GustafsonNakanishiTsai2010CMP}. For completeness, we sketch the proof.

	Since $0\leq k_a\leq \min\{a^{-1},r^{-1}\}$ with $\Delta_2^{(1)}(r^{-1})=0$ and $\Delta_2^{(1)}(a^{-1})=-a^{-1}r^{-2}\leq0$, the maximum principle gives $0\leq e^{\sigma\Delta_2^{(1)}}k_a\leq \min\{a^{-1},r^{-1}\}$. Moreover, by the conjugation $\Delta_2^{(1)}F=r\Delta_{\bbR^4}(F/r)$,
	\begin{equation*}
		e^{\sigma\Delta_2^{(1)}}k_a(r)\leq r e^{\sigma\Delta_{\bbR^4}}(r^{-2})
		=r^{-1}(1-e^{-r^2/(4\sigma)})\lesssim r\sigma^{-1}.
	\end{equation*}

	For \eqref{eq:exterior weight double}, fix $m\in\{1,2\}$ and define $\sigma=t-s$. By \eqref{eq:exterior weight pointwise},
	\begin{equation*}
		\int_0^\infty\big|e^{\sigma\Delta_2^{(1)}}k_{\lambda(t)}(r)\big|^2
		\langle r \lambda(t-\sigma)^{-1}\rangle^{-4m}\frac{dr}{r}
		\lesssim
		\begin{cases}
			\lambda(t)^{-2}\big(1+\log\frac{\lambda(t)^2}{\sigma}\big),&0<\sigma<\lambda(t)^2,\\
			\lambda(t-\sigma)^2\sigma^{-2},&\lambda(t)^2 \leq \sigma.
		\end{cases}
	\end{equation*}
	The first bound is integrable in $\sigma$. If $t-T>\lambda(t)^2$, then
	\begin{equation*}
		\lambda(t-\sigma)=\lambda(t)-\sigma{\textstyle\int_0^1}\lambda_t(t-\theta\sigma)\,d\theta,
	\end{equation*}
	and Minkowski's inequality yields
	\begin{equation*}
		\|\sigma^{-1}\lambda(t-\sigma)\|_{L^2_\sigma(\lambda(t)^2,t-T)} \lesssim1+\|\lambda_t\|_{L^2_s(T,t)}.
	\end{equation*}
	Thus, integration in $\sigma$ proves \eqref{eq:exterior weight double}.

	Let $J_0$ and $J_1$ be the Bessel functions of the first kind of orders zero and one, respectively. Let $\calF_1$ denote the Fourier--Bessel transform of order one:
	\begin{equation*}
		(\calF_1F)(\rho)\coloneqq{\textstyle\int_0^\infty} J_1(r\rho)F(r)r\,dr.
	\end{equation*}
	It is unitary on $L^2_2$, diagonalizes $\Delta_2^{(1)}$, and satisfies
	\begin{equation*}
		\calF_1(k_c-k_R)(\rho) =\rho^{-1}(J_0(c\rho)-J_0(R\rho)), \qquad 0<c<R.
	\end{equation*}
	Applying this identity to $k_R-k_M$, letting $M\to\infty$ and using Plancherel's theorem, we obtain \eqref{eq:exterior remote endpoint}:
	\begin{equation*}
		\|(e^{\sigma\Delta_2^{(1)}}-I)k_R\|_{L^2_2}^2
		={\textstyle\int_0^\infty}|e^{-\sigma \xi^2/R^2}-1|^2J_0(\xi)^2\tfrac{d\xi}{\xi}\to0.
	\end{equation*}
	The standard bounds for $J_0$ at zero and infinity justify this limit by DCT and also show that $(e^{\sigma\Delta_2^{(1)}}-I)k_R\in L^2_2$.

	For fixed $a>0$, the multiplier $\rho^{-1}[e^{-\sigma\rho^2}J_0(a\rho)-J_0(\sqrt\sigma\rho)]$ is uniformly bounded in $L^2_2$ after the change of variables $\xi=\sqrt\sigma\rho$, and its $L^2_2$ norm on $\rho\geq\delta$ tends to zero for every $\delta>0$. Splitting $\calF_1f$ at $\rho=\delta$ and then letting $\delta\to0$, we obtain
	\begin{equation*}
		(e^{\sigma\Delta_2^{(1)}}k_a-k_a,f)_2 =-{\textstyle\int_a^{\sqrt\sigma}}f(r)\,dr+o_{\sigma\to\infty}(1).
	\end{equation*}
	Finally, Plancherel and $b(\sigma)/\sqrt\sigma\to0$ imply
	\begin{equation*}
		\|e^{\sigma\Delta_2^{(1)}}(k_{b(\sigma)}-k_a)\|_{L^2_2}^2
		=\int_0^\infty e^{-2\xi^2}
		\bigg|J_0\bigg(\frac{b(\sigma)\xi}{\sqrt\sigma}\bigg)
		-J_0\bigg(\frac{a\xi}{\sqrt\sigma}\bigg)\bigg|^2\frac{d\xi}{\xi}
		\to0 .
	\end{equation*}
	For all sufficiently large $\sigma$, the two arguments of $J_0$ are at most $\xi$. The estimate $|J_0(\beta \xi)-J_0(\alpha \xi)|\lesssim \xi^2$ for $0\leq\alpha,\beta\leq1$ gives an integrable bound $\xi^3$ for the integrand on $0<\xi\leq1$, while $Ce^{-2\xi^2}\xi^{-1}$ is an integrable bound on $\xi\geq1$. The DCT proves the limit and hence \eqref{eq:exterior endpoint asymptotic}.
\end{proof}

\bibliographystyle{abbrv}
\bibliography{reference}

@article{GustafsonNakanishiTsai2010CMP,
	author = {Gustafson, Stephen and Nakanishi, Kenji and Tsai, Tai-Peng},
	doi = {10.1007/s00220-010-1116-6},
	fjournal = {Communications in Mathematical Physics},
	issn = {0010-3616},
	journal = {Comm. Math. Phys.},
	mrclass = {58E20 (53C44 82D40)},
	mrnumber = {2725187},
	mrreviewer = {R\'{e}mi Carles},
	number = {1},
	pages = {205--242},
	title = {Asymptotic stability, concentration, and oscillation in harmonic map heat-flow, {L}andau-{L}ifshitz, and {S}chr\"{o}dinger maps on {$\mathbb{R}^2$}},
	url = {https://doi.org/10.1007/s00220-010-1116-6},
	volume = {300},
	year = {2010}}

@article{KimTKwon2024arxivSolResol,
      title={Soliton resolution for {C}alogero--{M}oser derivative nonlinear {S}chr\"{o}dinger equation}, 
      journal = {preprint arXiv:2408.12843, to appear in J. Eur. Math. Soc.},
      author={Kim, Taegyu and Kwon, Soonsik},
      year={2024}
}

@article{HerreroVelazquez1992,
  title={A blow up result for semilinear heat equations in the supercritical case},
  author={Herrero, Miguel A. and Vel\'{a}zquez, Juan J. L.},
  journal={preprint},
  year={1992}
}

@article {HerreroVelazquez1994CRASPSI,
    AUTHOR = {Herrero, Miguel A. and Vel\'{a}zquez, Juan J. L.},
     TITLE = {Explosion de solutions d'\'{e}quations paraboliques
              semilin\'{e}aires supercritiques},
   JOURNAL = {C. R. Acad. Sci. Paris S\'{e}r. I Math.},
  FJOURNAL = {Comptes Rendus de l'Acad\'{e}mie des Sciences. S\'{e}rie I.
              Math\'{e}matique},
    VOLUME = {319},
      YEAR = {1994},
    NUMBER = {2},
     PAGES = {141--145},
      ISSN = {0764-4442},
   MRCLASS = {35B40 (35K57)},
  MRNUMBER = {1288393},
MRREVIEWER = {Sergey\ A.\ Vakulenko},
}

@article{FilippasHerreroVelazquez2000,
	author = {Filippas, Stathis and Herrero, Miguel A. and Vel\'{a}zquez, Juan J. L.},
	doi = {10.1098/rspa.2000.0648},
	fjournal = {The Royal Society of London. Proceedings. Series A. Mathematical, Physical and Engineering Sciences},
	issn = {1364-5021},
	journal = {R. Soc. Lond. Proc. Ser. A Math. Phys. Eng. Sci.},
	mrclass = {35K55 (35B33 35B40 35C20)},
	mrnumber = {1843848},
	mrreviewer = {Juli\'{a}n Aguirre},
	number = {2004},
	pages = {2957--2982},
	title = {Fast blow-up mechanisms for sign-changing solutions of a semilinear parabolic equation with critical nonlinearity},
	url = {https://doi.org/10.1098/rspa.2000.0648},
	volume = {456},
	year = {2000}}

@article {Mizoguchi2007Math.Ann.,
    AUTHOR = {Mizoguchi, Noriko},
     TITLE = {Rate of type {II} blowup for a semilinear heat equation},
   JOURNAL = {Math. Ann.},
  FJOURNAL = {Mathematische Annalen},
    VOLUME = {339},
      YEAR = {2007},
    NUMBER = {4},
     PAGES = {839--877},
      ISSN = {0025-5831,1432-1807},
   MRCLASS = {35K55 (35B40 35K15)},
  MRNUMBER = {2341904},
MRREVIEWER = {Juli\'{a}n\ Aguirre},
       DOI = {10.1007/s00208-007-0133-z},
       URL = {https://doi.org/10.1007/s00208-007-0133-z},
}

@article {Mizoguchi2011TranAMS,
    AUTHOR = {Mizoguchi, Noriko},
     TITLE = {Blow-up rate of type {II} and the braid group theory},
   JOURNAL = {Trans. Amer. Math. Soc.},
  FJOURNAL = {Transactions of the American Mathematical Society},
    VOLUME = {363},
      YEAR = {2011},
    NUMBER = {3},
     PAGES = {1419--1443},
      ISSN = {0002-9947,1088-6850},
   MRCLASS = {35K91 (35B07 35B44 35K20 57M07)},
  MRNUMBER = {2737271},
MRREVIEWER = {Christian\ Stinner},
       DOI = {10.1090/S0002-9947-2010-04784-1},
       URL = {https://doi.org/10.1090/S0002-9947-2010-04784-1},
}

@article{DengSunWei2025Duke,
author = {Deng, Bin and Sun, Liming and Wei, Jun-cheng},
title = {{Sharp quantitative estimates of Struwe’s decomposition}},
volume = {174},
journal = {Duke Mathematical Journal},
number = {1},
publisher = {Duke University Press},
pages = {159 -- 228},
year = {2025},
doi = {10.1215/00127094-2024-0026},
URL = {https://doi.org/10.1215/00127094-2024-0026}
}

@article{KimMerle2025CPAM,
author = {Kim, Kihyun and Merle, Frank},
title = {On classification of global dynamics for energy-critical equivariant harmonic map heat flows and radial nonlinear heat equation},
journal = {Communications on Pure and Applied Mathematics},
volume = {78},
number = {9},
pages = {1783--1842},
doi = {https://doi.org/10.1002/cpa.22253},
url = {https://onlinelibrary.wiley.com/doi/abs/10.1002/cpa.22253},
eprint = {https://onlinelibrary.wiley.com/doi/pdf/10.1002/cpa.22253},
year = {2025}
}

@article{RaphaelRodnianski2012,
	author = {Rapha\"{e}l, Pierre and Rodnianski, Igor},
	doi = {10.1007/s10240-011-0037-z},
	fjournal = {Publications Math\'{e}matiques. Institut de Hautes \'{E}tudes Scientifiques},
	issn = {0073-8301,1618-1913},
	journal = {Publ. Math. Inst. Hautes \'{E}tudes Sci.},
	mrclass = {58E20 (35A20 35B44 35L70 58E15 81T13)},
	mrnumber = {2929728},
	mrreviewer = {Andreas\ Gastel},
	pages = {1--122},
	title = {Stable blow up dynamics for the critical co-rotational wave maps and equivariant {Y}ang-{M}ills problems},
	url = {https://doi.org/10.1007/s10240-011-0037-z},
	volume = {115},
	year = {2012}}

@article{RaphaelSchweyer2013CPAMHeat,
	author = {Rapha\"{e}l, Pierre and Schweyer, Remi},
	doi = {10.1002/cpa.21435},
	fjournal = {Communications on Pure and Applied Mathematics},
	issn = {0010-3640,1097-0312},
	journal = {Comm. Pure Appl. Math.},
	mrclass = {35R01 (35B44 35K91 53Cxx 58J35)},
	mrnumber = {3008229},
	mrreviewer = {Shu-Yu\ Hsu},
	number = {3},
	pages = {414--480},
	title = {Stable blowup dynamics for the 1-corotational energy critical harmonic heat flow},
	url = {https://doi.org/10.1002/cpa.21435},
	volume = {66},
	year = {2013}}

@article{RaphaelSchweyer2014AnalPDEHeatQuantized,
	author = {Rapha\"{e}l, Pierre and Schweyer, Remi},
	doi = {10.2140/apde.2014.7.1713},
	fjournal = {Analysis \& PDE},
	issn = {2157-5045,1948-206X},
	journal = {Anal. PDE},
	mrclass = {35K91 (35B44)},
	mrnumber = {3318739},
	mrreviewer = {Christopher\ P.\ Grant},
	number = {8},
	pages = {1713--1805},
	title = {Quantized slow blow-up dynamics for the corotational energy-critical harmonic heat flow},
	url = {https://doi.org/10.2140/apde.2014.7.1713},
	volume = {7},
	year = {2014}}

@article{delPinoMussoWei2020Infheat,
	author = {del Pino, Manuel and Musso, Monica and Wei, Juncheng},
	doi = {10.2140/apde.2020.13.215},
	fjournal = {Analysis \& PDE},
	issn = {2157-5045,1948-206X},
	journal = {Anal. PDE},
	mrclass = {35K91 (35B33 35B40 35B44 35K15 35K58)},
	mrnumber = {4047646},
	number = {1},
	pages = {215--274},
	title = {Infinite-time blow-up for the 3-dimensional energy-critical heat equation},
	url = {https://doi.org/10.2140/apde.2020.13.215},
	volume = {13},
	year = {2020}}

@article{delPinoMussoWeiZhang2020HeatQuantized,
	author = {del Pino, Manuel and Musso, Monica and Wei, Juncheng and Zhang, Qidi and Zhou, Yifu},
	journal = {preprint arXiv:2002.05765},
	title = {Type {II} Finite time blow-up for the three dimensional energy critical heat equation},
	year = {2020}}

@article {MerleZaag2005Math.Ann.,
    AUTHOR = {Merle, Frank and Zaag, Hatem},
     TITLE = {Determination of the blow-up rate for a critical semilinear
              wave equation},
   JOURNAL = {Math. Ann.},
  FJOURNAL = {Mathematische Annalen},
    VOLUME = {331},
      YEAR = {2005},
    NUMBER = {2},
     PAGES = {395--416},
      ISSN = {0025-5831,1432-1807},
   MRCLASS = {35L70 (35B40)},
  MRNUMBER = {2115461},
MRREVIEWER = {Masahito\ Ohta},
       DOI = {10.1007/s00208-004-0587-1},
       URL = {https://doi.org/10.1007/s00208-004-0587-1},
}

@article {MerleZaag2003AJM,
    AUTHOR = {Merle, Frank and Zaag, Hatem},
     TITLE = {Determination of the blow-up rate for the semilinear wave
              equation},
   JOURNAL = {Amer. J. Math.},
  FJOURNAL = {American Journal of Mathematics},
    VOLUME = {125},
      YEAR = {2003},
    NUMBER = {5},
     PAGES = {1147--1164},
      ISSN = {0002-9327,1080-6377},
   MRCLASS = {35L70 (35B40 35L15)},
  MRNUMBER = {2004432},
MRREVIEWER = {Masahito\ Ohta},
       URL ={http://muse.jhu.edu/journals/american_journal_of_mathematics/v125/125.5merle.pdf},
}

@article {Harada2020AIHPC,
    AUTHOR = {Harada, Junichi},
     TITLE = {A higher speed type {II} blowup for the five dimensional
              energy critical heat equation},
   JOURNAL = {Ann. Inst. H. Poincar\'e{} C Anal. Non Lin\'eaire},
  FJOURNAL = {Annales de l'Institut Henri Poincar\'e{} C. Analyse Non
              Lin\'eaire},
    VOLUME = {37},
      YEAR = {2020},
    NUMBER = {2},
     PAGES = {309--341},
      ISSN = {0294-1449,1873-1430},
   MRCLASS = {35K91 (35B44)},
  MRNUMBER = {4072807},
MRREVIEWER = {Daniele\ Andreucci},
       DOI = {10.1016/j.anihpc.2019.09.006},
       URL = {https://doi.org/10.1016/j.anihpc.2019.09.006},
}

@article {MatanoMerle2004CPAM,
    AUTHOR = {Matano, Hiroshi and Merle, Frank},
     TITLE = {On nonexistence of type {II} blowup for a supercritical
              nonlinear heat equation},
   JOURNAL = {Comm. Pure Appl. Math.},
  FJOURNAL = {Communications on Pure and Applied Mathematics},
    VOLUME = {57},
      YEAR = {2004},
    NUMBER = {11},
     PAGES = {1494--1541},
      ISSN = {0010-3640,1097-0312},
   MRCLASS = {35K55 (35B40)},
  MRNUMBER = {2077706},
MRREVIEWER = {Juli\'{a}n\ Aguirre},
       DOI = {10.1002/cpa.20044},
       URL = {https://doi.org/10.1002/cpa.20044},
}

@article{Kim2025JEMS,
	author = {Kim, Kihyun},
	journal = {J. Eur. Math. Soc.},
	title = {Rigidity of smooth finite-time blow-up for equivariant self-dual {C}hern-{S}imons-{S}chr\"{o}dinger equation},
    volume = {28},
    number = {12},
    pages = {5469--5572},
  url = {https://doi.org/10.4171/jems/1569},
	year = {2026}}

@article {KimKwonOh2025AJM,
    AUTHOR = {Kim, Kihyun and Kwon, Soonsik and Oh, Sung-Jin},
     TITLE = {Soliton resolution for equivariant self-dual
              {C}hern-{S}imons-{S}chr\"odinger equation in weighted
              {S}obolev class},
   JOURNAL = {Amer. J. Math.},
  FJOURNAL = {American Journal of Mathematics},
    VOLUME = {147},
      YEAR = {2025},
    NUMBER = {6},
     PAGES = {1475--1508},
      ISSN = {0002-9327,1080-6377},
   MRCLASS = {35},
  MRNUMBER = {4995127},
}

@article {MerleZaag2012AJM,
    AUTHOR = {Merle, Frank and Zaag, Hatem},
     TITLE = {Existence and classification of characteristic points at
              blow-up for a semilinear wave equation in one space dimension},
   JOURNAL = {Amer. J. Math.},
  FJOURNAL = {American Journal of Mathematics},
    VOLUME = {134},
      YEAR = {2012},
    NUMBER = {3},
     PAGES = {581--648},
      ISSN = {0002-9327,1080-6377},
   MRCLASS = {35L71 (35B44 35L15)},
  MRNUMBER = {2931219},
MRREVIEWER = {Mohammad\ A.\ Rammaha},
       DOI = {10.1353/ajm.2012.0021},
       URL = {https://doi.org/10.1353/ajm.2012.0021},
}

@article {CoteZaag2013CPAM,
    AUTHOR = {C\^ote, Rapha\"el and Zaag, Hatem},
     TITLE = {Construction of a multisoliton blowup solution to the
              semilinear wave equation in one space dimension},
   JOURNAL = {Comm. Pure Appl. Math.},
  FJOURNAL = {Communications on Pure and Applied Mathematics},
    VOLUME = {66},
      YEAR = {2013},
    NUMBER = {10},
     PAGES = {1541--1581},
      ISSN = {0010-3640,1097-0312},
   MRCLASS = {35L71 (35B44 35L15)},
  MRNUMBER = {3084698},
       DOI = {10.1002/cpa.21452},
       URL = {https://doi.org/10.1002/cpa.21452},
}

@article{MartelMerleRaphael2014Acta,
	author = {Martel, Yvan and Merle, Frank and Rapha\"{e}l, Pierre},
	doi = {10.1007/s11511-014-0109-2},
	fjournal = {Acta Mathematica},
	issn = {0001-5962,1871-2509},
	journal = {Acta Math.},
	mrclass = {35Q53 (35B33 35B44 35C08)},
	mrnumber = {3179608},
	mrreviewer = {Masayoshi\ Tsutsumi},
	number = {1},
	pages = {59--140},
	title = {Blow up for the critical generalized {K}orteweg--de {V}ries equation. {I}: {D}ynamics near the soliton},
	url = {https://doi.org/10.1007/s11511-014-0109-2},
	volume = {212},
	year = {2014}}

@article {KimKimKwon2024arxiv,
	author={Kim, Kihyun and Kim, Taegyu and Kwon, Soonsik},
	journal = {preprint arXiv:2404.09603, to appear in Mem. Amer. Math. Soc.},
	title={Construction of smooth chiral finite-time blow-up solutions to {C}alogero--{M}oser derivative nonlinear {S}chr\"{o}dinger equation}, 
	year={2024}
}

@article {DuyckaertsKenigMerle2023Acta,
	AUTHOR = {Duyckaerts, Thomas and Kenig, Carlos and Merle, Frank},
	TITLE = {Soliton resolution for the radial critical wave equation in
	all odd space dimensions},
	JOURNAL = {Acta Math.},
	FJOURNAL = {Acta Mathematica},
	VOLUME = {230},
	YEAR = {2023},
	NUMBER = {1},
	PAGES = {1--92},
	ISSN = {0001-5962,1871-2509},
	MRCLASS = {35Q35 (35R11)},
	MRNUMBER = {4567713},
	DOI = {10.4310/acta.2023.v230.n1.a1},
	URL = {https://doi.org/10.4310/acta.2023.v230.n1.a1},
}

@article {Raphael2005MathAnnalen,
	AUTHOR = {Raphael, Pierre},
	TITLE = {Stability of the log-log bound for blow up solutions to the
	critical non linear {S}chr\"odinger equation},
	JOURNAL = {Math. Ann.},
	FJOURNAL = {Mathematische Annalen},
	VOLUME = {331},
	YEAR = {2005},
	NUMBER = {3},
	PAGES = {577--609},
	ISSN = {0025-5831,1432-1807},
	MRCLASS = {35Q55 (35B35 35B40)},
	MRNUMBER = {2122541},
	MRREVIEWER = {Masahito\ Ohta},
	DOI = {10.1007/s00208-004-0596-0},
	URL = {https://doi.org/10.1007/s00208-004-0596-0},
}

@article {MerleRaphael2005CMP,
	AUTHOR = {Merle, Frank and Raphael, Pierre},
	TITLE = {Profiles and quantization of the blow up mass for critical
	nonlinear {S}chr\"odinger equation},
	JOURNAL = {Comm. Math. Phys.},
	FJOURNAL = {Communications in Mathematical Physics},
	VOLUME = {253},
	YEAR = {2005},
	NUMBER = {3},
	PAGES = {675--704},
	ISSN = {0010-3616,1432-0916},
	MRCLASS = {35Q55 (35B40)},
	MRNUMBER = {2116733},
	MRREVIEWER = {Justin\ A.\ Holmer},
	DOI = {10.1007/s00220-004-1198-0},
	URL = {https://doi.org/10.1007/s00220-004-1198-0},
}

@article {DuyckaertsKenigMerle2013Camb,
	AUTHOR = {Duyckaerts, Thomas and Kenig, Carlos and Merle, Frank},
	TITLE = {Classification of radial solutions of the focusing,
	energy-critical wave equation},
	JOURNAL = {Camb. J. Math.},
	FJOURNAL = {Cambridge Journal of Mathematics},
	VOLUME = {1},
	YEAR = {2013},
	NUMBER = {1},
	PAGES = {75--144},
	ISSN = {2168-0930,2168-0949},
	MRCLASS = {35L71 (35B07 35B30 35B40 35B44 35C08)},
	MRNUMBER = {3272053},
	MRREVIEWER = {Bruno\ Scheurer},
	DOI = {10.4310/CJM.2013.v1.n1.a3},
	URL = {https://doi.org/10.4310/CJM.2013.v1.n1.a3},
}

@article {JendrejLawrie2025JAMS,
    AUTHOR = {Jendrej, Jacek and Lawrie, Andrew},
     TITLE = {Soliton resolution for energy-critical wave maps in the
              equivariant case},
   JOURNAL = {J. Amer. Math. Soc.},
  FJOURNAL = {Journal of the American Mathematical Society},
    VOLUME = {38},
      YEAR = {2025},
    NUMBER = {3},
     PAGES = {783--875},
      ISSN = {0894-0347,1088-6834},
   MRCLASS = {35L71 (35B40 35C08 37K40)},
  MRNUMBER = {4890658},
       DOI = {10.1090/jams/1012},
       URL = {https://doi.org/10.1090/jams/1012},
}

@article {Cote2015CPAMsolitonResol,
	AUTHOR = {C\^ote, R.},
	TITLE = {On the soliton resolution for equivariant wave maps to the
	sphere},
	JOURNAL = {Comm. Pure Appl. Math.},
	FJOURNAL = {Communications on Pure and Applied Mathematics},
	VOLUME = {68},
	YEAR = {2015},
	NUMBER = {11},
	PAGES = {1946--2004},
	ISSN = {0010-3640,1097-0312},
	MRCLASS = {35L52 (35C08 58E20)},
	MRNUMBER = {3403756},
	MRREVIEWER = {Bruno\ Scheurer},
	DOI = {10.1002/cpa.21545},
	URL = {https://doi.org/10.1002/cpa.21545},
}

@article {JiaKenig2017AJMwaveSolResol,
	AUTHOR = {Jia, Hao and Kenig, Carlos},
	TITLE = {Asymptotic decomposition for semilinear wave and equivariant
	wave map equations},
	JOURNAL = {Amer. J. Math.},
	FJOURNAL = {American Journal of Mathematics},
	VOLUME = {139},
	YEAR = {2017},
	NUMBER = {6},
	PAGES = {1521--1603},
	ISSN = {0002-9327,1080-6377},
	MRCLASS = {35L71 (35B40 35B44 35C08 35L15)},
	MRNUMBER = {3730929},
	MRREVIEWER = {Jean-Marc\ Delort},
	DOI = {10.1353/ajm.2017.0039},
	URL = {https://doi.org/10.1353/ajm.2017.0039},
}

@article {DJKM2017GaFASolresolSequence,
	AUTHOR = {Duyckaerts, Thomas and Jia, Hao and Kenig, Carlos and Merle,
	Frank},
	TITLE = {Soliton resolution along a sequence of times for the focusing
	energy critical wave equation},
	JOURNAL = {Geom. Funct. Anal.},
	FJOURNAL = {Geometric and Functional Analysis},
	VOLUME = {27},
	YEAR = {2017},
	NUMBER = {4},
	PAGES = {798--862},
	ISSN = {1016-443X,1420-8970},
	MRCLASS = {35L71 (35B33 35C08 35L15)},
	MRNUMBER = {3678502},
	MRREVIEWER = {Joseph\ L.\ Shomberg},
	DOI = {10.1007/s00039-017-0418-7},
	URL = {https://doi.org/10.1007/s00039-017-0418-7},
}

@article {JL2023CVPDE,
	AUTHOR = {Jendrej, Jacek and Lawrie, Andrew},
	TITLE = {Bubble decomposition for the harmonic map heat flow in the
	equivariant case},
	JOURNAL = {Calc. Var. Partial Differential Equations},
	FJOURNAL = {Calculus of Variations and Partial Differential Equations},
	VOLUME = {62},
	YEAR = {2023},
	NUMBER = {9},
	PAGES = {Paper No. 264, 36},
	ISSN = {0944-2669,1432-0835},
	MRCLASS = {35L71 (35B40 37K40)},
	MRNUMBER = {4662424},
	DOI = {10.1007/s00526-023-02597-1},
	URL = {https://doi.org/10.1007/s00526-023-02597-1},
}

@article {JLS2025Pi,
    AUTHOR = {Jendrej, Jacek and Lawrie, Andrew and Schlag, Wilhelm},
     TITLE = {Continuous in time bubble decomposition for the harmonic map
              heat flow},
   JOURNAL = {Forum Math. Pi},
  FJOURNAL = {Forum of Mathematics. Pi},
    VOLUME = {13},
      YEAR = {2025},
     PAGES = {Paper No. e4, 37},
      ISSN = {2050-5086},
   MRCLASS = {58E20 (35K58 58J35)},
  MRNUMBER = {4862091},
MRREVIEWER = {Junichi\ Harada},
       DOI = {10.1017/fmp.2024.15},
       URL = {https://doi.org/10.1017/fmp.2024.15},
}

@article{Aryan2024solResolHeat,
	author={Aryan, Shrey},
	journal = {preprint arXiv:2405.06005, to appear in Analysis \& PDE},
	title={Soliton resolution for the energy-critical nonlinear heat equation in the radial case}, 
	year={2024}
}

@article {JendrejLawrie2023AnnPDESolResol,
	AUTHOR = {Jendrej, Jacek and Lawrie, Andrew},
	TITLE = {Soliton resolution for the energy-critical nonlinear wave
	equation in the radial case},
	JOURNAL = {Ann. PDE},
	FJOURNAL = {Annals of PDE. Journal Dedicated to the Analysis of Problems
	from Physical Sciences},
	VOLUME = {9},
	YEAR = {2023},
	NUMBER = {2},
	PAGES = {Paper No. 18, 117},
	ISSN = {2524-5317,2199-2576},
	MRCLASS = {35L71 (35B40 37K40)},
	MRNUMBER = {4650926},
	DOI = {10.1007/s40818-023-00159-4},
	URL = {https://doi.org/10.1007/s40818-023-00159-4},
}

@article {delPinoMussoWei2021AnalPDE,
	AUTHOR = {del Pino, Manuel and Musso, Monica and Wei, Juncheng},
	TITLE = {Existence and stability of infinite time bubble towers in the
	energy critical heat equation},
	JOURNAL = {Anal. PDE},
	FJOURNAL = {Analysis \& PDE},
	VOLUME = {14},
	YEAR = {2021},
	NUMBER = {5},
	PAGES = {1557--1598},
	ISSN = {2157-5045,1948-206X},
	MRCLASS = {35K58 (35B44)},
	MRNUMBER = {4307216},
	MRREVIEWER = {Junichi\ Harada},
	DOI = {10.2140/apde.2021.14.1557},
	URL = {https://doi.org/10.2140/apde.2021.14.1557},
}

@article {Hout2003JDE,
	AUTHOR = {van der Hout, Rein},
	TITLE = {On the nonexistence of finite time bubble trees in symmetric
	harmonic map heat flows from the disk to the 2-sphere},
	JOURNAL = {J. Differential Equations},
	FJOURNAL = {Journal of Differential Equations},
	VOLUME = {192},
	YEAR = {2003},
	NUMBER = {1},
	PAGES = {188--201},
	ISSN = {0022-0396,1090-2732},
	MRCLASS = {53C44 (35K50 35K55 58E20)},
	MRNUMBER = {1987090},
	MRREVIEWER = {Roger\ Moser},
	DOI = {10.1016/S0022-0396(03)00043-3},
	URL = {https://doi.org/10.1016/S0022-0396(03)00043-3},
}

@book {Topping2000,
	AUTHOR = {Topping, Peter},
	TITLE = {An example of a nontrivial bubble tree in the harmonic map
	heat flow},
	BOOKTITLE = {Harmonic morphisms, harmonic maps, and related topics
	({B}rest, 1997)},
	SERIES = {Chapman \& Hall/CRC Res. Notes Math.},
	VOLUME = {413},
	PAGES = {185--191},
	PUBLISHER = {Chapman \& Hall/CRC, Boca Raton, FL},
	YEAR = {2000},
	ISBN = {1-584880-32-5},
	MRCLASS = {58E20 (53C44)},
	MRNUMBER = {1735698},
	MRREVIEWER = {Andreas\ Gastel},
}

@article {Jendrej2017AnalPDE,
	AUTHOR = {Jendrej, Jacek},
	TITLE = {Construction of two-bubble solutions for the energy-critical
	{NLS}},
	JOURNAL = {Anal. PDE},
	FJOURNAL = {Analysis \& PDE},
	VOLUME = {10},
	YEAR = {2017},
	NUMBER = {8},
	PAGES = {1923--1959},
	ISSN = {2157-5045,1948-206X},
	MRCLASS = {35Q55 (35B40)},
	MRNUMBER = {3694010},
	MRREVIEWER = {Yuichiro\ Kawahara},
	DOI = {10.2140/apde.2017.10.1923},
	URL = {https://doi.org/10.2140/apde.2017.10.1923},
}

@article {Jendrej2019AJM,
	AUTHOR = {Jendrej, Jacek},
	TITLE = {Construction of two-bubble solutions for energy-critical wave
	equations},
	JOURNAL = {Amer. J. Math.},
	FJOURNAL = {American Journal of Mathematics},
	VOLUME = {141},
	YEAR = {2019},
	NUMBER = {1},
	PAGES = {55--118},
	ISSN = {0002-9327,1080-6377},
	MRCLASS = {35L71 (35L15)},
	MRNUMBER = {3904767},
	MRREVIEWER = {Mihai\ Tohaneanu},
	DOI = {10.1353/ajm.2019.0002},
	URL = {https://doi.org/10.1353/ajm.2019.0002},
}

@article {JendrejLawrie2018Invent,
	AUTHOR = {Jendrej, Jacek and Lawrie, Andrew},
	TITLE = {Two-bubble dynamics for threshold solutions to the wave maps
	equation},
	JOURNAL = {Invent. Math.},
	FJOURNAL = {Inventiones Mathematicae},
	VOLUME = {213},
	YEAR = {2018},
	NUMBER = {3},
	PAGES = {1249--1325},
	ISSN = {0020-9910,1432-1297},
	MRCLASS = {58E20 (35L51 35L70)},
	MRNUMBER = {3842064},
	MRREVIEWER = {Leszek\ Gasi\'nski},
	DOI = {10.1007/s00222-018-0804-2},
	URL = {https://doi.org/10.1007/s00222-018-0804-2},
}

@article {MerleRaphael2005AnnMath,
	AUTHOR = {Merle, Frank and Raphael, Pierre},
	TITLE = {The blow-up dynamic and upper bound on the blow-up rate for
	critical nonlinear {S}chr\"odinger equation},
	JOURNAL = {Ann. of Math. (2)},
	FJOURNAL = {Annals of Mathematics. Second Series},
	VOLUME = {161},
	YEAR = {2005},
	NUMBER = {1},
	PAGES = {157--222},
	ISSN = {0003-486X,1939-8980},
	MRCLASS = {35Q55 (35B40)},
	MRNUMBER = {2150386},
	MRREVIEWER = {John\ Albert},
	DOI = {10.4007/annals.2005.161.157},
	URL = {https://doi.org/10.4007/annals.2005.161.157},
}

@article {MerleRaphael2003GAFA,
	AUTHOR = {Merle, F. and Raphael, P.},
	TITLE = {Sharp upper bound on the blow-up rate for the critical
	nonlinear {S}chr\"odinger equation},
	JOURNAL = {Geom. Funct. Anal.},
	FJOURNAL = {Geometric and Functional Analysis},
	VOLUME = {13},
	YEAR = {2003},
	NUMBER = {3},
	PAGES = {591--642},
	ISSN = {1016-443X,1420-8970},
	MRCLASS = {35Q55 (35B40)},
	MRNUMBER = {1995801},
	DOI = {10.1007/s00039-003-0424-9},
	URL = {https://doi.org/10.1007/s00039-003-0424-9},
}

@article {MerleRaphael2004Invent,
	AUTHOR = {Merle, Frank and Raphael, Pierre},
	TITLE = {On universality of blow-up profile for {$L^2$} critical
	nonlinear {S}chr\"odinger equation},
	JOURNAL = {Invent. Math.},
	FJOURNAL = {Inventiones Mathematicae},
	VOLUME = {156},
	YEAR = {2004},
	NUMBER = {3},
	PAGES = {565--672},
	ISSN = {0020-9910,1432-1297},
	MRCLASS = {35Q55 (35B33 35B40 35B65 35Q51)},
	MRNUMBER = {2061329},
	MRREVIEWER = {Masahito\ Ohta},
	DOI = {10.1007/s00222-003-0346-z},
	URL = {https://doi.org/10.1007/s00222-003-0346-z},
}

@article {MerleRaphael2006JAMS,
	AUTHOR = {Merle, Frank and Raphael, Pierre},
	TITLE = {On a sharp lower bound on the blow-up rate for the {$L^2$}
	critical nonlinear {S}chr\"odinger equation},
	JOURNAL = {J. Amer. Math. Soc.},
	FJOURNAL = {Journal of the American Mathematical Society},
	VOLUME = {19},
	YEAR = {2006},
	NUMBER = {1},
	PAGES = {37--90},
	ISSN = {0894-0347,1088-6834},
	MRCLASS = {35Q55 (35B40)},
	MRNUMBER = {2169042},
	MRREVIEWER = {Masahito\ Ohta},
	DOI = {10.1090/S0894-0347-05-00499-6},
	URL = {https://doi.org/10.1090/S0894-0347-05-00499-6},
}

@article{WangWei2021arXiv,
	author = {Wang, Kelei and Wei, Juncheng},
	journal = {preprint arXiv:2101.07186},
	title = {Refined blowup analysis and nonexistence of type {II} blowups for an energy critical nonlinear heat equation},
	year = {2021}
}

@article{SunWeiZhang2022CVPDE,
  author  = {Sun, Liming and Wei, Jun-cheng and Zhang, Qidi},
  title   = {Bubble towers in the ancient solution of energy-critical heat equation},
  journal = {Calc. Var. Partial Differential Equations},
  volume  = {61},
  year    = {2022},
  pages   = {Paper No. 200},
  doi     = {10.1007/s00526-022-02296-3}
}

@article{JendrejKrieger2025arXiv,
	author = {Jendrej, Jacek and Krieger, Joachim},
	journal = {preprint arXiv:2501.08396},
	title = {Concentric bubbles concentrating in finite time for the energy critical wave maps equation},
	year = {2025}
}

@article{JeongKimKimKwon2026arXiv,
	author = {Jeong, Uihyeon and Kim, Kihyun and Kim, Taegyu and Kwon, Soonsik},
	journal = {preprint arXiv:2601.07410},
	title = {Classification of single-bubble blow-up solutions for {C}alogero--{M}oser derivative nonlinear {S}chr\"odinger equation},
	year = {2026}
}

@article{KimMerle2026arXiv,
	title={Rigidity results in multi-bubble dynamics for non-radial energy-critical heat equation},
	author={Kim, Kihyun and Merle, Frank},
	journal = {preprint arXiv:2601.12517},
	year={2026}
}

@article{Kim2026arXivNobubbletree,
    title={No bubble trees for the $1$-equivariant harmonic map heat flow and the radial energy-critical nonlinear heat equation in low dimensions}, 
    author={Kim, Taegyu},
    journal = {preprint arXiv:2608.22447},
    year={2026}
}

@article{JendrejLiXu2026arXiv,
	author = {Jendrej, Jacek and Li, Xuemei and Xu, Guixiang},
	journal = {preprint arXiv:2602.08490},
	title = {Construction of two-bubble solutions for the energy-critical {H}artree equation},
	year = {2026}
}

@article{KriegerPalacios2026arXiv,
	author = {Krieger, Joachim and Palacios, Jos\'{e} M.},
	journal = {preprint arXiv:2602.22825},
	title = {Long finite time bubble trees for two co-rotational wave maps},
	year = {2026}
}

@article{HwangKim2026arXiv,
	author = {Hwang, Seunghwan and Kim, Kihyun},
	journal = {preprint arXiv:2603.01793},
	title = {Construction of infinite time bubble tower solutions to critical wave maps equation},
	year = {2026}
}

@article{Shen2026arXiv,
	author = {Shen, Ruipeng},
	journal = {preprint arXiv:2603.21602},
	title = {Nonexistence of multi-bubble radial solutions to the {3D} energy critical wave equation},
	year = {2026}
}

@article{Samuelian2026arXiv,
	author = {Samuelian, Dylan},
	journal = {preprint arXiv:2606.24555},
	title = {Nonexistence of finite-time blow-up for the equivariant harmonic map heat flow from {$B^2$} to {$S^2$}},
	year = {2026}
}

@article{JendrejLiXu2026arXivNLS,
	author = {Jendrej, Jacek and Li, Xuemei and Xu, Guixiang},
	journal = {preprint arXiv:2608.16186},
	title = {Construction of two-bubble solutions for the energy-critical {NLS} in dimension 6},
	year = {2026}
}

@article{EellsSampson1964AJM,
	author = {Eells, Jr., James and Sampson, Joseph H.},
	journal = {Amer. J. Math.},
	number = {1},
	pages = {109--160},
	title = {Harmonic mappings of {R}iemannian manifolds},
	volume = {86},
	year = {1964}}

@article{Struwe1985,
	author = {Struwe, Michael},
	journal = {Comment. Math. Helv.},
	number = {4},
	pages = {558--581},
	title = {On the evolution of harmonic mappings of {R}iemannian surfaces},
	volume = {60},
	year = {1985}}

@article {SacksUhlenbeck1981,
    AUTHOR = {Sacks, Jonathan and Uhlenbeck, Karen},
     TITLE = {The existence of minimal immersions of {$2$}-spheres},
   JOURNAL = {Ann. of Math. (2)},
  FJOURNAL = {Annals of Mathematics. Second Series},
    VOLUME = {113},
      YEAR = {1981},
    NUMBER = {1},
     PAGES = {1--24},
      ISSN = {0003-486X},
   MRCLASS = {58E12 (53C42 58E20)},
  MRNUMBER = {604040},
MRREVIEWER = {John\ C.\ Wood},
       DOI = {10.2307/1971131},
       URL = {https://doi.org/10.2307/1971131},
}

@article{Topping1997JDG,
	author = {Topping, Peter M.},
	journal = {J. Differential Geom.},
	number = {3},
	pages = {593--610},
	title = {Rigidity in the harmonic map heat flow},
	volume = {45},
	year = {1997}}

@article{Topping2004AnnMath,
	author = {Topping, Peter},
	journal = {Ann. of Math. (2)},
	number = {2},
	pages = {465--534},
	title = {Repulsion and quantization in almost-harmonic maps, and asymptotics of the harmonic map flow},
	volume = {159},
	year = {2004}}

@article{Topping2004MathZ,
	author = {Topping, Peter},
	journal = {Math. Z.},
	number = {2},
	pages = {279--302},
	title = {Winding behaviour of finite-time singularities of the harmonic map heat flow},
	volume = {247},
	year = {2004}}

@article{Weissler1980,
	author = {Weissler, Fred B.},
	journal = {Indiana Univ. Math. J.},
	number = {1},
	pages = {79--102},
	title = {Local existence and nonexistence for semilinear parabolic equations in {$L^p$}},
	volume = {29},
	year = {1980}}

@article{BrezisCazenave1996,
	author = {Brezis, Ha{\"\i}m and Cazenave, Thierry},
	journal = {J. Anal. Math.},
	pages = {277--304},
	title = {A nonlinear heat equation with singular initial data},
	volume = {68},
	year = {1996}}

@article{CollotMerleRaphael2017CMP,
	author = {Collot, Charles and Merle, Frank and Rapha{\"e}l, Pierre},
	journal = {Comm. Math. Phys.},
	number = {1},
	pages = {215--285},
	title = {Dynamics near the ground state for the energy critical nonlinear heat equation in large dimensions},
	volume = {352},
	year = {2017}}

@article {WZZ2026JFA,
    AUTHOR = {Wei, Juncheng and Zhang, Qidi and Zhou, Yifu},
     TITLE = {Trichotomy dynamics of the 1-equivariant harmonic map flow},
   JOURNAL = {J. Funct. Anal.},
  FJOURNAL = {Journal of Functional Analysis},
    VOLUME = {290},
      YEAR = {2026},
    NUMBER = {4},
     PAGES = {Paper No. 111248, 95},
      ISSN = {0022-1236,1096-0783},
   MRCLASS = {35K10 (35B40)},
  MRNUMBER = {4987373},
MRREVIEWER = {Codru\c ta\ S.\ Stoica},
       DOI = {10.1016/j.jfa.2025.111248},
       URL = {https://doi.org/10.1016/j.jfa.2025.111248},
}

@article {Schweyer2012JFA,
    AUTHOR = {Schweyer, R\'emi},
     TITLE = {Type {II} blow-up for the four dimensional energy critical
              semi linear heat equation},
   JOURNAL = {J. Funct. Anal.},
  FJOURNAL = {Journal of Functional Analysis},
    VOLUME = {263},
      YEAR = {2012},
    NUMBER = {12},
     PAGES = {3922--3983},
      ISSN = {0022-1236,1096-0783},
   MRCLASS = {35K91 (35A20 35B44 35B45 35K15)},
  MRNUMBER = {2990063},
MRREVIEWER = {Juli\'an\ Aguirre},
       DOI = {10.1016/j.jfa.2012.09.015},
       URL = {https://doi.org/10.1016/j.jfa.2012.09.015},
}

@article {PMW2019,
    AUTHOR = {del Pino, Manuel and Musso, Monica and Wei, Jun Cheng},
     TITLE = {Type {II} blow-up in the 5-dimensional energy critical heat
              equation},
   JOURNAL = {Acta Math. Sin. (Engl. Ser.)},
  FJOURNAL = {Acta Mathematica Sinica (English Series)},
    VOLUME = {35},
      YEAR = {2019},
    NUMBER = {6},
     PAGES = {1027--1042},
      ISSN = {1439-8516,1439-7617},
   MRCLASS = {35K91 (35B40 35B44 35K15)},
  MRNUMBER = {3952701},
MRREVIEWER = {Hongwei\ Chen},
       DOI = {10.1007/s10114-019-8341-5},
       URL = {https://doi.org/10.1007/s10114-019-8341-5},
}

@article {WZZ2024JDE,
    AUTHOR = {Wei, Juncheng and Zhang, Qidi and Zhou, Yifu},
     TITLE = {On {F}ila-{K}ing conjecture in dimension four},
   JOURNAL = {J. Differential Equations},
  FJOURNAL = {Journal of Differential Equations},
    VOLUME = {398},
      YEAR = {2024},
     PAGES = {38--140},
      ISSN = {0022-0396,1090-2732},
   MRCLASS = {35K91 (35B40 35K15)},
  MRNUMBER = {4725137},
       DOI = {10.1016/j.jde.2024.03.004},
       URL = {https://doi.org/10.1016/j.jde.2024.03.004},
}

@article {LWZZ2024,
    AUTHOR = {Li, Zaizheng and Wei, Juncheng and Zhang, Qidi and Zhou, Yifu},
     TITLE = {Long-time dynamics for the energy critical heat equation in
              {$\mathbb{R}^5$}},
   JOURNAL = {Nonlinear Anal.},
  FJOURNAL = {Nonlinear Analysis. Theory, Methods \& Applications. An
              International Multidisciplinary Journal},
    VOLUME = {247},
      YEAR = {2024},
     PAGES = {Paper No. 113594, 15},
      ISSN = {0362-546X,1873-5215},
   MRCLASS = {35B40 (35K15 35K55)},
  MRNUMBER = {4764473},
       DOI = {10.1016/j.na.2024.113594},
       URL = {https://doi.org/10.1016/j.na.2024.113594},
}

@article {DPS2018crelle,
    AUTHOR = {Daskalopoulos, Panagiota and del Pino, Manuel and Sesum,
              Natasa},
     TITLE = {Type {II} ancient compact solutions to the {Y}amabe flow},
   JOURNAL = {J. Reine Angew. Math.},
  FJOURNAL = {Journal f\"ur die Reine und Angewandte Mathematik. [Crelle's
              Journal]},
    VOLUME = {738},
      YEAR = {2018},
     PAGES = {1--71},
      ISSN = {0075-4102,1435-5345},
   MRCLASS = {53C44 (35K55 53A30 58J35)},
  MRNUMBER = {3794888},
MRREVIEWER = {Shouwen\ Fang},
       DOI = {10.1515/crelle-2015-0048},
       URL = {https://doi.org/10.1515/crelle-2015-0048},
}

@article {Samuelian2026CVPDE,
    AUTHOR = {Samuelian, Dylan},
     TITLE = {On blow-up trees for the harmonic map heat flow from {$B^2$}
              to {$S^2$}},
   JOURNAL = {Calc. Var. Partial Differential Equations},
  FJOURNAL = {Calculus of Variations and Partial Differential Equations},
      YEAR = {2026},
    NUMBER = {8},
    volume  = {65},
  pages = {Paper No. 238},
      ISSN = {0944-2669,1432-0835},
   MRCLASS = {35K55 (35B40)},
  MRNUMBER = {5106288},
       DOI = {10.1007/s00526-026-03396-0},
       URL = {https://doi.org/10.1007/s00526-026-03396-0},
}

@article {Mizoguchi2022CPAM,
    AUTHOR = {Mizoguchi, Noriko},
     TITLE = {Refined asymptotic behavior of blowup solutions to a
              simplified chemotaxis system},
   JOURNAL = {Comm. Pure Appl. Math.},
  FJOURNAL = {Communications on Pure and Applied Mathematics},
    VOLUME = {75},
      YEAR = {2022},
    NUMBER = {8},
     PAGES = {1870--1886},
      ISSN = {0010-3640,1097-0312},
   MRCLASS = {35Q92 (92C17)},
  MRNUMBER = {4465904},
MRREVIEWER = {Shanbing\ Li},
       DOI = {10.1002/cpa.21954},
       URL = {https://doi.org/10.1002/cpa.21954},
}

@article {GigaKohn1985CPAM,
    AUTHOR = {Giga, Yoshikazu and Kohn, Robert V.},
     TITLE = {Asymptotically self-similar blow-up of semilinear heat
              equations},
   JOURNAL = {Comm. Pure Appl. Math.},
  FJOURNAL = {Communications on Pure and Applied Mathematics},
    VOLUME = {38},
      YEAR = {1985},
    NUMBER = {3},
     PAGES = {297--319},
      ISSN = {0010-3640,1097-0312},
   MRCLASS = {35K55 (35B40)},
  MRNUMBER = {784476},
MRREVIEWER = {Mitsuhiro\ Nakao},
       DOI = {10.1002/cpa.3160380304},
       URL = {https://doi.org/10.1002/cpa.3160380304},
}

@article {GMS2004,
    AUTHOR = {Giga, Yoshikazu and Matsui, Shin'ya and Sasayama, Satoshi},
     TITLE = {Blow up rate for semilinear heat equations with subcritical
              nonlinearity},
   JOURNAL = {Indiana Univ. Math. J.},
  FJOURNAL = {Indiana University Mathematics Journal},
    VOLUME = {53},
      YEAR = {2004},
    NUMBER = {2},
     PAGES = {483--514},
      ISSN = {0022-2518,1943-5258},
   MRCLASS = {35K55 (35B40 35K15)},
  MRNUMBER = {2060042},
MRREVIEWER = {Julio\ D.\ Rossi},
       DOI = {10.1512/iumj.2004.53.2401},
       URL = {https://doi.org/10.1512/iumj.2004.53.2401},
}

@article {Harada2020annPDE,
    AUTHOR = {Harada, Junichi},
     TITLE = {A type {II} blowup for the six dimensional energy critical
              heat equation},
   JOURNAL = {Ann. PDE},
  FJOURNAL = {Annals of PDE. Journal Dedicated to the Analysis of Problems
              from Physical Sciences},
    VOLUME = {6},
      YEAR = {2020},
    NUMBER = {2},
     PAGES = {Paper No. 13, 63},
      ISSN = {2524-5317,2199-2576},
   MRCLASS = {35K91 (35B44 35K15)},
  MRNUMBER = {4151635},
       DOI = {10.1007/s40818-020-00088-6},
       URL = {https://doi.org/10.1007/s40818-020-00088-6},
}

@article {Harada2026CVPDE,
    AUTHOR = {Harada, Junichi},
     TITLE = {Dynamics near the ground state for the {S}obolev critical
              {F}ujita equation in 6{D}},
   JOURNAL = {Calc. Var. Partial Differential Equations},
  FJOURNAL = {Calculus of Variations and Partial Differential Equations},
    VOLUME = {65},
      YEAR = {2026},
    NUMBER = {6},
     PAGES = {Paper No. 179, 71},
      ISSN = {0944-2669,1432-0835},
   MRCLASS = {35B44 (35B35 35B40 35K58)},
  MRNUMBER = {5065893},
       DOI = {10.1007/s00526-026-03348-8},
       URL = {https://doi.org/10.1007/s00526-026-03348-8},
}

@article {CDY1992JDG,
    AUTHOR = {Chang, Kung-Ching and Ding, Wei Yue and Ye, Rugang},
     TITLE = {Finite-time blow-up of the heat flow of harmonic maps from
              surfaces},
   JOURNAL = {J. Differential Geom.},
  FJOURNAL = {Journal of Differential Geometry},
    VOLUME = {36},
      YEAR = {1992},
    NUMBER = {2},
     PAGES = {507--515},
      ISSN = {0022-040X,1945-743X},
   MRCLASS = {58E20 (35K55 58G11)},
  MRNUMBER = {1180392},
MRREVIEWER = {Martin\ Fuchs},
       URL = {http://projecteuclid.org/euclid.jdg/1214448751},
}

@article {Qing1995,
    AUTHOR = {Qing, Jie},
     TITLE = {On singularities of the heat flow for harmonic maps from
              surfaces into spheres},
   JOURNAL = {Comm. Anal. Geom.},
  FJOURNAL = {Communications in Analysis and Geometry},
    VOLUME = {3},
      YEAR = {1995},
    NUMBER = {1-2},
     PAGES = {297--315},
      ISSN = {1019-8385,1944-9992},
   MRCLASS = {58G11 (35K55 58E20)},
  MRNUMBER = {1362654},
MRREVIEWER = {Joseph\ F.\ Grotowski},
       DOI = {10.4310/CAG.1995.v3.n2.a4},
       URL = {https://doi.org/10.4310/CAG.1995.v3.n2.a4},
}

@article {DingTian1995,
    AUTHOR = {Ding, Weiyue and Tian, Gang},
     TITLE = {Energy identity for a class of approximate harmonic maps from
              surfaces},
   JOURNAL = {Comm. Anal. Geom.},
  FJOURNAL = {Communications in Analysis and Geometry},
    VOLUME = {3},
      YEAR = {1995},
    NUMBER = {3-4},
     PAGES = {543--554},
      ISSN = {1019-8385,1944-9992},
   MRCLASS = {58E20 (58G11)},
  MRNUMBER = {1371209},
MRREVIEWER = {Joseph\ F.\ Grotowski},
       DOI = {10.4310/CAG.1995.v3.n4.a1},
       URL = {https://doi.org/10.4310/CAG.1995.v3.n4.a1},
}

@article {Wang1996,
    AUTHOR = {Wang, Changyou},
     TITLE = {Bubble phenomena of certain {P}alais-{S}male sequences from
              surfaces to general targets},
   JOURNAL = {Houston J. Math.},
  FJOURNAL = {Houston Journal of Mathematics},
    VOLUME = {22},
      YEAR = {1996},
    NUMBER = {3},
     PAGES = {559--590},
      ISSN = {0362-1588},
   MRCLASS = {58E20 (49Q10)},
  MRNUMBER = {1417632},
MRREVIEWER = {Martin\ Fuchs},
}

@article {QingTian1997CPAM,
    AUTHOR = {Qing, Jie and Tian, Gang},
     TITLE = {Bubbling of the heat flows for harmonic maps from surfaces},
   JOURNAL = {Comm. Pure Appl. Math.},
  FJOURNAL = {Communications on Pure and Applied Mathematics},
    VOLUME = {50},
      YEAR = {1997},
    NUMBER = {4},
     PAGES = {295--310},
      ISSN = {0010-3640,1097-0312},
   MRCLASS = {58E20 (58G11)},
  MRNUMBER = {1438148},
MRREVIEWER = {Daniel\ Pollack},
       DOI = {10.1002/(SICI)1097-0312(199704)50:4<295::AID-CPA1>3.0.CO;2-5},
       URL =
              {https://doi.org/10.1002/(SICI)1097-0312(199704)50:4<295::AID-CPA1>3.0.CO;2-5},
}

@article {LinWang1998CVPDE,
    AUTHOR = {Lin, Fanghua and Wang, Changyou},
     TITLE = {Energy identity of harmonic map flows from surfaces at finite
              singular time},
   JOURNAL = {Calc. Var. Partial Differential Equations},
  FJOURNAL = {Calculus of Variations and Partial Differential Equations},
    VOLUME = {6},
      YEAR = {1998},
    NUMBER = {4},
     PAGES = {369--380},
      ISSN = {0944-2669,1432-0835},
   MRCLASS = {58E20 (35K55 58G11)},
  MRNUMBER = {1624304},
MRREVIEWER = {Joseph\ F.\ Grotowski},
       DOI = {10.1007/s005260050095},
       URL = {https://doi.org/10.1007/s005260050095},
}

\end{document}